\documentclass[12pt]{amsart}
\usepackage{latexsym}
\usepackage{amssymb}
\usepackage{amscd}
\usepackage{mathrsfs}
\usepackage{CJK}
\usepackage[all]{xy}
\usepackage{tikz-cd}

\renewcommand\a{\alpha}
\renewcommand\b{\beta}
\newcommand\g{\gamma}
\renewcommand\d{\delta}
\newcommand\la{\lambda}
\newcommand\z{\zeta}

\renewcommand\th{\theta}

\newcommand\s{\sigma}
\newcommand\x{\chi}
\newcommand\f{\phi}
\newcommand\vf{\varphi}

\renewcommand\t{\tau}
\renewcommand\r{\rho}
\newcommand\Om{\Omega}
\newcommand\w{\omega}

\newcommand\vS{\varSigma}

\newcommand\vD{\varDelta}

\newcommand\vL{\varLambda}

\newcommand\ve{\varepsilon}

\newcommand\Ql{\bar{\mathbf Q}_l}
\newcommand\BA{\mathbf A}
\newcommand\BP{\mathbf P}
\newcommand\BQ{\mathbf Q}
\newcommand\BF{\mathbf F}

\newcommand\BR{\mathbf R}

\newcommand\BZ{\mathbf Z}

\newcommand\BG{\mathbf G}

\newcommand\Bk{\mathbf k}

\newcommand\BSigma{\boldsymbol\Sigma}

\newcommand\CQ{\mathcal{Q}}

\newcommand\SB{\mathscr{B}}

\newcommand\SE{\mathscr{E}}
\newcommand\SF{\mathscr{F}}
\newcommand\SG{\mathscr{G}}

\newcommand\SM{\mathscr{M}}
\newcommand\SN{\mathscr{N}} 

\newcommand\SP{\mathscr{P}}

\newcommand\SH{\mathscr{H}}

\newcommand\SW{\mathscr{W}}

\newcommand\Fc{\mathfrak c}
\newcommand\Fg{\mathfrak g}

\newcommand\Fn{\mathfrak n}

\newcommand\iv{^{-1}}
\newcommand\wh{\widehat}
\newcommand\wt{\widetilde}
\newcommand\wg{^{\wedge}}

\newcommand\ol{\overline}

\newcommand\hra{\hookrightarrow}
\newcommand\lra{\leftrightarrow}

\newcommand\IC{\operatorname{IC}}
\newcommand\Ker{\operatorname{Ker}}
\newcommand\Hom{\operatorname{Hom}}
\newcommand\End{\operatorname{End}}

\newcommand\Ind{\operatorname{Ind}}

\newcommand\supp{\operatorname{supp}\,}
\newcommand\Lie{\operatorname{Lie}}
\newcommand\Tr{\operatorname{Tr}\,}
\newcommand\ch{\operatorname{ch}}

\newcommand\Ad{\operatorname{Ad}}
\newcommand\ad{\operatorname{ad}}
\newcommand\ex{_{\operatorname{ex}}}

\newcommand\uni{_{\operatorname{uni}}}

\newcommand\lp{\operatorname{\!\langle\!}}
\newcommand\rp{\operatorname{\!\rangle\!}}
\renewcommand\Im{\operatorname{Im}}

\newcommand\odd{\operatorname{odd}}
\newcommand\Spin{\operatorname{Spin}}

\newcommand\da{\dot a}

\newcommand\dc{\dot c}

\newcommand{\isom}{\,\raise2pt\hbox{$\underrightarrow{\sim}$}\,}
\numberwithin{equation}{section}

\newtheorem{thm}{Theorem}[section]
\newtheorem{lem}[thm]{Lemma}

\newtheorem{prop}[thm]{Proposition}

\def \para#1{\par\medskip\textbf{#1}
              \addtocounter{thm}{1}}

\def \remark#1{\par\medskip\noindent
                \textbf{Remark #1}
                \addtocounter{thm}{1}}

\def \remarks#1{\par\medskip\noindent
                \textbf{Remarks #1}
                \addtocounter{thm}{1}}

\begin{document}
\setlength{\baselineskip}{4.9mm}
\setlength{\abovedisplayskip}{4.5mm}
\setlength{\belowdisplayskip}{4.5mm}
\renewcommand{\theenumi}{\roman{enumi}}
\renewcommand{\labelenumi}{(\theenumi)}
\renewcommand{\thefootnote}{\fnsymbol{footnote}}
\renewcommand{\thefootnote}{\fnsymbol{footnote}}
\parindent=20pt
\medskip
\begin{center}
  {\bf Generalized Green functions and unipotent classes  \\
    for finite reductive groups, IV }
\par
\vspace{1cm}
Frank L\"ubeck and Toshiaki Shoji
\\
\vspace{0.7cm}
\title{}
\end{center}

\begin{abstract}
In this paper, we formulate the notion of split elements of
a unipotent class in a connected reductive group $G$.
Generalized Green functions of $G$ can
be computed by using Lusztig's algorithm, if split elements
exist for any unipotent class. The existence of split elements
is reduced to the case where $G$ is a simply connected, almost simple group.
We show, in the case of classical groups, that there exist split elements,
which is a refinement of previous results. In the case of exceptional groups,
we show the existence of split elements, possibly except one class 
for $G$ of type $E_7$. 
\end{abstract}

\maketitle
\pagestyle{myheadings}
\markboth{L\"UBECK AND SHOJI}{GENERALIZED GREEN FUNCTIONS}

\bigskip
\medskip

\par
\medskip\noindent
\addtocounter{section}{-1}
\section{Introduction} 
This paper is a continuation of a series of papers 
\cite[I, II, III]{Sh4}.  
Let $G$ be a connected reductive group defined over a finite field 
$\BF_q$ of characteristic $p$, with Frobenius map $F : G \to G$, 
and $G\uni$ the unipotent variety of $G$. 
The aim of these papers is to determine the generalized Green functions 
of $G$. This problem is reduced, by Lusztig's algorithm (\cite[IV]{L3}),
to the computation of certain functions $Y_j$ on the space of
$G^F$-invariant functions on $G\uni^F$.  Since $Y_j$ is written as 
$Y_j = \g_j Y_j^0$ for a certain scalar $\g_j$, where $Y_j^0$ is a function 
explicitly computable, the above problem is reduced to the determination 
of those scalars $\g_j$. On the other hand, the determination of generalized Green functions 
is reduced to the case where $G$ is a simply connected, almost simple group.    
In \cite[II,III]{Sh4}, $\g_j$ was determined in the case of classical groups, 
by an inductive formula. 
In this paper, we give a closed formula for $\g_j$ in the case of classical groups. 
Also we determine $\g_j$ in the case of exceptional groups (except
one class in $G = E_7$). 
Thus the generalized Green functions are computable for all the cases, modulo
the above exception.  
\par
Here we give some historical backgrounds.  The Green functions $Q^G_T$ of $G^F$, where
$T$ is an $F$-stable maximal torus of $G$, were 
introduced by Deligne-Lusztig \cite{DL} in 1976, 
in connection with the representation theory of finite reductive groups $G^F$ over $\Ql$.  
Springer \cite{Spr} defined certain functions, by making use of the Springer 
representations of Weyl groups (also constructed in that paper), and conjectured 
that those functions coincide with Deligne-Lusztig's Green functions.  This conjecture
was proved by Kazhdan \cite{K} under the assumption that $p$ and $q$ are large enough. 
The algorithm of computing Springer's Green functions was given by
\cite{Sh2} for exceptional groups of type $F_4$, and by \cite{Sh-cl} for classical groups.
Based on this algorithm, Green functions for exceptional groups were computed, 
by \cite{Sh2} for type $F_4$, and by Beynon-Spaltenstein \cite{BS} for types $E_6, E_7$ and $E_8$, 
under the assumption that $p$ is a good prime (no restriction on $q$). The case $G_2$ was 
already computed by Springer \cite{Spr}.    
\par
The generalized Green functions $Q_{L, C_0,\SE_0, \vf_0}^G$ were introduced  
by Lusztig \cite{L3} in the study of character sheaves (see 1.2 for the definition), 
based on the theory of generalized Springer correspondence (\cite{L2}), where 
$L$ is an $F$-stable Levi subgroup of a parabolic subgroup of $G$. 
The special case where $L = T$ is a maximal torus is nothing but Springer's Green functions,
which corresponds to the case of Springer correspondence (see (1.1.4)).
The coincidence of those Green functions with Deligne-Lusztig's Green functions 
(and its generalization to the case of generalized Green functions) 
was proved in \cite{L4} in 1990 for a large enough $q$ (for any $p$). 
This restriction on $q$ was removed in \cite{Sh3} in the case where 
$G$ has a connected center, by applying the theory of Shintani descent. 
\par
\par
In 2020, Green functions for exceptional groups in bad characteristic 
were computed by Geck \cite{G2}, except for type $E_8$. He reduced in \cite{G1}
the computation for the case of $\BF_q$ to the case of the prime field $\BF_p$,
and computed the Green functions in that case by using CHEVIE,
by making use of the representation theory of $G^F$ (note that the computation
of Green functions is reduced to the case of adjoint groups, and by the previous results
those Green functions coincide with Deligne-Lusztig's Green functions).
Recently, by extending the method of Geck, one of the authors computed in \cite{Lu}
the Green functions of type $E_8$ in bad characteristic. 
Hence all the Green functions were computed.
However, as far as the generalized Green functions are
concerned, no results in the case of exceptional groups are known.  
\par
\par
As mentioned above, the theory of generalized Green functions is based on the 
theory of generalized Springer correspondence.  In the case of exceptional groups, 
the generalized Springer correspondence was determined by Spaltenstein \cite{Sp5} in 1985, 
except for a few cases. The remaining cases are one unipotent class in $G$, of type $E_6$ 
with $p \ne 3$, and of type $E_8$ with $p = 3$. In 2019, Lusztig \cite{L5} solved 
this problem in the case of $E_6$, by using a geometric method. 
Finally in 2023, Hetz \cite{H} solved it in the case of $E_8$, by computing values
of certain unipotent characters of $G^F$ at unipotent elements by the aide of CHEVIE
(here Deligne-Lusztig's Green functions and their generalization were used).
Hence the generalized Springer correspondence is now determined completely.
\par
In this paper, we compute the generalized Green functions for exceptional groups.
The main ingredient for the proof is the Frobenius version of the restriction 
formula, as in the case of \cite[II, III]{Sh4}.
Also the notion of split elements plays a crucial role.  
This notion already appeared in the previous papers, in ad hoc way.
In this paper, we give a notion of split elements, in a uniform way, which
is applicable for reductive groups in general.
If split elements exist for any unipotent class in $G$, then
$\g_j$ is determined explicitly, thus generalized Green functions are computable.
In  the case where $F$ is non-split,
our result gives a more precise description for $\g_j$. In particular,  in the case where $G$
is a special linear group with $F$ non-split, we obtain a closed formula
for $\g_j$ which improves the results in \cite[I, II]{Sh4}.  
In the case of other classical groups, in particular, of spin groups,
\cite[III]{Sh4} contains some errors.  In this paper we have corrected them. 
In the case of exceptional groups, we show that split elements exist in
any unipotent class $C$ in $G$, possibly except one class in $E_7$.
This determines $\g_j$ except for a few cases. 
\par
In the computation of generalized Green functions for exceptional groups,
we have made use of the results on the computation of Green functions.

\par\bigskip\noindent
{\bf Contents}
\par\medskip\noindent
0. \ Introduction
\\
1. \ Split elements in reductive groups
\\
2. \ Special linear groups
\\
3. \ Classical groups
\\
4. \ Preliminaries on exceptional groups
\\
5. \ The case $E_7$
\\
6. \ The case $E_8$
\\
7. \ The case $F_4$
\\
8. \ The case $G_2$
\\
9. \ The case $E_6$ 
\\
10. \ The case ${}^3D_4$
\par\bigskip
\section{Split elements in reductive groups}

\para{1.1}
This paper is a continuation of \cite [I,II,III]{Sh4}.
First we review the definition of generalized Springer correspondence
and generalized Green functions. 
Let $G$ be a connected reductive group defined over a finite field $\BF_q$ 
with Frobenius map $F$.  Let $\Bk$ be an algebraic closure of $\BF_q$ with 
$\ch \Bk = p$.  
Let $\SM_G$ be the set of triples $(L, C_0, \SE_0)$, up to $G$-conjugacy, 
where $L$ is a Levi subgroup of a parabolic subgroup $P$ of $G$, and 
$\SE_0$ is a cuspidal local system on a unipotent class $C_0$ of $L$.  
As in (1.2.2) in \cite[I]{Sh4}, one can define a semisimple perverse sheaf $K$ on $G$ 
associated to the 
triple $(L, C_0, \SE_0) \in \SM_G$.  Let $\SW_L = N_G(L)/L$. Then 
$\SW_L$ is a Coxeter group, and $\End K \simeq \Ql[\SW_L]$. Hence $K$ is decomposed as
\begin{equation*}
\tag{1.1.1}
K \simeq \bigoplus_{E \in \SW_L\wg}E\otimes K_E,
\end{equation*} 
where $K_E \simeq \Hom (E, K)$ is a simple perverse sheaf on $G$ 
associated to $E \in \SW_L\wg$. 
\par
Let $G\uni$ be the unipotent variety of $G$, and let $\SN_G$ be the set of $(C, \SE)$, 
where $C$ is a unipotent class of $G$, and $\SE$ is a $G$-equivariant
simple local system on $C$. 
Let $Z_L$ be the center of $L$, and put $d = \dim Z_L$. It is known that 
$K[-d]|_{G\uni}$ is a semisimple perverse sheaf on $G\uni$, equipped with $\SW_L$-action, 
and is decomposed as 
\begin{equation*}
\tag{1.1.2}
K[-d]|_{G\uni} \simeq \bigoplus_{(C,\SE) \in \SN_G}V_{(C,\SE)}\otimes \IC(\ol C, \SE)[\dim C],
\end{equation*}
where the multiplicity space $V_{(C,\SE)}$ has a structure of
an irreducible $\SW_L$-module if it is non-zero. 
It follows from (1.1.1) and (1.1.2) that $K_E|_{G\uni}$ is expressed as 
\begin{equation*}
\tag{1.1.3}
K_E|_{G\uni} \simeq \IC(\ol C, \SE)[\dim C + \dim Z_L]
\end{equation*}
for some $(C,\SE) \in \SN_G$, and we have a map $\SW_L\wg \to \SN_G, E \mapsto (C,\SE)$
such that $E \simeq V_{(C,\SE)}$.  
The generalized Springer correspondence (\cite[Thm. 6.5]{L2}) 
asserts that this map gives a bijection 
\begin{equation*}
\tag{1.1.4}
\bigsqcup_{(L,C_0, \SE_0) \in \SM_G}\SW_L\wg  \isom \SN_G.
\end{equation*}
In the case where $L = T$ is a maximal torus, then 
$(T, \{ 1\}, 1) \in \SM_G$, where $\{1 \}$ is the trivial class in $T$, and 
$1 = \Ql$ is the constant sheaf on $\{ 1\}$.  $\SW_L$ corresponding to
$(T,\{1\}, 1)$ is the Weyl group $W = N_G(T)/T$ of $G$,
and the injection $W\wg \to \SN_G$ gives the Springer correspondence. 

\par
If we choose $u \in C$, then $\SE$ corresponds bijectively to $\r \in A_G(u)\wg$, 
in such a way that the stalk $\SE_u$ has a natural structure of $A_G(u)$-module 
isomorphic to $\r$. (In \cite[III]{Sh4}, $\r^*$ was used instead of $\r$.
This is just a problem of the parametrization.) In this way, the pair 
$(C,\SE)$ is in bijection with the $G$-orbit of the pair $(u,\r)$.
We write it as $(u,\r) \lra (C,\SE) \in \SN_G$, and  
denote the representation as $V_{(u,\r)} = V_{(C,\SE)}$. 

\para{1.2.}
Let $Z_G$ be the center of $G$.  The generalized
Springer correspondence (1.1.4) has a partition in terms of
$Z_G\wg$ as follows (see \cite[14.2]{L2}, \cite[I, 1.9]{Sh4}).
For $\x \in Z_G\wg$, let $\SN_{\x}$ be the subset of $\SN_{\x}$
consisting of $(u,\r) \lra (C,\SE)$ such that $Z_G$ acts on
$\SE$ according to $\x$.  Let $\SM_{\x}$ be the subset of
$\SM_G$ consisting of $(L,C_0, \SE_0)$ such that $Z_G$ acts on
$\SE_0$ according to $\x$.  Then we have partitions
$\SN_G = \bigsqcup_{\x \in Z_G\wg}\SN_{\x}$, $\SM_G = \bigsqcup_{\x \in Z_G\wg}\SM_{\x}$,
and (1.1.4) restricts to a bijection, for each $\x \in Z_G\wg$,
\begin{equation*}
\tag{1.2.1}  
\bigsqcup_{(L, C_0, \SE_0) \in \SM_{\x}}\SW_L\wg  \isom \SN_{\x}.
\end{equation*}  

\para{1.3.}
$F$ acts naturally on the set $\SN_G$ and $\SM_G$ by 
$(C, \SE) \mapsto (F\iv(C), F^*\SE)$, 
$(L, C_0, \SE_0) \mapsto (F\iv(L), F\iv(C_0), F^*\SE_0)$. 
The map in (1.1.4) is compatible with this $F$-action. 
We choose $(L, C_0, \SE_0) \in \SM_G^F$. Here we may assume that 
$L$ is an $F$-stable Levi subgroup of a (not necessarily $F$-stable)
parabolic subgroup $P$ of $G$,
and that $F(C_0) = C_0, F^*\SE_0 \simeq \SE_0$.  
Since $\SE_0$ is a simple local system, the isomorphism $\vf_0: F^*\SE_0 \isom \SE_0$
is unique up to scalar.
We choose $\vf_0$ by the condition that 
the induced map on the stalk of $\SE_0$ at any point in $C_0^F$ is of finite order.  
$\vf_0$ induces a natural isomorphism $\vf : F^*K \isom K$. 
The generalized Green function $Q^G_{L,C_0, \SE_0, \vf_0}$ is defined as
the restriction 
of the characteristic function $\x_{K,\vf}$ on $G^F$ to $G^F\uni$ (cf. \cite[II]{L3}), 
which is a $G^F$-invariant function on $G^F\uni$ with values in $\Ql$.
In the case where $(L, C_0, \SE_0) = (T, \{1\}, 1) \in \SM_G$, $Q^G_{L,C_0,\SE_0, \vf_0}$
(for the canonical isomorphism $\vf_0 : F^*\Ql \isom \Ql$)
is called the Green function of $G^F$, and is denoted by $Q^G_T$.

\para{1.4.}
Let $B$ be an $F$-stable Borel subgroup of $G$, and $T$ an $F$-stable maximal torus
of $G$ contained in $B$. 
We assume that $P$ is an $F$-stable parabolic subgroup of $G$ containing $B$, and 
$L$ an $F$-stable Levi subgroup of $P$ containing $T$.
$F$ acts naturally on $\SW_L$, which induces a Coxeter group automorphism 
$\s$ of order $c$. We consider the semidirect product $\wt\SW_L = \lp\s\rp\ltimes \SW_L$.   
Let $(\SW_L)\wg\ex$ be the set of (isomorphism classes of) 
irreducible representations of $\SW_L$ which are extendable to $\wt\SW_L$. 
For each $E \in (\SW_L)\wg\ex$, we choose a specific extension $\wt E$ to $\wt\SW_L$,
called the preferred extension of $E$ (for the definition, see 1.6), 
and let $\s_E : E \isom E$ be the action of $\s$ on $\wt E$.
By using the decomposition (1.1.1), one sees, for each $E \in (\SW_L)\wg\ex$, that
there exists a unique isomorphism $\vf_E : F^*K_E \isom K_E$ such that 
$\s_E\otimes \vf_E : F^*(E\otimes K_E) \isom E\otimes K_E$ gives the restriction of 
$\vf : F^*K \isom K$.   
By making use of the isomorphism (1.1.3), one can define an isomorphism 
$\psi : F^*\SE \isom \SE$ induced from $\vf_E : F^*K_E \isom K_E$ 
by the condition that $q^{(a_0+r)/2}\psi$ corresponds to the map 
$F^*(\SH^{a_0}K_E) \isom \SH^{a_0}K_E$ induced from $\vf_E$, where 
\begin{align*}
\tag{1.4.1}
a_0 &= -\dim Z_L - \dim C, \\
r &= \dim G - \dim L + \dim (C_0 \times Z_L). 
\end{align*}
Note that $\SH^{a_0}K_E|_C = \SE$ and $r = \dim \supp K_E$.  Also note the relation 
\begin{equation*}
\tag{1.4.2}
a_0 + r = (\dim G - \dim C) - (\dim L - \dim C_0).
\end{equation*}

The generalized Green function $Q^G_{L', C'_0, \SE_0', \vf'_0}$ defined in 1.3
(for an $F$-stable Levi subgroup $L'$ of a not necessarily $F$-stable parabolic
subgroup $P'$ of $G$) is obtained from this $L$ as
$Q^G_{L^w, C_0^w, \SE_0^w, \vf_0^w} = \x_{K, w\circ \vf}$, where $(L^w, C_0^w, \SE_0^w, \vf_0^w)$
are systems obtained from $(L, C_0, \SE_0, \vf_0)$ by twisting $w \in \SW_L$
(see \cite[II, 10.6]{L3}).  In the case of Green functions, it is denoted by
$Q^G_{T_w}$, where $T_w$ is an $F$-stable maximal torus twisted from $T$ by $w \in W$. 
\par
We denote by $I$ the set $\SN_G$. For $i = (C, \SE) \in I^F$ belonging to
$(L,C_0,\SE_0) \in \SM_G^F$, let $E_i \in (\SW_L)\wg\ex$ be the representation
defined by $E_i = V_{(C,\SE)}$. We define a function $X_i$ on $G^F\uni$ by

\begin{equation*}
\tag{1.4.3}  
X_i(g) = \sum_a(-1)^{a + a_0}\Tr(\vf_{E_i}, \SH^a_gK_{E_i})q^{-(a_0+r)/2}
\end{equation*}
for $g \in G^F\uni$, where $\vf_{E_i} : F^*K_{E_i} \isom K_{E_i}$ is
as in 1.4, and $a_0 + r$ is defined with respect to $C$ for $i = (C,\SE)$. 
Then $Q^G_{L^w,C_0^w, \SE_0^w, \vf_0^w}(g)$ can be written as

\begin{equation*}
\tag{1.4.4}  
  Q^G_{L^w, C_0^w, \SE_0^w, \vf_0^w}(g)
         = \sum_{i \in I^F}\Tr(w\s_{E_i}, \wt E_i)X_i(g)(-1)^{(a_0)_i}q^{(a_0+r)_i/2},
\end{equation*}
where $(a_0)_i, (a_0 + r)_i$ denote $a_0, a_0 + r$ with respect to $i = (C, \SE)$. 

\par
For each $i \in I^F$, we define a function $Y_i$ on $G^F\uni$ by

\begin{equation*}
\tag{1.4.5}
  Y_i(g) = \begin{cases}
              \Tr(\psi, \SE_g)  &\quad\text{ if } g \in C^F, \\
              0                 &\quad\text{ otherwise.} 
           \end{cases}
\end{equation*}
Then $\{ Y_i \mid i \in I^F\}$ gives rise to a basis of the $\Ql$-space
of $G^F$-invariant functions of $G^F\uni$.  In \cite[V]{L3}, Lusztig gave
an algorithm of computing $X_i$ as an explicit linear combination of
those functions $Y_j$. In view of (1.4.4), the determination of generalized
Green functions is reduced to the determination of the functions $Y_i$. 
But the computation of $Y_i$ is very subtle. 

\para{1.5.}
Since $\SE$ is a simple local system, the isomorphism $\psi : F^*\SE \isom \SE$
is unique up to scalar. Among them, the typical isomorphism $\psi_0 : F^*\SE \isom \SE$ 
is given as follows; choose $u \in C^F$, and put $A_G(u) = Z_G(u)/Z_G^0(u)$. 
Then $F$ acts on $Z_G(u)$, and acts on $A_G(u)$ as a permutation of connected components
of $Z_G(u)$. 
Let $\tau$ be the action of $F$ on $A_G(u)$, and consider the semidirect product 
$\wt A_G(u) = \lp\tau\rp \ltimes A_G(u)$.
Let $A_G(u)\wg\ex$ be the set of irreducible representations of $A_G(u)$ (up to isomorphism)
which are extendable to an irreducible representation of $\wt A_G(u)$. 
Assume that $(u,\r) \lra (C,\SE) \in \SN_G^F$, then 
$\r \in A_G(u)\wg\ex$.  The condition
$F^*\SE \simeq \SE$ is equivalent to saying that there exists an extension
$\wt\r$ of $\r$ such that the stalk $\SE_u$ has the structure of $\wt A_G(u)$-module
corresponding to $\wt\r$. 

\par
We define $\psi_0: F^*\SE \isom \SE$ by the condition that the induced map 
on $\SE_u$ coincides with the action of $\tau$ on $\SE_u = \wt\r$, and define 
a function $Y_i^0$ on $G\uni^F$ in a similar way as $Y_i$, but by replacing $\psi$ by $\psi_0$.
Note that the set of $G^F$-conjugacy classes in $C^F$ is in bijection with 
the set of $A_G(u)$-conjugacy classes in the coset $A_G(u)\tau \subset \wt A_G(u)$.    
We denote by $u_a$ the $G^F$-class corresponding to the class $a\tau \in A_G(u)\tau$.
Then $Y_i^0$ is explicitly given as follows;

\begin{equation*}
\tag{1.5.1}
Y_i^0 = \begin{cases}
          \Tr(a\tau, \wt\r) &\quad\text{ if $g$ is $G^F$-conjugate to $u_a$, } \\
           0                  &\quad\text{ if $g \notin C^F$.}
        \end{cases}
\end{equation*} 

\par
Note that $Y_i^0$ is computable by the formula (1.5.1).  Since $\psi = \g \psi_0$ 
for some $\g \in \Ql^*$, we have $Y_i = \g Y_i^0$.  Thus the computation of
$Y_i$ is reduced to the determination of $\g$. 
We choose $u_0 \in C_0^F$, and consider $\wt A_L(u_0) = \lp\tau_0\rp \ltimes A_L(u_0)$, 
where $\tau_0$ is the permutation action of $F$ on $A_L(u_0)$.  
Let $\r_0 \in A_L(u_0)\wg\ex$ be the one corresponding to $\SE_0$, and 
$\wt\r_0$ its extension to $\wt A_L(u_0)$. 
We can define $\vf_0 : F^*\SE_0 \isom \SE_0$ by the condition that 
the induced map on $(\SE_0)_{u_0}$ coincides with the action of $\tau_0$ on 
$\wt\r_0$. 
Thus $\g$ depends on the choice of $(u_0, \wt\r_0, u, \wt \r)$, which we denote by
$\g = \g(u_0, \wt\r_0, u, \wt\r)$. 
Note that $\vf_0: F^*\SE_0 \isom \SE_0$ satisfies the requirement of $\vf_0$
in 1.3, since the order of $\tau_0$ is finite. 

\para{1.6.}
Here we review the definition of the preferred extensions following
\cite[V, 17.2]{L3}.  Let $(W, S)$ be a Coxeter system of a Weyl group $W$,
and $\s : W \to W$ the Coxeter group automorphism of order $c$.  We denote by
$\wt W = \lp \s\rp \ltimes W$ the semidirect product of $W$ with the cyclic group
$\lp \s\rp$ of order $c$. Let $E$ be an irreducible $W$-module (over $\Ql$)
which is extendable to a $\wt W$-module. Then $E$ can be extended in $c$ different ways
to a $\wt W$-module.  A particular type of extension, called the preferred extension,
$\wt E$ of $E$ is defined for each case separately as follows.
\par\medskip
(a) \ $W$ is irreducible, and $c = 1$, then take $\wt E = E$.
\par\smallskip
(b) \ $W$ of type $A_n$ ($n \ge 2$) or $E_6$, and $c = 2$.
Define $\wt E$ by the condition that $\s : E \to E$ acts as $(-1)^{a_E}\cdot w_0$,
where $w_0$ is the longest element in $W$, and $a : E \to a_E$ is the $a$-function
defined in \cite[4.1]{L1}. 
\par\smallskip
(c) \ $W$ of type $D_4$ and $c = 3$.  Define $\wt E$ to be the unique extension of
$E$ which is defined over $\BQ$.
\par\smallskip
(d) \ $W$ of type $D_n$ ($n \ge 4$) and $c = 2$.
In this case, $\wt W$ is identified with the Weyl group $W_n$ of type $B_n$.
The irreducible representations of $W_n$ are parametrized by the pairs $(\la; \mu)$ of
partitions,
$\la : 0 \le \la_1 \le \cdots \le \la_r, \mu : 0 \le \mu_1 \le \cdots \le \mu_r$
such that $\sum \la_i + \sum \mu_i = n$
(with the same $r$). 
We denote by the same symbol $(\la;\mu)$ the corresponding irreducible $W_n$-module. 
If $\la \ne \mu$, the restriction of $(\la;\mu)$ gives an irreducible $W$-module,
and in that case, $(\la;\mu)$ and $(\mu;\la)$ give isomorphic $W$-modules.
Any extendable
irreducible $W$-module $E$ is obtained in this way, and the two extensions to $\wt W$ are given
by $(\la;\mu)$ and $(\mu;\la)$.  Now the preferred extension $\wt E$ of $E$ is given by
$(\la;\mu)$, where $\la_i > \mu_i$ if $i$ is the smallest integer such that $\la_i \ne \mu_i$.
\par\smallskip
(e) \ Assume that $W = W_1 \times W_2 \times \cdots \times W_r$, with $W_i$ irreducible
Weyl groups mutually isomorphic,
and $\s$ permutes cyclically the factors $W_i$. 
Then $E$ can be written as an external tensor product
$E = E_1 \boxtimes E_2 \boxtimes \cdots \boxtimes E_r$, and for $E \in W\wg\ex$,
the preferred extension $\wt E$ is constructed from the preferred extension $\wt E_r$
with respect to the isomorphism $\s^r: W_r \isom W_r$ as in \cite[IV, 17.2]{L3}.  
\par\smallskip
(f) \ In the general case, one can write uniquely
$W = W^{(1)} \times \cdots \times W^{(t)}$ and
$E = E^{(1)}\boxtimes\cdots \boxtimes E^{(t)}$,
where $W^{(i)}$ are $\s$-stable Weyl groups satisfying the condition in (e), and
$E^{(i)}$ are irreducible $W^{(i)}$-modules. Define the preferred extension
$\s : \wt E \isom  \wt E$ by the condition that $\s$ is an external tensor
product of the maps $\s^{(i)}$ on $\wt E^{(i)}$ constructed as in (e).

\para{1.7.}
Returning to the original setup in 1.4, we consider the generalized Green functions
$Q^G_{L^w, C_0^w, \SE_0^w, \vf_0^w}$ associated to $(L, C_0, \SE_0) \in \SM_G^F$.
Note that these functions are determined from a given isomorphism
$\vf_0: F^*\SE_0 \isom \SE_0$.  We choose such an isomorphism $\vf_0$ for each
$(L,C_0,\SE_0) \in \SM_G^F$ as above, and fix a set

\begin{equation*}
\tag{1.7.1}  
\Xi_G = \{ \vf_0 = \vf_{(C_0,\SE_0)} : F^*\SE_0 \isom \SE_0 \mid (L, C_0, \SE_0) \in \SM_G^F\}.
\end{equation*}  
Hence $\Xi_G$ determines the set of generalized Green functions
$\{ Q^G_{L^w, C_0^w, \SE_0^w, \vf_0^w} \}$.
\par
For each $F$-stable unipotent class $C$ of $G$, we shall define a specific
representative $u \in C^F$, called a split element. 
By induction on $\dim G$, if $L \ne G$, then we may assume that there exists
a split element $u_0 \in C_0^F$.  Assume that $(C_0, \SE_0) \lra (u_0, \r_0)$.
Then $\vf_0: F^*\SE_0 \isom \SE_0$ determines the extension $\wt\r_0 $
of $\r_0$ to $\wt A_L(u_0)$.
\par
We define a split element $u \in C^F$, associated to $\Xi_G$, for various cases separately.
\par\medskip
(a) \ Assume that $G$ is a simply connected, almost simple group, with $F$ split.
First consider the case where $G$ is not of type $E_8$.
If $(C,\SE) \in \SN_G^F$ is not cuspidal, then $(u,\r) \lra (C,\SE)$ belongs to
$(L, C_0, \SE_0)$ with $L \ne G$.  Then for $\r \in A_G(u)\wg\ex$, and
its extension $\wt\r$, one can consider $\g(u_0, \wt\r_0, u, \wt\r)$.
In the case where $(C, \SE)$ is cuspidal, $\vf_0: F^*\SE_0 \isom \SE_0$ for $\SE_0 = \SE$ 
is given by (1.7.1).  This determines an extension $\wt\r$ of
$\r \in A_G(u)\wg\ex$.
Now $u \in C^F$ is called a split element if for each $\r \in A_G(u)\wg\ex$,
there exists an extension $\wt\r$ of $\r$ such that

\begin{equation*}
\tag{1.7.2}  
\begin{aligned}  
  \g(u_0, \wt\r_0, u, \wt\r) &= 1,
        \quad\text{ if $(u,\r)$ not cuspidal,} \\
  \g(u, \wt\r, u, \wt\r)     &= 1,  \quad\text{ if $(u,\r)$ cuspidal.}
\end{aligned}
\end{equation*}  
Note that the second condition just requires that the function
$\vf_0: F^*\SE \isom \SE$ coincides with $\psi_0: F^*\SE \isom \SE$ for $\SE = \SE_0$,
and is automatically satisfied since $\psi_0$ satisfies the condition for $\vf_0$ in 1.3. 
\par
Next consider the case where $G$ is of type $E_8$
(for the notation used below, see Section 4).  Unless $q \equiv -1 \pmod 3$,
we use the same formula as (1.7.2) for defining a split element in $C^F$.
Now assume that $q \equiv -1 \pmod 3$.  If $C$ is not of type $D_8(a_3)$, (1.7.2) is
also used.  So assume that $C$ is of type $D_8(a_3)$.
In this case, any $(C,\SE)$ belongs to $(T, \{ 1\}, 1)$. For $u \in C$,
$A_G(u) \simeq S_3$.  Let $\ve$ be the sign representation of
$S_3$.  We define a split element $u \in C^F$ by the condition that

\begin{equation*}
\tag{1.7.3}  
  \g(u_0, \wt\r_0, u, \r) = \begin{cases}
                   1  &\quad\text{ if $\r \ne \ve$, }, \\
                  -1  &\quad\text{ if $\r = \ve$,}
                            \end{cases}
\end{equation*}
where $(u_0, \wt\r_0) = (1,1)$. 

\par\smallskip
(b) Assume that $G$ is a simply connected, almost simple group of type $A_n$ or $E_6$,
with $F$ non-split. Let $B$ be a Borel subgroup of $G$, and for $u \in G\uni$, define
$\SB_u = \{ gB \in G/B \mid g\iv ug \in B\}$.  We put $d_u = \dim \SB_u$.
Take $(C, \SE) \in \SN_G^F$ such that $(C,\SE)$ belongs to
$\eta = (L, C_0, \SE_0) \in \SM_G^F$.
Let $E = V_{(u,\r)} \in (\SW_L)\wg\ex$ for $(u, \r) \lra (C, \SE)$.
Let us define an integer $\d_E$ by 

\begin{equation*}
\tag{1.7.4}  
  \d_E  = \begin{cases}
             a_E - d_u  &\quad\text{ if $\SW_L$ : type $A_m$ or $E_6$, } \\
              0         &\quad\text{ otherwise. }
          \end{cases}
\end{equation*}

We define a split element $u \in C^F$ by the condition that there exists
an extension $\wt\r$ of $\r$ such that 

\begin{equation*}
\tag{1.7.5}  
\begin{aligned}  
  \g(u_0, \wt\r_0, u, \wt \r) &= \nu_{u, \eta}(-1)^{\d_E},
              &\quad &\text{ if $(u,\r)$ not cuspidal, with $E = V_{(u,\r)}$, }   \\
  \g(u, \wt\r, u, \wt \r)     &= 1     &\quad &\text{ if $(u,\r)$ cuspidal,}    
\end{aligned}
\end{equation*}
where $\nu_{u, \eta}$ is a root of unity determined by the class $C$ and $\eta$.
For the explicit description of $\nu_{u,\eta}$, see Section 2 (resp. Section 9)
for $G = SL_n$ (resp. for $G = E_6$). 

\par\smallskip
(c) \ Assume that $G$ is a simply connected, almost simple group of
type $D_4$ with $F$ non-split such that $F^3$ is split.  In this case,
we define a split element $u \in C^F$ by the condition (1.7.2). 

\par\smallskip
(d) \ Assume that $G$ is a simply connected, almost simple group of
type $D_n$ with $F$ non-split such that $F^2$ is split.  In this case,
we define a split element $u \in C^F$ by the condition (1.7.2).  

\par\smallskip
(e) \ Assume that $G = G_1 \times G_2 \times \cdots \times G_r$, where 
$G_i$ is a simply connected, almost simple group, and
$F$ permutes cyclically the factors $G_i$; $F : G_1 \to G_2 \to \cdots \to G_r \to G_1$.
Hence $F^r : G_1 \to G_1$ is a Frobenius map on $G_1$. 
Then a unipotent class $C = C_1 \times \cdots \times C_r$ in $G$ is $F$-stable
if and only if $F$ permutes factors $C_i$ cyclically, and
$u_1 \mapsto u = (u_1, F(u_1), F^2(u_1), \cdots F^{r-1}(u_1))$ gives a bijection 
$C_1^{F^r} \isom C^F$. 
We define a split element $u \in C^F$ by the condition that $u_1 \in C_1^{F^r}$
is a split element in the sense of (a) $\sim$ (d).

\par\smallskip
(f) \ Assume that $G = G^{(1)} \times G^{(2)} \times \cdots \times G^{(t)}$, and
$G^{(i)}$ are $F$-stable such that $(G^{(i)},F)$ satisfies the condition in (e) for each $i$.
Let $C = C^{(1)} \times C^{(2)} \times \cdots \times C^{(t)}$
be an $F$-stable unipotent class in $G$.
Then $C^{(i)}$ is $F$-stable for each $i$. We define a split element
$u = (u^{(1)}, \dots, u^{(t)}) \in C^F$ by the condition that
$u^{(i)} \in (C^{(i)})^F$ is split for each $i$ in the sense of (e). 

\par\smallskip
(g) \ Assume that $G$ is a connected reductive group of general type.
Let $\pi : G \to G/Z_G^0$ be the natural projection, where $G/Z_G^0$ is semisimple.
We define a split element $u \in G^F$ by the condition that $\pi(u)$ is
split in the semisimple group $(G/Z_G^0)^F$.

\para{1.8.}
Let $G$ be connected reductive, and $C$ an $F$-stable
unipotent class in $G$.  Assume that there exists a split element $u \in C^F$.
It is clear that if $u_1 \in C^F$ is $G^F$-conjugate to $u$, 
then $u_1$ is also split.  
\par
We consider the converse of the above situation.
Hence assume that $u, u_1 \in C^F$, and choose $h \in G$ such that 
$u_1 = huh\iv$ with $h\iv F(h) = \da \in Z_G(u)$, where $\da$ is a representative
of $a \in A_G(u)$.  
Then $\ad h$ gives an isomorphism $A_G(u) \isom A_G(u_1)$. It is easy to
check that under this isomorphism, the Frobenius action of $F$ 
on $A_G(u_1)$ corresponds to the action of $aF$ on $A_G(u)$, 
where $a$ acts as an inner automorphism on $A_G(u)$.  We now assume that 
the isomorphism $A_G(u) \isom A_G(u_1)$ can be extended to the isomorphism 
$\wt A_G(u) \isom \wt A_G(u_1)$, via $\tau \mapsto \tau_1$. This means that 
$a \in A_G(u)$ is a central element. 
For each $\r \in A_G(u)\wg$, $\r_1 \in A_G(u_1)\wg$ is determined by 
the isomorphism $A_G(u) \isom A_G(u_1)$. If $\r$ is extendable, then
$\r_1$ is also extendable, and an extension $\wt\r$ determines a unique 
extension $\wt\r_1$ of $\r_1$. 
The following lemma shows that a split element is unique up to $G^F$-conjugate, 
if it exists. Thus obtained $G^F$-class in $C^F$ is called the split class. 

\begin{lem} 
Let $u, u_1 \in C^F$ be elements such that  
$\wt A_G(u) \isom \wt A_G(u_1)$ as in 1.8.
Assume that $u, u_1$ satisfy the condition that 
$\g(u_0, \wt\r_0, u, \wt\r) = \g(u_0, \wt\r_0, u_1, \wt\r_1)$
for any pair $(C,\SE) \in \SN_G^F$, where $\wt\r$ and $\wt\r_1$ are
related as in 1.8.  
Then $u_1$ is $G^F$-conjugate to $u$.     
\end{lem}
\begin{proof}
Let $u, u_1 \in C^F$ be as in the lemma. 
We assume that $u_1 = huh\iv$ with $h \in G$ such that $h\iv F(h) = \da \in Z_G(u)$. 
For $i = (C, \SE) \in \SN_G^F$, let $Y_i^0$ be the function defined in terms 
of $u \in C^F$, and $Y'^0_i$ the function defined in terms of $u_1 \in C^F$. 
Thus $Y_i^0(x) = \Tr(c\tau, \wt\r)$ if $x$ is $G^F$-conjugate to $u_c$, 
and $Y'^0_i(x) = \Tr(c_1\tau_1, \wt\r_1)$ if $x$ is $G^F$-conjugate to $(u_1)_{c_1}$.
Since $\dc_1 \in Z_G(u_1) = hZ_G(u)h\iv$, one can write as 
$h\iv \dc_1 h = \dc \in Z_G(u)$. 
Since $(u_1)_{c_1} = gu_1g\iv$ for $g \in G$ such that $g\iv F(g) = \dc_1$, we have
\begin{equation*}
\dc = h\iv \dc_1 h = h\iv (g\iv F(g))h = (gh)\iv F(gh)\cdot \da\iv, 
\end{equation*}    
and so $\dc \da = (gh)\iv F(gh) \in Z_G(u)$.  This means that $(u_1)_{c_1}$ is 
$G^F$-conjugate to $u_{ca}$.  By assumption, we have 
$Y_i^0 = \g\iv Y_i = Y_i'^0$, where
$\g = \g(u_0, \wt\r_0, u, \wt\r) = \g(u_0, \wt\r_0, u_1, \wt\r_1)$. It follows that 
\begin{equation*}
\tag{1.9.1}
\Tr(c_1\tau_1, \wt\r_1) = \Tr(ca\tau, \wt\r).
\end{equation*}
On the other hand, by the isomorphism $\wt A_G(u) \isom \wt A_G(u_1)$, 
we have 
\begin{equation*}
\tag{1.9.2}
\Tr(c_1\tau_1, \wt\r_1) = \Tr(c\tau, \wt\r).
\end{equation*} 
By (1.9.1) and (1.9.2), the equality 
$\Tr(ca\tau, \wt\r) = \Tr(c\tau, \wt\r)$ holds for any $\r \in A_G(u)\wg\ex$.
By using the orthogonality relations of $\t$-twisted class functions on $A_G(u)$, we see 
that $ca$ is $\t$-twisted conjugate to $c$ for any $c \in A_G(u)$, and so 
$a$ is $\t$-twisted conjugate to 1. This means that $u_1$ is $G^F$-conjugate to
$u$.  The lemma is proved.  
\end{proof}

The purpose of this paper is to prove that the following conjecture
holds for a certain choice of $\Xi_G$,
possibly except one class in $E_7$. 
Note that the conjecture implies that the generalized Green
functions $Q^G_{L^w, C_0^w, \SE_0^w, \vf_0^w}$ are computable. 

\par\medskip\noindent
\addtocounter{thm}{1}
{\bf Conjecture 1.10.}  
{\it Let $G$ be a connected reductive group with Frobenius map $F$.
For an $F$-stable unipotent class $C$ in $G$, there exists
a split element $u \in C^F$. Moreover, for any $\r \in A_G(u)\wg\ex$,
an extension $\wt\r$ of $\r$ satisfying the conditions appeared
in (a) $\sim$ (g) in 1.7 is uniquely determined, which we call
the split extension of $\r$.} 

\para{1.11.}
The proof of the conjecture is reduced to the case where
$G$ is a simply connected, almost simple group.  
In the case of exceptional groups, we make use of the computation of
Green functions by \cite{Spr}, \cite{Sh2}, \cite{BS} for the case $p$ is good,
and by \cite{G2}, \cite{Lu} for the case $p$ is bad.

\para{1.12.}  The main ingredient for the proof is again    
a variant of the restriction theorem as in 
\cite[II,III]{Sh4}, which we review below.
First recall the original restriction theorem.  
\par
Under the setup in 1.1, fix $(L, C_0, \SE_0) \in \SM_G$.
Let $P$ be a parabolic subgroup of $G$ 
whose Levi subgroup is $L$.  We consider another parabolic subgroup $Q \supset P$ of $G$ 
with Levi subgroup $M \supset L$ and the unipotent radical $U_Q$.
Then $\SW_L' = N_M(L)/L$ is in a natural way 
a subgroup of $\SW_L$.
Let $C$ (resp. $C'$) be a unipotent class in $G$ (resp. in $M$),
and choose $u \in C, u' \in C'$.
Put 
\begin{align*}
\tag{1.12.1}
Y_{u,u'} &= \{g \in G \mid g\iv ug \in u'U_Q\}, \\
\CQ_{u,C'} &= \{ gQ \in G/Q \mid g\iv ug \in C'U_Q\}.
\end{align*}
We have a map $q: Y_{u,u'} \to \CQ_{u,C'}$ by $g \mapsto gQ$.
$q$ is $Z_G(u)$-equivariant under the natural action of $Z_G(u)$ on them.
$Z_G(u) \times Z_M(u')U_Q$ acts on $Y_{u,u'}$, and 
we have an isomorphism 
\begin{equation*}
\tag{1.12.2}
Y_{u,u'}/Z_M(u')U_Q \simeq \CQ_{u,C'},
\end{equation*}
which is compatible with the action of $Z_G(u)$. 

\par
Let $X_{u,u'}$ be the 
set of irreducible components of $Y_{u,u'}$ with possible maximal dimension 
$d_{u,u'} = (\dim Z_G(u) + \dim Z_M(u'))/2 + \dim U_Q$.  Then 
$A_G(u) \times A_M(u')$ acts on $X_{u,u'}$ as a permutation of irreducible 
components.  Denote by $\ve_{u,u'}$ this permutation representation 
of $A_G(u) \times A_M(u')$ on $X_{u,u'}$.
We consider the generalized Springer correspondence for $M$. 
Take $(C',\SE') \in \SN_M$ with $(C',\SE') \lra (u', \r')$.
Let $V_{(u',\r')}$ be 
the irreducible representation of $\SW_L'$ corresponding to $(u', \r') \in \SN_M$
under (1.1.4).  
Then the restriction theorem (\cite[1.2 (II)]{Sp2}) asserts that 
\begin{equation*}
\tag{1.12.3}
\lp \ve_{u,u'}, \r \otimes \r'^*\rp_{A_G(u) \times A_M(u')}  
       = \lp V_{(u,\r)}|_{\SW_L'}, V_{(u',\r')}\rp_{\, \SW_L'},
\end{equation*}
if $(u,\r)$ and $(u',\r')$ belong to the same series $(L,C_0,\SE_0)$, and 
the left hand side is equal to zero if $(u,\r)$ and $(u',\r')$ belong to 
different series. 
Moreover, every irreducible representation of 
$A_G(u) \times A_M(u')$ which occurs in $\ve_{u,u'}$ is obtained in this way. 

\para{1.13.}
We assume that $L \subset P, M \subset Q$ are all $F$-stable. 
Consider $Y_{u,u'}$ in (1.12.1), and assume that $u, u'$ are $F$-stable. 
Then $F$ acts naturally on $Y_{u,u'}$, and acts on $X_{u,u'}$ as a permutation
of irreducible components. 
Put $A(u,u') = A_G(u) \times A_M(u')$, then $F$ acts diagonally on $A(u,u')$.
We define $\wt A(u,u') = \lp \tau \rp \ltimes A(u,u')$, where 
$\lp \tau \rp$ is an infinite group generated by $\tau$, 
and $\tau$ acts on $A(u,u')$ via $F$. 
Now $\ve_{u,u'}$ turns out to be 
an $\wt A(u,u')$-module, which we denote by $\wt\ve_{u,u'}$. 
In turn, if we choose extensions $\wt\r, \wt\r'$ of $\r, \r'$, respectively, 
it defines an extension of $\r \otimes \r'^*$ 
to the irreducible $\wt A(u,u')$-module $\wt{\r \otimes \r'^*}$. 

\par
For $(u,\r) \in \SN_G^F$, let $\s_{(u,\r)} : V_{(u,\r)} \isom  V_{(u,\r)}$ 
be the isomorphism  defined by the condition that
the restriction of $\vf$ on $V_{(u,\r)} \otimes F^*\IC(\ol C, \SE)[\dim C]$
coincides with $\s_{(u,\r)}\otimes \wt\psi_0$, where $\wt\psi_0$ is the isomorphism
$F^*\IC(\ol C,\SE)[\dim C] \isom \IC(\ol C,\SE)[\dim C]$ 
induced from $\psi_0: F^*\SE \isom \SE$. 
(Don't confuse $\s_{(u,\r)}$ with the preferred extension $\s_E$ for $E = V_{(u,\r)}$.) 
Similarly, for the reductive group $M$, and $(C',\SE') \in \SN_M^F$, 
we obtain the map $\s_{(u',\r')} : V_{(u',\r')} \isom V_{(u',\r')}$, 
where $(u',\r') \lra (C', \SE')$. 
We can decompose $V_{(u,\r)}$ as $\SW_L'$-module,    
\begin{equation*}
\tag{1.13.1}
V_{(u,\r)} = \biggl(\bigoplus_{(u',\r') \in \SN_M^F} M_{\r, \r'}
                      \otimes V_{(u',\r')}\biggl) \oplus V_{(u,\r)}',
\end{equation*}
where $M_{\r, \r'} = \Hom_{\SW_L'} (V_{(u',\r')}, V_{(u,\r)}|_{\SW_L'})$ is the 
multiplicity space for $V_{(u',\r')}$, and $V'_{(u,\r)}$ is a sum of $V_{(u',\r')}$
such that $V_{(u',\r')} \notin (\SW_L')\wg\ex$.  
Thus we can define an isomorphism $\s_{\r,\r'}: M_{\r,\r'} \isom M_{\r,\r'}$
such that the restriction of $\s_{(u,\r)}$ on $M_{\r,\r'}\otimes V_{(u',\r')}$
coincides with $\s_{\r,\r'}\otimes\s_{(u',\r')}$.    

\par
A variant of the restriction theorem (1.12.3)
was proved in \cite[II, Cor. 1.9]{Sh4}.  
Under the notation above, the following formula holds.
\begin{equation*}
\tag{1.13.2}
\Tr(\s_{\r,\r'}, M_{\r,\r'}) = q^{-(\dim C - \dim C')/2 + \dim U_Q}
                 \lp \wt\ve_{u,u'}, \wt{\r \otimes \r'^*}\rp_{A(u,u')\tau}, 
\end{equation*}
where for representations $V_1, V_2$ of $\wt A(u,u')$, 
$\lp V_1 ,V_2\rp_{A(u,u')\tau}$ is defined as
\begin{equation*}
\lp V_1, V_2\rp_{A(u,u')\tau} = |A(u,u')|\iv \sum_{a \in A(u,u')}
                                  \Tr(a\tau,V_1)\Tr((a\tau)\iv, V_2). 
\end{equation*} 

\para{1.14.}
In order to obtain an information on $\g$, we apply the restriction formula (1.13.2).
Here we consider $F$-stable Levi subgroups $L \subset M \subset G$ as in 1.13.  
Consider $(u,\r) \in \SN_G^F$ and $(u',\r') \in \SN_M^F$ with 
$u \in C^F, u' \in C'^F$. 
Take $(u_0,\r_0) \in \SN_L^F$ with $u_0 \in C_0^F$, and fix an extension 
$\wt\r_0$ of $\r_0$. Assume that an extension 
$\wt\r'$ of $\r'$ is fixed.
Assume that $\r\otimes\r'^*$ appears in the decomposition of $\ve_{u,u'}$. 
Let $\ve_{u,u'}(\r,\r')$ be the $\r\otimes\r'^*$ isotypic component of 
the $A(u,u')$-module $\ve_{u,u'}$.   Then $F$ acts naturally on $\ve_{u,u'}(\r,\r')$, 
and we regard it as a submodule of $\wt\ve_{u,u'}$. 
Assume that $F$ acts trivially on $\ve_{u,u'}(\r,\r')$.   
Then there exists a unique extension $\wt\r$ of $\r$ 
with $\wt{\r\otimes\r'^*} = \wt\r \otimes \wt\r'^*$ such that  
\begin{equation*}
\tag{1.14.1}
\lp \wt\ve_{u,u'}, \wt\r \otimes \wt\r'^*\rp_{A(u,u')\tau}
          = \lp \ve_{u,u'}, \r\otimes \r'^*\rp_{A(u,u')}. 
\end{equation*}  

On the other hand, let $\wt V_{(u,\r)}$ (resp. $\wt V_{(u',\r')}$) be the 
preferred extension of $V_{(u,\r)}$ (resp. $V_{(u',\r')}$). We consider 
the following condition. 
\par\medskip\noindent
(1.14.2) \ Let $V_{\r,\r'} = M_{\r,\r'}\otimes V_{(u',\r')}$ be the 
$V_{(u',\r')}$-component of $V_{(u,\r)}$ as in (1.13.1).  Then the restriction 
of $\wt V_{(u,\r)}$ on $V_{\r,\r'}$ coincides with $\d \otimes \wt V_{(u',\r')}$, 
where $\d$ is a scalar multiplication by a $c$-th root of unity for $c$ as in 1.6.     

\par\medskip 
The following lemma is a refinement of Lemma 4.6 in \cite[II]{Sh4}.
  
\begin{lem}  
Assume that $\g(u_0, \wt\r_0, u', \wt\r') = \nu'$. 
Under the notation of 1.14, assume that $F$ acts trivially on 
$\ve_{u,u'}(\r,\r')$. Let $\wt\r$ be the extension of $\r$ which 
satisfies the condition in (1.14.1). 
Then (1.14.2) holds with respect to $V_{(u,\r)}$ and $V_{(u',\r')}$, 
and we have $\g(u_0, \wt\r_0, u, \wt\r) = \d\iv\nu'$. 
\end{lem} 

\begin{proof}
By (1.13.2) and (1.14.1), we have 
\begin{align*}
\tag{1.15.1}
\Tr(\s_{\r,\r'}, M_{\r,\r'}) &= q^{-(\dim C - \dim C')/2 + \dim U_Q}
                                 \lp \ve_{u,u'}, \r\otimes \r'^*\rp_{A(u,u')} \\
                             &= q^{-(\dim C - \dim C')/2 + \dim U_Q}\dim M_{\r,\r'}.
\end{align*}
The second equality follows from the original restriction formula (1.12.3). 
Since $q^{(\dim C - \dim C')/2 - \dim U_Q}\s_{\r,\r'}$ is an automorphism 
of finite order on $M_{\r,\r'}$,  (1.15.1) implies that 
$\s_{\r,\r'}$ coincides with the scalar multiplication by  
$q^{-(\dim C - \dim C')/2 + \dim U_Q}$.  
This implies that (1.14.2) holds. 
\par
By our assumption, we have $\g(u_0, \wt\r_0, u', \wt\r') = \nu'$.  It follows that 
$q^{-(a_0' + r')/2}\s_{(u',\r')}$ gives the preferred extension of $V_{(u',\r')}$,
multiplied by $\nu'$, where $a_0', r'$ are
defined similarly to (1.4.1) with respect to $C' \subset M$.
This implies, by (1.14.2), 
that the restriction of $q^{-(a_0 + r)/2}\s_{(u,\r)}$ on 
$V_{\r,\r'}$ coincides with (the restriction of) 
the preferred extension of $V_{(u,\r)}$, multiplied by $\d\iv\nu'$.
Hence $q^{-(a_0+r)/2}\s_{(u,\r)}$ 
is the preferred extension of $V_{(u,\r)}$, multiplied by $\d\iv\nu'$, 
and so $\g(u_0,\wt\r_0, u, \wt\r) = \d\iv\nu'$ as asserted.  
\end{proof}

\para{1.16.}
Here we explain more precisely Lusztig's algorithm in \cite[V]{L3} of expressing 
the function $X_i$ in terms of $Y_j$. Under the notation in 1.4, 
write $X_i$ as a linear combination of $Y_j$ for $j \in I^F$ as

\begin{equation*}
\tag{1.16.1}
X_i = \sum_{j \in I^F}p_{ji}Y_j, \quad (p_{ij} \in \Ql).
\end{equation*}

For $i,j \in I^F$, we define $\w_{ij}$ as follows. 
If $i = (C,\SE), j = (C',\SE')$ belong to the same series $(L,C_0, \SE_0)$,
put
\begin{align*}
\tag{1.16.2}  
  \w_{ij} &= |\SW_L|\iv \sum_{w \in \SW_L}
             \Tr((w \s_{E_i})\iv, \wt E_i)\Tr(w \s_{E_j}, \wt E_j) \\
          &\qquad\times \frac{|G^F|}{|Z^{0F}_{L^w}|}q^{-\dim G}q^{-(a_0 + a_0')/2},
\end{align*}  
where $a_0 = -\dim C - \dim Z_L$, $a_0' = -\dim C' - \dim Z_L$.
Furthermore, $E_i = V_{(C,\SE)} \in (\SW_L)\wg\ex$,
and $\wt E_i$ is the preferred extension of $E_i$. 
We put $\w_{ij} = 0$ if $i$ and $j$ belong to different series.
We define matrices 
$P = (p_{ij}), \ \Om = (\w_{ij})$  indexed by $I^F$.  
It is proved in \cite[V, Thm. 24.4]{L3} the matrices $P$ and $\Om$
satisfy the matrix equation,

\begin{equation*}
\tag{1.16.3}
{}^tP\vL P = \Om
\end{equation*}
for a certain matrix $\vL = (\la_{ij})$.
Moreover, for a suitable choice of the total order on $I^F$,
$P$ turns out to be a block-wise upper triangular matrix where the diagonal blocks
are identity matrices, and $\vL$ is a non-singular, block-wise diagonal matrix.
Here $\Om = (\w_{ij})$ is computable if we know the character table of $\wt\SW_L$.
Then the matrices $P$ and $\Om$ are determined uniquely from the matrix equation
(1.16.3).
\par
Also it was proved in [loc. cit.] that $p_{ij} = 0$ if $i$ and $j$ belong to
different series in $\SM_G$.
Thus if we denote by $I_{\eta}$
a subset of $I$ consisting of $i$ belonging to a fixed $\eta = (L, C_0, \SE_0) \in \SM_G^F$,
and consider the submatrices $P_{\eta}, \vL_{\eta}, \Om_{\eta}$ indexed by $I^F_{\eta}$,
then the matrix equation (1.16.3) can be rewritten more effectively as

\begin{equation*}
\tag{1.16.4}  
{}^tP_{\eta}\vL_{\eta}P_{\eta} = \Om_{\eta}, \qquad (\eta \in \SM_G^F).
\end{equation*}  

\para{1.17.}
Assume that $G = SL_n$ or a simply connected group of type $E_6$.
We consider a non-split Frobenius map $F$ on $G$.
Let $(L, C_0, \SE_0) \in \SM_G$. 
In addition to $\SW_L = W$ for $L = T$, we need to consider $\SW_L \simeq S_m$
for some $m \ge 1$ if $G = SL_n$. While in the case $G = E_6$, $L$
is of type $D_4$ ($p = 2$) or $2A_2$ ($p \ne 3$), and $\SW_L \simeq S_3$ or
isomorphic to the Weyl group $W(G_2)$ of type $G_2$, respectively, for $L \ne T$. 
In any case $L$ is $F$-stable, and 
$F$ induces a non-trivial action on $\SW_L \simeq S_m$, and the trivial action on
$\SW_L \simeq W(G_2)$. 
\par
We also consider another group $G_0$ of the same type as $G$, but $G_0$ has
an $\BF_{q_0}$-structure with a split Frobenius map $F_0$, where $q_0$ is a power
of a prime number $p_0$, and $p \ne p_0$ in general. 
We compare the matrix equations  (1.16.4)
for $(G,F)$ and $(G_0, F_0)$.
We assume that
\par\medskip\noindent
(1.17.1) \ 
There exists an embedding $(Z_G\wg)^F \hra (Z_{G_0}\wg)^{F_0}$. 
\par\medskip
Let $(L, C_0, \SE_0) \in \SM_G$. 
Then $L$ and $C_0$ are both $F$ and $F_0$-stable.
Fix a character $\x \in (Z_G\wg)^F$, then by (1.17.1), 
$\x$ is regarded as an $F_0$-stable character of $Z_{G_0}$. 
We assume that there is a natural bijection $(\SM_G)_{\x} \simeq (\SM_{G_0})_{\x}$.
We consider $\eta = (L, C_0, \SE_0) \in (\SM_G)_{\x}$.
Then $\eta$ is also regarded as an element in $(\SM_{G_0})_{\x}$, and we denote
the corresponding element in $(\SM_{G_0})_{\x}$ as 
$(L,C_0, \SE_0)$ by using the same symbol, by abbreviation.
We assume that $(L, C_0, \SE_0)$ is $F$ and $F_0$-stable. 
\par
Let $w_0$ be the longest element in $\SW_L$.
Note that $|G_0^{F_0}|$ and $|Z_{L_0^w}^{0F_0}|$ are regarded as rational functions on $q_0$. 
Here we use the following notation;
for a rational function $f = f(q_0)$ on $q_0$, we denote by $f(-q)$
the rational function obtained from $f$ by substituting $-q$ for $q_0$.
We assume that $F$ and $F_0$
satisfy the condition that
\begin{equation*}
\tag{1.17.2}  
\begin{aligned}
|G^F| &= (-1)^{\dim T}|G_0^{F_0}|(-q), \\
|Z_{L^{ww_0}}^{0F}| &= (-1)^{\dim Z_L}|Z_{L_0^w}^{0F_0}|(-q)
        &\quad&\text{ if $\SW_L  \simeq S_m$, }  \\
|Z_{L^{w}}^{0F}| &= (-1)^{\dim Z_L}|Z_{L_0^w}^{0F_0}|(-q)
        &\quad&\text{ if , $\SW_L \simeq W(G_2)$,}
\end{aligned}
\end{equation*}
where $T$ is a maximal torus in $G$, $L_0^w$ is a Levi subgroup in $G_0$ corresponding to
a Levi subgroup $L^w$ in $G$.
We assume further that there is a natural bijection 
$(\SN_G)_{\x} \simeq (\SN_{G_0})_{\x}$, and $I_{\eta}^F \simeq (I_0)_{\eta}^{F_0}$,
where $I = \SN_G$ and $I_0 = \SN_{G_0}$. 
\par
Let $\w_{ij}, p_{ij}, \la_{ij}$\ ($i,j \in I_{\eta}^F$)
be the elements defined in 1.16 for $(G,F)$, and
denote by $\w_{ij}^0, p_{ij}^0, \la_{ij}^0$ \ ($i,j \in (I_0)_{\eta}^{F_0}$)
the corresponding elements for $(G_0, F_0)$.
By (1.16.2), $\w_{ij}^0$ is regarded as a rational function on $q_0$
(note that in the split case the preferred extension is trivial). 
Since $p_{ij}^0, \la_{ij}^0$ are obtained from the matrix equation (1.16.3),
they are also rational functions on $q_0$.
Thus $\w_{ij}^0(-q), p_{ij}^0(-q), \la_{ij}^0(-q)$ can be defined. 
\par
Recall the definition $\d_E$ in (1.7.4). Assume that $i = (C,\SE)$ with $u \in C$.
Let $E_i = V_{(u,\r)} \in (\SW_L)\wg\ex$ for $(u,\r) \lra (C,\SE)$ and put $\d_i = \d_{E_i}$.
We show the following. 

\begin{lem}   
For $i,j \in I^F_{\eta} = (I_0)^{F_0}_{\eta}$, we have 
\begin{enumerate}
\item \ 
$\w_{ij} = (-1)^{\d_i + \d_j}\w^0_{ij}(-q)$. 
\item \ 
$p_{ij}  = (-1)^{\d_i + \d_j}p_{ij}^0(-q)$.
\item \
$\la_{ij} = (-1)^{\d_i + \d_j}\la^0_{ij}(-q)$.
\end{enumerate}
\end{lem}  

\begin{proof}
Assuming that $\SW_L \ne W(G_2)$, we prove (i).  The preferred extension $\wt E_i$ 
is given by the action of $\s_{E_i} = (-1)^{a_{E_i}}w_0$ on
$E_i = V_{(u,\r)}\in (\SW_L)\wg\ex$. 
Hence, for $i = (C,\SE), j = (C',\SE')$,  by (1.16.2) we have 

\begin{align*}
\tag{1.18.1}
\w_{ij} &= |\SW_L|\iv 
        \sum_{w \in \SW_L}\Tr((ww_0)\iv, E_i)\Tr(ww_0, E_j) \\  
     &\qquad \times  (-1)^{a_{E_i} + a_{E_j}}
         \frac{|G^F|}{|Z_{L^w}^{0F}|} q^{-\dim G}q^{-(a_0 + a_0')/2}  \\
        &= |\SW_L|\iv \sum_{w \in \SW_L}
        \Tr(w\iv, E_i)\Tr(w, E_j)  \\
     &\qquad \times(-1)^{a_{E_i} + a_{E_j}}
            \frac{|G^F|}{|Z_{L^{ww_0}}^{0F}|} q^{-\dim G}q^{-(a_0 + a_0')/2},
\end{align*}
where $a_0 = -\dim C - \dim Z_L, a_0' = -\dim C' - \dim Z_L$.
Since $a_0 = 2d_u - \dim G + \dim T - \dim Z_L$, we have
\begin{equation*}
(a_0 + a_0')/2 + \dim G = d_u + d_{u'} + \dim T - \dim Z_L. 
\end{equation*}  
Hence by (1.17.2), the last expression in (1.18.1) is equal to

\begin{align*}
  &|\SW_L|\iv \sum_{w \in \SW_L}\Tr(w\iv, E_i)\Tr(w, E_j)  \\
      &\qquad \times (-1)^{\d_i + \d_j}\frac{|G_0^{F_0}|(-q)}{|Z_{L_0^w}^{0F_0}|(-q)}\cdot 
           (-q)^{-\dim G_0}(-q)^{-(a_0 + a_0')/2}  = (-1)^{\d_i + \d_j}\w^0_{ij}(-q). 
\end{align*}  
Hence we obtain the equality in (i).  The case where $\SW_L \simeq W(G_2)$
is dealt similarly (in this case, the preferred extensions are trivial).
Thus (i) holds. 
\par
Next we show (ii) and (iii). 
Let $\Om = (\w_{ij}), P = (p_{ij}), \vL = (\la_{ij})$ be the matrices
associated to $(G, F)$.  Similarly, we define the matrices 
$\Om^0 = (\w_{ij}^0), P^0 = (p_{ij}^0), \vL^0 = (\la^0_{ij})$ associated
to $(G_0, F_0)$ (actually we consider the matrices indexed by
$I^F_{\eta}$ or $(I_0)^{F_0}_{\eta}$ as in (1.16.4) for
a fixed $\eta$.  But in the computation below, we omit $\eta$).
By (1.16.4), those matrices satisfy the relations
\begin{equation*}
{}^tP\vL P = \Om, \qquad {}^tP^0 \vL^0 P^0 = \Om^0.
\end{equation*}
Define a diagonal matrix $D$ by $D = ((-1)^{\d_i})$.
Then $D^2 = I$, and by (i) we have
$\Om^0(-q)  = D\Om D$. It follows that
\begin{equation*}
{}^t(DPD)(D \vL D)(DPD) = \Om^0(-q) = {}^tP^0(-q)\vL^0(-q)P^0(-q). 
\end{equation*}  
Here $DPD$ is a block-wise upper triangular matrix, where the diagonal blocks
are identity matrices, and $D\vL D$ is a
block-wise diagonal matrix. Hence by the uniqueness of the matrix equation (1.16.4),
we have $DPD = P^0(-q)$, and $D\vL D = \vL^0(-q)$. This proves (ii) and (iii). 
\end{proof}

\par\bigskip
\section{ Special linear groups }

\para{2.1.}
In this section, we assume that $G$ is a special linear group. 
We prove the existence of split elements for $G$.
Actually, in the case where $F$ is split, split elements were
explicitly constructed in \cite[I, Thm. 3.4]{Sh4}.  
In the case where $F$ is non-split, they were constructed
in \cite[I, Thm. 4.4]{Sh4}, under the assumption that $p$ is large enough.
After that, for any $p$ in the non-split $F$ case,
split elements were constructed in \cite[III, Thm. 4.13]{Sh4}, in a recursive
way by making use of the restriction formula such as Lemma 1.15.
But the proof contains some error.  Although the recursive process is
done through the procedures (I), (II), (III) in \cite[III, 4.5]{Sh4},
there exists a case in (III) where Lemma 1.16 cannot be applied.
\par
The simplest example is as follows.  We consider $G = SL_3$, and
let $W = S_3$ be the Weyl group of $G$, and $\s : W \to W$ the automorphism
of oder 2.  The $\s$-stable subgroup $W' = \{ 1\}$ corresponds to
the Weyl group of $M  = T$.
Let $\la = (12)$ be a partition of 3, and let
$C = C_{\la}$ the corresponding unipotent class in $G$. Take $(u,\r) \lra (C,\SE)$
and consider $E = V_{(u,\r)} \in S_3\wg$, which is an irreducible representation of
degree 2.  We consider the class $C' = \{ 1\}$, and $(u',\r') = (1,1) \in \SN_M$. 
Let $\s_E : \wt E \to \wt E$ be the preferred extension of $E$.
Then $\Tr(\s_E, E) = 0$, and $\s_E$ acts non-trivially on $M_{\r,\r'} = E$
(in the notation of 1.14).  In this case, the formula such as (1.14.1) does not
give any information for $\wt{\r\otimes \r'^*}$ in $\wt\ve_{u,u'}$. 

\para{2.2.}
In the discussion below, by expanding the method in \cite[I, \S 3]{Sh4}, we prove
the existence of split elements for any $p$, in a simultaneous way for the case where
$F$ is split or non-split.  In particular, in the non-split case, split
elements are given by an exact formula, not as in the recursive formula in
\cite[III, \S 4]{Sh4}. 
\par
Assume that $G = SL_n = SL(V)$, with $F$ split or non-split.
The unipotent classes of $G$ are parametrized by partitions of $n$.
We denote by $C_{\la}$ the class corresponding to the partition
$\la = (\la_1, \dots, \la_r)$ of $n$ with $0 \le \la_1 \le \la_2 \le \cdots \le \la_r$. 
Let $n'$ be the largest divisor of $n$ which is prime to $p$, and
$n'_{\la}$ the greatest common divisor of $n', \la_1, \dots, \la_r$.
Then $Z_G$ is a cyclic group of order $n'$, and $A_G(u)$ is a cyclic group
of order $n'_{\la}$. The natural map $\pi : Z_G \to A_G(u)$ is given by
the surjective map $\BZ_{n'} \to \BZ_{n'_{\la}}$. 
\par
The generalized Springer correspondence (1.2.1) pertaining to $\x \in Z_G\wg$
is given as follows (see \cite[\S 5]{LS2}). 
Let $A_G(u)\wg_{\x}$ be as in 1.2 for $\x \in Z_G\wg$.  Then
for $\x$ of order $d$, $|A_G(u)\wg_{\x}| = 1$ if $\la_i$ is divisible by $d$
for any $i$, and is 0 otherwise.
In the former case, $A_G(u)\wg_{\x}$ consists of $\r \in A_G(u)\wg$ such that
the pullback of $\pi^*(\r)$ of $\r \in A_G(u)\wg$ coincides with $\x$.
Let $(L, C_0, \SE_0) \in \SM_{\x}$.  Then $L$ is a Levi subgroup of type
$A_{d-1} + \cdots + A_{d-1}$ with $m = n/d$ factors, and $C_0$ is the regular
unipotent class in $L$. If $(u_0, \r_0) \lra (C_0, \SE_0)$, then
$\r_0$ is the unique element contained in $A_L(u_0)\wg_{\x}$.
In particular, $\SM_{\x}$ is a one-point set. 
In this case, $\SW_L \simeq S_m$, and $V_{(u,\r)}$ coincides with
the irreducible representation of $S_m$ of type $\mu = (\mu_1, \dots, \mu_r)$
with $\mu_i = \la_i/d$.

\para{2.3.}
Any unipotent class $C$ in $G$ is $F$-stable.  For $(L, C_0, \SE_0) \in \SM_G$,
we may assume that $L$ is an $F$-stable Levi subgroup of an $F$-stable parabolic
subgroup of $G$, and that $C_0$ is $F$-stable.  For each $C = C_{\la}$,
we shall define a ``split''element $u = u_{\la} \in C^F$ as follows.
First assume that $F$ is split.  We fix an $F$-stable  basis $e_1, \dots, e_n$
of $V$, up to $SL(V)$-conjugate. We define $u = u_{\la} \in C^F$ by
using the Jordan basis $\{ v_{k,j} \mid 1 \le k \le r, 1 \le j \le \la_k \}$ of $V$
such that $(u-1)v_{k,j} = v_{k, j-1}$ for $j \ge 2$, and $(u-1)v_{k,1} = 0$ for $j = 1$. 
Here we assume that the Jordan basis is obtained from the basis $e_1, \dots, e_n$
in such a way that $v_{1,1}, \dots, v_{1,\la_1}, v_{2,1}, \dots, v_{2,\la_2}, \dots$
coincides with $e_1, \dots, e_n$ in this order.
\par
Next assume that $F$ is non-split with respect to the $\BF_q$-structure of $G$.
Then $F^2$ is a split Frobenius map with respect to the $\BF_{q^2}$-structure.
Let $e_1, \dots, e_n$ be the $\BF_{q^2}$-basis of $V$, and $V_0$ be the $\BF_{q^2}$-subspace
of $V$ spanned by $e_1, \dots, e_n$.
Define $u_{\la} \in C^{F^2}$ by applying the above discussion to the split Frobenius $F^2$.
We define a sesquilinear form $(\ ,\ )$ on $V_0$ as in (4.7.1) in \cite[III]{Sh4} by
\begin{equation*}
\tag{2.3.1}  
  (v_{k,j}, v_{k',j'}) = \begin{cases}
                           (-1)^{j + a_k}  &\quad\text{ if } k = k', j + j' = \la_k, \\
                           0               &\quad\text{ otherwise, }
                         \end{cases}
\end{equation*}  
where $a_k = \pm 1$.  Then $u_{\la}$ leaves the form $(\ ,\ )$ invariant, and
we have $u_{\la} \in C^F$ with respect to the $\BF_q$-structure induced from $(\ ,\ )$.
In later discussion, we shall prove that $u_{\la} \in C^F$ actually gives a split element,
under a suitable choice of $\Xi_G$ (see (1.7.1)). 
\par
For $F$ : split or non-split, let $P = LU_P$ be an $F$-stable parabolic subgroup of $G$ such that
$L$ is an $F$-stable Levi subgroup of type $A_{d-1} + \cdots + A_{d-1}$
with $(L, C_0, \SE_0) \in \SM_G^F$.  
We define $u'_{\la} \in C_0^F$ as the projection of $u_{\la} \in P^F$ to $L^F$.  Explicitly,
$u'_{\la}$ is defined by the condition that
\begin{equation*}
\tag{2.3.2}  
  (u'_{\la}-1)v_{k,j} = \begin{cases}
                           0   &\quad\text{ if } j \equiv 1 \pmod d, \\
                           v_{k, j-1}   &\quad\text{ otherwise.}
                        \end{cases}
\end{equation*}  

By applying the procedures (e) and (f) in 1.7, one can define
a ``split'' element $u_0 \in C_0^F$, starting from  the split element $u' \in C'^F$, where
$C'$ is the regular unipotent class in $SL_d$. 
Then $u_0$ coincides with $u'_{\la} \in C_0^F$ for $\la = (n)$, i.e, $u_{\la}$
is a regular unipotent element in $G^F$.
\par
Assume that $F$ is non-split. 
Assume that $(u_0, \r_0) \lra (C_0, \SE_0)$.
We have $A_L(u_0) \simeq \BZ_{d'}$, where $d'$ is the largest divisor of $d$
which is prime to $p$.  We consider the extended group $\wt A_L(u_0)$, where
$F$ acts on $A_L(u_0)$ via $\tau_0$.
We choose an isomorphism $\vf_0 : F^*\SE_0 \isom \SE_0$
in $\Xi_G$ as the trivial extension $\wt\r_0$ of
$\r_0 \in A_L(u_0)\wg\ex$ (note that $A_L(u_0)$ is abelian).
For a given $u_{\la} \in C^F$, let $u'_{\la} \in C_0^F$ be as above.
Then $u'_{\la}$ is obtained from $u_0$ by twisting with $c_{\la} \in A_L(u_0)$, where
$c_{\la} \in A_L(u_0)$ is determined uniquely, up to $F$-conjugacy.
Since $\r_0 \in A_L(u_0)\ex\wg$, $\r_0(c_{\la}) = \wt\r_0(c_{\la}\tau_0)$
is determined uniquely, which we denote
by $\nu_{\la} = \nu_{\la, \eta}$ for $\eta = (L, C_0, \SE_0) \in \SM_G$. 
\par
In the case where $F$ is split, for a given $u_{\la} \in C^F$, $u'_{\la} \in C_0^F$
is defined similarly to the non-split case. In that case, we have $c_{\la} = 1$ and
$\nu_{\la, \eta} = 1$ for any $\la$. 
\remark{2.4.}
The elements $u_{\la}, u'_{\la}$ and $u_0$ defined in 2.2 correspond
to $y_{\la}, y_1$ and $y_0$ in \cite[I, 4.2]{Sh4}, and $c_{\la}, \nu_{\la,\eta}$
correspond to $c_1, \eta_{\la}$ there.

\para{2.5.}
Here we recall the construction of the complex $K$ on $G$ given in 1.1
more precisely.  For $(L, C_0, \SE_0) \in \SM_G$,
put $\vS = Z_L^0 \times C_0 \subset L$, and let $\ol \vS$ be the closure
$\vS$ in $L$. Put

\begin{align*}
  \wt X &= \{ (x, gP) \in G \times G/P \mid g\iv xg \in \ol\vS U_P\}, \\
  \wt X_0 &= \{ (x, gP) \in G \times G/P \mid g\iv xg \in \vS U_P\}, \\
  \wh X_0 &= \{ (x, g) \in G \times G \mid g\iv xg \in \vS U_P\}, \\
  \ol Y &= \bigcup_{g\in G}g(\ol\vS U_P)g\iv,  
\end{align*}
and consider the diagram

\begin{equation*}
\begin{CD}
  \vS  @<\vf <<  \wh X_0  @>\psi >> \wt X_0    
\end{CD}    
\end{equation*}
\par\medskip\noindent
defined by $\vf : (x, g) \mapsto \text{ projection of $g\iv xg$ to $\vS$}$, 
$\psi : (x, g) \mapsto (x, gP)$. 
We define a map $\f : \wt X \to \ol Y$ by $(x, gP) \mapsto x$. Then $\f$ is
a proper map.
Let $\SE_1 = \Ql\boxtimes \SE_0$ be the $L$-equivariant local system on $\vS$.
Then there exists a local system $\ol\SE_1$ on $\wt X_0$ such that
$\vf^*\SE_1 \simeq \psi^*\ol\SE_1$.  Since $\wt X_0$ is open dense in $\wt X$,
one can consider the complex $K_{\ol\SE_1} = \IC(\wt X, \ol\SE_1)[\dim \wt X]$. 
We define a complex $K$ on $G$ by $K = \f_*K_{\ol\SE_1}$. 
Then $K$ turns out to be a semisimple perverse sheaf on $G$, equipped with
$\SW_L$-action. 
\par
Under the setup in 1.4, assume that $(C,\SE) \in \SN_G^F$ belongs to 
the series $(L, C_0, \SE_0) \in \SM_G^F$. 
For $u \in C^F$, consider the varieties
\begin{align*}
  \SP_u &= \{ gP \in G/P \mid g\iv ug \in C_0U_p\}, \\
  \wh \SP_u &= \{ g \in G \mid g\iv ug \in C_0U_P\},  \\
\end{align*}
and the diagram 
\begin{equation*}
\tag{2.5.1}
\begin{CD}
  C_0    @<\a<<   \wh\SP_u   @>\b>>  \SP_u,  
\end{CD}    
\end{equation*}
where $\a$ is defined by $g \mapsto$ projection of $g\iv ug$ on $C_0$,
and $\b : g \mapsto gP$.  
\par
Let $K$ be as above. 
Let $\dot\SE_0$ be the local system on $\SP_u$ defined by
the condition that $\a^*\SE_0 \simeq \b^*\dot\SE_0$. 
Then $\SH^a_uK \simeq H^a(\ol\SP_u, D)$, where
$\ol\SP_u = \{ gP \in G/P \mid g\iv ug \in \ol C_0U_P\}$, and
$D$ is the restriction of $K_{\ol\SE_1}$ on $\ol\SP_u$. 
It is known that $\IC(\ol C_0, \SE_0)$ is clean, namely,
$\IC(\ol C_0, \SE_0)$ is $\SE_0$ extended by 0 on $\ol C_0 - C_0$. 
It follows that $K_{\ol\SE_1}$ coincides with $\dot\SE_0[r]$ extended by 0
on $\ol\SP_u - \SP_u$, where $r = \dim Y$. 
We obtain (see \cite[V, (24.2.5)]{L3}) that 
\begin{equation*}
\tag{2.5.2}  
\SH^a_uK \simeq H^{a+r}_c(\SP_u, \dot\SE_0).
\end{equation*}  
Since $\SH^a_uK$ has a structure of $\SW_L$-module, $H^{a+r}_c(\SP_u, \dot\SE_0)$
has a structure of $\SW_L$-module.
It is known by \cite[1.2 (b)]{L2} that $\dim \SP_u \le (a_0 + r)/2$.
Since $\SH^{a_0}_uK \ne 0$, we see that
\begin{equation*}
\tag{2.5.3}  
\dim \SP_u = (a_0+r)/2.
\end{equation*}

The isomorphism $\vf_0 : F^*\SE_0 \isom \SE_0$ induces an isomorphism
$\dot\vf_0 : F^*\dot\SE_0 \isom \dot\SE_0$.  Then $\dot\vf_0$ induces an
action $\Phi$ on $H_c^{a + r}(\SP_u, \dot\SE_0)$.  In the case where $a = a_0$,
$\Phi$ is a scalar multiplication
$q^{(a_0+r)/2}$ times a map of finite order. 

\para{2.6.}
We apply the discussion in 2.5 for $G = SL_n$, where
$L$ is of type $A_{d-1} + \cdots + A_{d-1}$ with $m = n/d$-factors, and
$C_0$ is the regular unipotent class in $L$.  For the moment, we don't
consider the $\BF_q$-structure. 
Let $u \in C = C_{\la}$ with $\la = d\mu$.  Consider
$\SP_u = \{ gP \in G/P \mid g\iv ug \in C_0U_P\}$. We also
consider $\SP_{u_0}$ by regarding $u_0 \in C_0$ as an element in $G$
(of Jordan type $(d, \dots, d)$, $m$ copies).
Let $\SF$ be the set of flags $D = (V_d \subset V_{2d} \subset \cdots \subset V_{(m-1)d})$
with $\dim V_{kd} = kd$.  Then $\SF \simeq G/P$.
For $u \in C$ as above, let $\SF_u$ be the set of $u$-stable flags
such that $u$ induces a regular unipotent transformation on $V_{kd}/V_{(k-1)d}$.
Then $\SF_u \simeq \SP_u$. 
\par
Let $\SG$ be the set of subspaces $W$ of $\dim W = d$.  We have a natural map
$\pi : \SF \to \SG$ by $D = (V_d, \dots, V_{(m-1)d}) \mapsto V_d$.
Let $\SG_u$ be the set of $W \in \SG$ such that $W$ is $u$-stable, and that
$u|_W$ is regular unipotent. Then $\pi$ induces a surjective map
$\pi : \SF_u \to \SG_u$.
Now $\SG_u$ can be identified with the variety $\BP(\Ker x^d) - \BP(\Ker x^{d-1})$
for $x = u-1$; for each $v \in \Ker x^d - \Ker x^{d-1}$, the space spanned by
$v, xv, \dots x^{d-1}v$ gives an element $V_d$ in $\SG_u$.
\par
We fix $V_d  \in \SG_u$, and put $\ol V = V/V_d$.  Let $\ol u \in SL(\ol V)$
be the element induced from $u \in SL(V)$.  Then we have a natural isomorphism
\begin{equation*}
\tag{2.6.1}  
\pi\iv(V_d) \simeq \ol\SP_{\ol u},
\end{equation*}
where $\ol\SP_{\ol u}$ is the variety defined for $\ol u \in SL(\ol V)$. 

\par
Write $\la$ as $\la = (a_1^{m_1}, a_2^{m_2}, \dots a_h^{m_h})$, with
$0 < a_1 < a_2 < \cdots < a_h$, where $m_i = \sharp\{ k \mid \la_k = a_i\}$.
Put $W = \Ker x^d$ and define, for $i = 1, \dots, h$,
a subspace $W^i$ of $W$ by $W^i = (\Ker x^d \cap \Im x^{a_i-1}) + \Ker x^{d-1}$. 
Then we have a filtration of $W$,
\begin{equation*}
  W  = W^h \supsetneq W^{h-1} \supsetneq \cdots
                  \supsetneq W^1 \supsetneq W^0 =  \Ker x^{d-1}. 
\end{equation*}  
Here we have $\dim W^i/W^{i-1} = m_i$ for each $i$. 
The following fact is easily checked.
\par\medskip\noindent
(2.6.2) \ For each $v \in W^i - \Ker x^{d-1}$, define $V_d \in \SG_u$
as above.  Then $\ol u \in SL(\ol V)$ has Jordan type $\la'$,
where $\la'$ is a partition of $n-d$ given by 
\begin{equation*}
\la' = (a_1^{m_1}, \dots, a_i -d, a_i^{m_i-1}, \dots, a_h^{m_h}). 
\end{equation*}  

\par\medskip
Let $\{ v_{k,j}\}$ be the Jordan basis of $V$
associated to $u$.
Then the image of $\{ v_{k, d} \mid 1 \le k \le r\}$ gives a basis of $\Ker x^d/\Ker x^{d-1}$,
where $r = \sum_{1 \le i \le h}m_i = \dim \Ker x^d/\Ker x^{d-1}$.
For each $1 \le k \le h$, define a subset $Z_k$ of $W$ by

\begin{equation*}
Z_k = \{ v_{k,d} + \sum_{s < k}b_s v_{s,d} \mid b_s \in \Bk\} \simeq \BA^{k-1}.
\end{equation*}  
Then $Z_k$ is a locally closed subvariety  of $\BP(W)$, and
\begin{equation*}
\tag{2.6.3}  
\BP(W^i) - \BP(W^{i-1}) = \bigsqcup_{k = m_1 + \cdots + m_{i-1} + 1}^{m_1 + \cdots + m_i}Z_k
\end{equation*}  
gives a paving of $\BP(W^i) - \BP(W^{i-1})$ by affine spaces.
The following result can be proved in a similar way as in Spaltenstein \cite{Sp1},
which corresponds to the case where $G = GL(V), d= 1$, and $\SP_u = \SB_u$. 

\begin{lem} 
We regard $Z_k$ as a subvariety of $\SG_u$.  Let $V_d$ be an element in $\SG_u$
generated by $v_{k,d} \ (1 \le k \le r)$, and $\ol u$ the element
in $SL(\ol V)$ for $\ol V = V/V_d$.
Then we have
\begin{equation*}
\tag{2.7.1}  
\pi\iv(Z_k) \simeq Z_k \times \pi\iv(V_d) \simeq \BA^{k-1} \times \ol\SP_{\ol u}. 
\end{equation*}  
\end{lem}  

\para{2.8.}
If $D = (V_d, \dots, V_{(m-1)d}) \in \SF_u$,
then $u_k = u|_{V_{kd}/V_{(k-1)d}}$ induces a regular transformation
on $\ol V_{kd} = V_{kd}/V_{(k-1)d}$. Thus we can define a regular element 
$(u_1, \dots, u_m)$ in $L$, which is conjugate to $u_0 \in C_0$.  It follows that
\begin{equation*}
\tag{2.8.1}  
\SF_u \subset \SF_{u_0}.
\end{equation*}  

We define a specific representative $u'_0 \in C_0$ attached to $u$ by
\begin{equation*}
\tag{2.8.2}  
  (u'_0-1)v_{k,j} = \begin{cases}
                0 &\quad\text{ if } j \equiv 1 \mod d, \\
                v_{k, j-1}  &\quad\text{ otherwise.}
             \end{cases}      
\end{equation*}  
(Hence this $u'_0$ corresponds to $u'_{\la}$ for $u = u_{\la}$ in the
notation of (2.3.2).)
We define $\SF_{u'_0}$, $\SG_{u'_0}$ and $\pi_0 : \SF_{u'_0} \to \SG_{u'_0}$,
similarly to $\SP_u, \SG_u$, etc.. 
Put $x_0 = u'_0-1$.  If we put $W = \Ker x^d$, and $V_0 = \Ker (x_0)^d$,
then we have  $W \subset V_0$. 

\par
By (2.8.1), $\SP_u$ is regarded as a subvariety of $\SP_{u'_0}$. 
The following result is proved as a corollary to Lemma 2.7.

\begin{prop}  
Let $u \in C$, and $u'_0 \in C_0$.
\begin{enumerate}
\item \ $\SP_u$ has an affine paving.
\item $\SP_{u_0'} - \SP_u$ has an affine paving.  
\end{enumerate}  
\end{prop}  

\begin{proof}
We show (i).  Since $\SG_u \simeq \BP(\Ker x^d) - \BP(\Ker x^{d-1})$,
by (2.6.3), we have

\begin{equation*}
\SF_u \simeq \bigsqcup_{k = 1}^h \pi\iv(Z_k).
\end{equation*}  
By induction, we may assume that $\ol\SP_{\ol u}$ has an affine paving.
Hence by Lemma 2.7, $\pi\iv(Z_k)$ has an affine paving.  Thus (i) holds.
\par
Next we show (ii). We have
\begin{equation*}
  \BP(\Ker x^d) - \BP(\Ker x^{d-1})  \subset \BP(\Ker x_0^d) - \BP(\Ker x_0^{d-1}),
\end{equation*}
and so $\SG_u \subset \SG_{u'_0}$. Let $X_1 = \pi_0\iv(\SG_u)$, and
$X_2 = \pi_0\iv(\SG_{u_0'} - \SG_{u})$. Since $\SP_u \subset X_1$, 
we have
\begin{equation*}
\SP_{u'_0} - \SP_u = (X_1 - \SP_u) \sqcup X_2. 
\end{equation*}
\par\medskip
Here $X_1 - \SP_u = \bigsqcup_{k = 1}^h(\pi_0\iv(Z_k) - \pi\iv(Z_k))$, and by a similar
discussion as in Lemma 2.7, we have
\begin{equation*}
\tag{2.9.1}  
\pi_0\iv(Z_k) - \pi\iv (Z_k) \simeq Z_k \times (\pi_0\iv (V_d) - \pi\iv (V_d)).
\end{equation*}  
Since $\pi_0\iv(V_d) - \pi\iv(V_d) \simeq \ol\SP_{\ol u'_0} - \ol\SP_{\ol u}$,
by induction and by (2.9.1), $\pi_0\iv(Z_k) - \pi\iv(Z_k)$ has an affine paving.
Hence $X_1 - \SP_u$ also has an affine paving. 
On the other hand, by applying the discussion in the proof of Lemma 2.7 for $u'_0$,
we see that $X_2$ has an affine paving.  Hence (ii) follows. 
\end{proof}

\para{2.10.}
Consider the isomorphism
$\pi_0\iv(Z_k) - \pi\iv(Z_k) \simeq Z_k \times (\ol\SP_{\ol u'_0} - \ol\SP_{\ol u})$
in (2.9.1)). 
Recall that $\dot\SE_0$ is the local system on $\SP_u$, and let $\dot\SE_0'$
be the local system on $\ol\SP_{\ol u}$ defined similarly to $\dot\SE_0$.
Then by a similar discussion as in \cite[I, (3.6.8)]{Sh4}, we see that
the restriction of $\dot\SE_0$ on $\pi_0\iv(Z_k) - \pi\iv(Z_k)$ is isomorphic to
$\Ql \boxtimes \dot\SE'_0$. Similarly,
in the formula $\pi\iv(Z_k) \simeq Z_k \times \ol \SP_{\ol u}$ in (2.7.1),
the restriction of $\dot\SE_0$ on $\pi\iv(Z_k)$ is isomorphic to $\Ql\boxtimes \dot\SE_0'$.
($\Ql$ is the constant sheaf on $Z_k \simeq \BA^{k-1}$.)
It follows, from Proposition 2.9, that

\begin{equation*}
\tag{2.10.1}  
H_c^{\odd}(\SP_{u'_0}-\SP_u, \dot\SE_0) = 0, \qquad H_c^{\odd}(\SP_u, \dot\SE_0) = 0. 
\end{equation*}  

We have the following result.

\begin{prop}  
Let $i : \SP_u \hra \SP_{u'_0}$ be the closed immersion.
\begin{enumerate}
\item 
The induced map $i^* : H^{2a}_c(\SP_{u'_0}, \dot\SE_0) \to H^{2a}_c(\SP_u,\dot\SE_0)$
is surjective.
\item
Then map $i^*$ is $\SW_L$-equivariant.   
\end{enumerate}
\end{prop}

\begin{proof}
(i) follows from the cohomology long exact sequence obtained from
the closed immersion $\SP_u \hra \SP_{u'_0}$, by using (2.10.1). 
\par
We show (ii).  Let $\Fg$ be the Lie algebra of $G$, and $\Fc_0$
the nilpotent orbit in $\Fg$ corresponding to $C_0$. We denote by
the same symbol $\SE_0$ the local system on $\Fc_0$ corresponding to
$\SE_0$ on $C_0$. Let $\Fc$ be the $G$-orbit in $\Fg$ corresponding to $C$.
Take $x_0 \in \Fc_0$ and $x \in \Fc$. Let $\Fn_P$ be the nilpotent radical
of $\Lie P$. Similarly to the group case, we define varieties
\begin{align*}
  \wt X &= \{ (x, gP) \in \Fg \times G/P \mid \Ad(g)\iv x \in \ol \vS + \Fn_P\}, \\
      \ol Y &= \bigcup_{g \in G}\Ad(g)(\ol\vS + \Fn_P). 
\end{align*}
and define $\f : \wt X \to \ol Y$ by $(x, gP) \mapsto x$,
   where $\ol\vS = \Lie Z_L \times \ol\Fc_0$.  
We can define a $\SW_L$-complex $K = \f_*K_{\ol\SE_1}$ on $\ol Y$ as in the group case. 
Put 
\begin{align*}
\SP_x = \{ gP \in G/P \mid Ad(g)\iv x \in \Fc_0 + \Fn_P\},
\end{align*}
and similarly define $\SP_{x_0}$.  Let $\dot\SE_0$ be the local system
on $\SP_x$.  Then $H^i_c(\SP_x, \dot\SE_0)$ has a structure of $\SW_L$-module, and
$H^i_c(\SP_x, \dot\SE_0) \simeq H^i_c(\SP_u, \dot\SE_0)$ as $\SW_L$-modules.
A similar property holds also for $\SP_{x_0}$. Hence it is enough to show (ii)
by replacing $G$ by $\Fg$. 
\par
By Spaltenstein \cite{Sp3}, under the condition that $p$ is a good prime (hence
for any $p$ in the case $G = SL_n$), that
there exists a 1-parameter subgroup $\la : \Bk^* \to G$ and a linear subspace
$\BSigma  \subset \Fg$ such that

\begin{equation*}
\Ad(\la(t))x_0 = t^{-c}x_0  \qquad \text{ with }  \quad c > 0. 
\end{equation*}
$\Ad \la(t)$ stabilizes $\BSigma$, and its weights on $\BSigma$ are of the form
$\xi(t) = t^b$ with $b \ge 0$.  Moreover, $\dim \BSigma = \dim Z_G(x_0)$. 
\par
Put $S = x_0 + \BSigma$.  Then $S$ is a transversal slice in $\Fg$ to the $G$-orbit
$\Fc_0$. Since $\Fc_0 \subset \ol\Fc$, $S \cap \ol\Fc$ is a transversal slice
in $\ol\Fc$ to $\Fc_0$. 
Since $S \cap \ol\Fc \ne \emptyset$, one can choose $x \in S$. Now
by considering the contraction to $x$ by the $\BG_m$-action on $S$, we have
\begin{equation*}
  H^a_c(S, K) \simeq \SH^a_{x_0}K
          \simeq H_c^{a+r}(\SP_{x_0}, \dot\SE_0),
\end{equation*}
and these isomorphisms are $\SW_L$-equivariant. 
On the other hand, the inclusion relation $\{ x\} \hra S$ induces
a $\SW_L$-equivariant map
\begin{equation*}
H^a_c(S, K) \to H_c^a(\{ x \}, K) = \SH^a_xK \isom  H_c^{a+r}(\SP_x, \dot\SE_0).
\end{equation*}
Hence we have a $\SW_L$-equivariant map
$H^{a+r}_c(\SP_{x_0}, \dot\SE_0) \to H^{a+r}_c(\SP_x, \dot\SE_0)$, which is nothing
but the map $i^*$ in (ii).  Thus (ii) is proved. 
\end{proof}

\para{2.12.}
Under the setup in 2.6, we consider the $\BF_q$-structure. 
Assume that $F$ is a Frobenius map on $G$, split or non-split type.
We choose $u = u_{\la} \in C^F$ and $u_0' = u_{\la}' \in C_0^F$ as in 2.3.
Then $\SP_u$ is an $F$-stable subvariety of $\SP_{u_0'}$, and
the map $i^* : H^{2a}_c(\SP_{u_0'}, \dot\SE_0) \to H^{2a}_c(\SP_u, \dot\SE_0)$
in Proposition 2.11 is $\Phi$-equivariant. 
\par
Since $p$ is good for $G$,
$\SH^a_gK$ is $(a+r)$-pure for $g \in \ol C^F$ and for all $a$ (\cite[V, (24.8.3)]{L3}).
Hence the eigenvalues of $\Phi$ on $H_c^{a+r}(\SP_u, \dot\SE_0)$ are algebraic numbers
all of whose complex conjugate have absolute value $q^{(a+r)/2}$.
Now assume that $F$ is split on $G$.  Then it is possible to choose $Z_k$ in Lemma 2.7
so that the isomorphism (2.7.1) is defined over $\BF_q$. (This does not hold
if $F$ is non-split.)  Thus also possible to choose an affine paving of $\SP_u$
in Proposition 2.9 so that it is defined over $\BF_q$.
In particular, the eigenvalues of $\Phi$ on $H^{2a}_c(\SP_u, \dot\SE_0)$
are of the form $q^b$ with $b \le a$.
It follows that

\begin{lem} 
Assume that $F$ is split.  Then the eigenvalues of
$\Phi$ on $H_c^{2a}(\SP_u, \dot\SE_0)$ coincide with $q^a$. 
\end{lem}

\para{2.14.}
Now assume that $F = F'_0\s$ is a non-split Frobenius map on $G$,
with respect to the $\BF_q$-structure, 
where $F'_0$ is a split Frobenius map, and $\s$ is an automorphism
on $G$ given by $g \mapsto n_0 ({}^tg\iv) n_0\iv$; 
$n_0 \in N_G(T)$ is a representative of the longest element
in $W = N_G(T)/T$.
Here $Z_G \simeq \BZ_{n'}$, and $F$ (resp. $F_0'$) acts on $Z_G$ via
$x \mapsto x^{-q}$ (resp. $x \mapsto x^q$) on $Z_G$.
\par
We note that
\par\medskip\noindent
(2.14.1) \ There exists a power $q_0$ of a prime number $p_0$ such that
$p_0$ is prime to $n$ 
and that
\begin{equation*}
  q_0 \equiv -q \pmod{n'}.
\end{equation*}
\par\medskip
In fact, since $q$ is a power of $p$ and $(p, n') = 1$,
$-q$ is prime to $n'$. Hence by Dirichlet's theorem on
arithmetic progression, there exists a prime number $p_0$ such that
$p_0 \equiv -q \pmod{n'}$.  We may choose $p_0$ large enough so that
$p_0 > n$.  Hence (2.14.1) holds for $q_0 = p_0$. 
\par\medskip
Let $G_0 = SL_n$ be a group different from $G$ with a split Frobenius map $F_0 : G_0 \to G_0$
with respect to an $\BF_{q_0}$-structure.
We have $Z_G\wg \simeq \BZ_{n'}$ and $Z_{G_0}\wg \simeq \BZ_n$.
As an additive group $\BZ_n \simeq \BZ/n\BZ$, we have a decomposition 
$\BZ_n \simeq \BZ_{n'}\oplus \BZ_{p^a}$ for $n = p^an'$, and
$\BZ_{n'}$ is regarded as a $\BZ$-submodule of $\BZ_n$.
\begin{lem}  
The embedding $\BZ_{n'} \hra \BZ_n$ induces an embedding
  $(Z_G\wg)^F \hra (Z_{G_0}\wg)^{F_0}$ as stated in (1.17.1).
\end{lem}  
\begin{proof}
The action of $F$ on $Z_G\wg \simeq \BZ_{n'}$ is given by
$x \mapsto (-q)x$, and the action of $F_0$ on $Z_{G_0}\wg \simeq \BZ_n$ is given by
$x \mapsto q_0x$.
Hence
\begin{equation*}
  \BZ_{n'}^F = \{ x \in \BZ_{n'} \mid (-q)x = x \}, \qquad
  \BZ_n^{F_0} = \{ x \in \BZ_n \mid q_0x = x\}.
\end{equation*}
The action of $F_0$ on $\BZ_n$ preserves the decomposition
$\BZ_n = \BZ_{n'}\oplus \BZ_{p^a}$. 
Let $x \in \BZ$ be a representative of an element in $\BZ_{n'}^F \subset \BZ_n$.
Then by (2.14.1) we have 
\begin{align*}
q_0x \equiv (-q)x \equiv x \pmod {n'},
\end{align*}
and $x$ gives an element in $\BZ_{n'}^{F_0} \subset \BZ_n^{F_0}$. 
Hence the natural inclusion $\BZ_{n'} \hra \BZ_n$ induces an embedding
$\BZ_{n'}^F = \BZ_{n'}^{F_0} \hra \BZ_n^{F_0}$.
Thus we have an embedding $(Z_G\wg)^F \hra (Z_{G_0}\wg)^{F_0}$ as asserted. 
\end{proof}

\para{2.16.}
Let $\x \in Z_G\wg$ be an $F$-stable character.  By Lemma 2.15,
$\x$ is regarded as an $F_0$-stable character of $Z_{G_0}$. 
Take $(L, C_0, \SE_0) \in \SM_G^F$ pertaining to $\x$.
Since $(\SM_G)_{\x}$ is a single class, we have
$(\SM_G)_{\x} \simeq (\SM_{G_0})_{\x}$.
Here $L$, $C_0$ are both $F$ and $F_0$-stable, and we obtain
$(L, C_0, \SE_0) \in \SM_{G_0}$ pertaining to $\x$.
(Here $L, C_0, \SE_0$ are associated to
$G_0$, but we use the same symbol by abbreviation.) We choose
$u_0 \in C_0^F$ and $v_0 \in C_0^{F_0}$ as in 2.3.
Let $\r_0 \in A_L(u_0)\wg$ be the representation corresponding to $\SE_0$.
Then the pullback of $\r_0$ coincides with $\x$ under the surjection
$Z_G \to A_L(u_0)\wg$, and we have $(L, C_0, \SE_0) \in \SM_G^F$.
Similarly, we have a representation $\r_0' \in A_L(v_0)\wg$ corresponding to
$\SE_0$, and $(L, C_0, \SE_0) \in \SM_{G_0}^{F_0}$. 
We fix $\eta = (L,C_0, \SE_0) \in (\SM_G)^F_{\x} = (\SM_{G_0})^{F_0}_{\x}$.
Note that $F$ and $F_0$ satisfy the first and the second relation in (1.17.2).
\par
We choose $u = u_{\la} \in C^F$ and $u_0' = u'_{\la} \in C_0^F$ in $G$
as in 2.12.  We also choose $v = v_{\la} \in C^{F_0}$ and
$v_0' = v'_{\la} \in C_0^{F_0}$ in $G_0$.  
\par
We consider the variety $\SP_{u_0'}$, regarding $u_0'$ as an element in $G^F$. 
Let $w_0$ be the longest element in $\SW_L \simeq S_m$.  Then $F$ acts on
$H^{2a}_c(\SP_{u_0'}, \dot\SE_0)$ via $\Phi$, and also we have an action of $w_0$ on it.
We have the following result. 

\begin{lem}  
Assume that $F$ is non-split.   Then eigenvalues of
$\Phi w_0$ on $H^{2a}_c(\SP_{u_{\la}'}, \dot\SE_0)$ are equal to $\nu_{\la, \eta}(-q)^a$,
where $\nu_{\la, \eta}$ is as in 2.3.  
\end{lem}  

\begin{proof}
By fixing $\eta \in (\SM_G)_{\x} = (\SM_{G_0})_{\x}$, we put $\nu_{\la} = \nu_{\la, \eta}$.   
Let $C'$ be the unipotent class of $G$ containing $u_0'$ of type
$(d, \dots, d)$. Then $A_G(u_0') \simeq A_L(u_0')$, and
$\r_0$ is regarded as $\r_0 \in A_G(u_0')\wg$. 
We consider the functions $X_i$ and
$Y_i$ for $i \in I_{\eta}^F$ used in 1.17, which we denote by
$X_{i,F}$ and $Y_{i,F}$ in order to indicate the functions on $G^F\uni$.
Put $i_0 = (u_0, \r_0) \in I_{\eta}^F$.
We also consider similar functions on $(G_0)^{F_0}\uni$, which we denote by
$X_{i, F_0}$ and $Y_{i, F_0}$ for $i \in (I_0)_{\eta}^{F_0}$, and put
$i_0 = (v_0, \r_0) \in (I_0)_{\eta}^{F_0}$.  
By (1.4.3), we have

\begin{align*}
  (-1)^{a_0+ r}\sum_a(-1)^{a+ r}&\Tr(\Phi, H^{a+r}_c(\SP_{u'_{\la}}, \dot\SE_0))  \\
  &=  \sum_{E_i \in (\SW_L)\wg\ex}
        \sum_a(-1)^{a + a_0}\Tr(\s_{E_i}, E_i)\Tr(\vf_{E_i}, \SH^a_{u'_{\la}}K_{E_i})  \\
  &=  \sum_{E_i}\Tr(\s_{E_i}, E_i)X_{i,F}(u'_{\la})q^{(a_0 + r)_i/2}  \\
  &=  \sum_{E_i}\Tr(\s_{E_i}, E_i)p_{i_0, i}Y_{i_0,F}(u'_{\la})q^{(a_0+r)_i/2},      
\end{align*}
where $a_0 + r$ is with respect to $C'$, and
$(a_0 + r)_i$ denotes $a_0 + r$ with respect to $i = (C_1,\SE_1)$. 

Since $a_0 + r$ is even, this implies that 

\begin{align*}
 \sum_{a} (-1)^{a+r}\Tr(\Phi w_0, H^{a+r}_c(\SP_{u'_{\la}}, \dot\SE_0))  
     = \sum_{E_i \in \SW_L\wg}\Tr(\s_{E_i}w_0, \wt E_i)p_{i_0,i}Y_{i_0,F}(u'_{\la})q^{(a_0+r)_i/2}.  
\end{align*}

Since $\wt E_i$ is the preferred extension, $\s_{E_i}w_0$ acts
on $E_i$ as a scalar multiplication by $(-1)^{a_{E_i}}$.
By definition (see 2.3), we have $Y_{i_0,F}(u'_{\la}) = \r_0(c_{\la}) = \nu_{\la}$.
Thus we have

\begin{align*}
\tag{2.17.1}  
\sum_{a} (-1)^{a}\Tr(\Phi w_0, H^{a}_c(\SP_{u'_{\la}}, \dot\SE_0))  
     = \nu_{\la}\sum_{E_i \in \SW_L\wg}(-1)^{a_{E_i}}(\dim E_i)p_{i_0,i}q^{(a_0+r)_i/2}.
\end{align*}  

On the other hand, we consider a similar situation for $(G_0, F_0)$.
Let $v'_{\la} \in C'^{F_0}$ be the corresponding element.
By definition in 2.3, we have $Y_{i_0, F_0}(v'_{\la}) = Y_{i_0, F_0}(v_0) = 1$.
We denote by $\Phi_0$ the map $\Phi$ on $H^a_c(\SP_{v_0'}, \dot\SE_0)$ induced
from $F_0$ on $\SP_{v_0'}$.  Then by a similar discussion as above,
we have

\begin{align*}
\tag{2.17.2}  
  \sum_a(-1)^{a}\Tr(\Phi_0, H^{a}_c(\SP_{v'_{\la}}, \dot\SE_0)) 
     = \sum_{E_i \in \SW_L\wg}(\dim E_i)p^0_{i_0,i}q_0^{(a_0 + r)_i/2}. 
\end{align*}  

Now by Lemma 1.18 (ii),
$p_{i_0, i} = (-1)^{\d_{i_0} + \d_i}p_{i_0,i}^0(-q) = (-1)^{\d_i}p_{i_0,i}^0(-q)$
since $\d_{i_0} = 1$.
Moreover, since $(a_0 + r)_i/2 = d_{u_1}$ for $u_1 \in C_1$, we have
$(-1)^{(a_0+r)_i/2}\cdot (-1)^{a_{E_i}} = (-1)^{\d_i}$. 
By comparing (2.17.1) and (2.17.2), we have

\begin{align*}
\tag{2.17.3}  
  \sum_{a} (-1)^{a}&\Tr(\Phi w_0, H^{a}_c(\SP_{u'_{\la}}, \dot\SE_0))   \\  
     &= \nu_{\la}\bigg[\sum_a(-1)^{a}\Tr(\Phi_0, H^{a}_c(\SP_{v'_{\la}}, \dot\SE_0))\biggr](-q).
\end{align*}

By (2.10.1), $H^a_c(\SP_{u_0'}, \dot\SE_0) = H^a_c(\SP_{v_0'}, \dot \SE_0) = 0$
if $a$ is odd.
By Lemma 2.13, the eigenvalues of $\Phi_0$ on $H^{2a}_c(\SP_{v_0'}, \dot\SE_0)$
are equal to $q_0^a$. Then (2.17.3) implies that

\begin{align*}
\tag{2.17.4}  
 \sum_{a}\Tr(\Phi w_0, H^{2a}_c(\SP_{u'_{\la}},\dot\SE_0))
  = \nu_{\la}\sum_{a}\dim H^{2a}_c(\SP_{v'_{\la}}, \dot\SE_0)(-q)^a. 
\end{align*}  

Note that $\dim H^{2a}_c(\SP_{u'_{\la}}, \dot\SE_0) = \dim H^{2a}_c(\SP_{v'_{\la}},\dot\SE_0)$
since $\SP_{u'_{\la}} \simeq \SP_{v'_{\la}}$, and $\dot\SE_0$ on
$\SP_{u'_{\la}}$ and on $\SP_{v'_{\la}}$
are both obtained from $\x \in (Z_G\wg)^F \hra (Z_{G_0}\wg)^{F_0}$, hence equal. 
Since the absolute value of the eigenvalues
of $\Phi$ on $H^{2a}_c(\SP_{u'_{\la}}, \dot\SE_0)$ is equal to $q^a$ (see 2.12),
the same is true also for the eigenvalues of $\Phi w_0$ on it.  
The lemma now follows from (2.17.4).
\end{proof}

We are now ready to prove the following result
which is
a generalization of Theorem 4.4 in \cite[I]{Sh4}. Theorem 4.4 was proved under the assumption
that $p$ is large enough, here no restriction on $p$.

\begin{prop}  
Let $C = C_{\la}$.  Take $u = u_{\la} \in C^F$ and $v = v_{\la} \in C^{F_0}$
the split elements defined in 2.3.
Let $a_0 + r$ be as in (1.4.2) with respect to $C$.  
\begin{enumerate}
\item \  $\Phi w_0$ acts on $H^{a_0+r}_c(\SP_{u}, \dot\SE_0)$ as a scalar
multiplication $\nu_{\la}(-q)^{(a_0+r)/2}$.  
\item \ $\Phi_0$ acts on $H^{a_0+r}_c(\SP_{v}, \dot\SE_0)$ as a scalar
multiplication $q_0^{(a_0+r)/2}$.
\end{enumerate}
\end{prop}  

\begin{proof}
First we show (i).  Let
$u'_0 \in C_0^F$ be the element as in 2.12.  We consider the
varieties $\SP_u$ and $\SP_{u_0'}$ as in 2.8.
Then we have a closed immersion $i : \SP_u \hra \SP_{u'_0}$.  Here $\SP_u, \SP_{u'_0}$
are $F$-stable, and $i$ is $F$-equivariant. 
By Proposition 2.11, $i^* : H^{2a}_c(\SP_{u'_0}, \dot\SE_0) \to H^{2a}_c(\SP_u, \dot\SE_0)$
is surjective, and $\SW_L$-equivariant.
Hence $i^*$ is $\Phi w_0$-equivariant. 
By Lemma 2.17, the eigenvalues of $\Phi w_0$ on $H^{2a}_c(\SP_{u'_0}, \dot\SE_0)$
are $\nu_{\la}(-q)^a$. 
\par
By (2.5.3) $\dim \SP_u = (a_0 + r)/2$.  It is known by the generalized Springer correspondence
that
\begin{equation*}
\tag{2.18.1}  
\SH^{a_0}_uK = H^{a_0+r}_c(\SP_u, \dot\SE_0) \simeq V_{(u,\r)}\otimes \r
\end{equation*}
as $\SW_L \times A_G(u)$-modules, where $V_{(u,\r)}$ is the irreducible
representation of $\SW_L \simeq S_m$ of type $\mu = \la/d$. 
Since the action of $\Phi w_0$ on $H^{a_0 + r}_c(\SP_u,\dot\SE_0)$ commutes
with the action of $\SW_L$, $\Phi w_0$ acts on $H^{a_0+r}_c(\SP_u,\dot\SE_0)$ as
a scalar multiplication.
Thus $\Phi w_0$ acts on $H^{a_0 + r}_c(\SP_u, \dot\SE_0)$ as a scalar multiplication
$\nu_{\la}(-q)^{(a_0+r)/2}$.  This proves (i).
\par
(ii) is proved similarly by using Lemma 2.13 instead of Lemma 2.17.
\end{proof}

As a corollary to Proposition 2.18, we have the following result, which is a generalization
of \cite[I, Thm. 4.4]{Sh4}. 

\begin{thm}  
Let $C = C_{\la}$, and assume that $i = (C,\SE) \in I_{\eta}^F = (I_0)_{\eta}^{F_0}$
belongs to the series $\eta = (L, C_0, \SE_0) \in (\SM_G)_{\x}^F = (\SM_{G_0})_{\x}^{F_0}$.
\begin{enumerate}
\item  \  Let $u = u_{\la} \in C^F$, and $u_0 \in C_0^F$ be as in 2.3.
Let $\wt\r_0$ be the trivial extension of $\r_0 \in A_L(u_0)\wg$.
For each $(u,\r) \lra (C,\SE)$, let $\wt\r$ be the trivial extension of $\r \in A_G(u)\wg$.
Then we have
\begin{equation*}
\g(u_0, \wt\r_0, u, \wt\r) = \nu_{\la,\eta}(-1)^{\d_i}.
\end{equation*}
\item \ Let $v = v_{\la} \in C^{F_0}$, and $v_0 \in C_0^{F_0}$ be
defined for $F_0$, similarly to $u, u_0$ in (i).    
Let $\wt\r_0, \wt\r$ be extensions of $\r_0 \in A_L(v_0)\wg, \r \in A_G(v)\wg$,
defined similarly to (i).  Then we have
\begin{equation*}
\g(v_0, \wt\r_0, v, \wt\r) = 1.
\end{equation*}  
\end{enumerate}  

In particular, $u = u_{\la} \in C^F$, $v = v_{\la} \in C^{F_0}$ are
split elements in the sense of 1.7.
\end{thm}  

\begin{proof}
First we show (i). Take $i = (C,\SE) \in I^F$. Let
$E = V_{(u,\r)} \in \SW_L\wg$ for $(u,\r) \lra (C,\SE)$. 
Then under the isomorphism (2.18.1), 
the restriction of $\vf : F^*K \isom K$ on $E \otimes K_E$ gives
$\s_{E} \otimes \vf_{E}$, and $\vf_E|_C$ is given by $q^{(a_0+r)/2}\psi$
for $\psi : F^*\SE \isom \SE$.  The function $Y_i$
is defined by $\psi : F^*\SE \isom \SE$. 
It follows that

\begin{align*}
\tag{2.19.1}  
\Tr (\Phi w_0, H^{a_0+r}_c(\SP_u, \dot\SE_0))  &= \Tr(\vf w_0, \SH^{a_0}_uK)  \\
                         &= q^{(a_0+r)/2}\Tr(\s_Ew_0, E)\Tr(\psi, \SE_u)  \\
                         &=  q^{(a_0+r)/2}(-1)^{a_E}(\dim E) Y_i(u). 
\end{align*}

On the other hand, by Proposition 2.18 (i), we have

\begin{align*}
\tag{2.19.2}  
\Tr(\Phi w_0, H^{a_0+r}_c(\SP_u, \dot\SE_0))
      = \nu_{\la,\eta}q^{(a_0+r)/2}(-1)^{(a_0+r)/2}\dim E.  
\end{align*}
Since $(a_0 + r)/2 = d_u$, and $\d_i = a_E - d_u$,  we have 
\begin{equation*}
\tag{2.19.3}  
Y_i(u) = \nu_{\la,\eta}(-1)^{\d_i}.
\end{equation*}  

Let $\wt\r$ be the trivial extension of $\r \in A_G(u)\wg$.  Then
\begin{equation*}
  Y_i(u) = \g(u_0, \wt\r_0, u, \wt\r)Y_i^0(u) = \g(u_0, \wt\r_0, u, \wt\r).
\end{equation*}
By comparing this with (2.19.3), we obtain (i).  (ii) is proved in a similar way as in (i)
by using Proposition 2.18 (ii).
\end{proof}  

\para{2.20.}
In the setup in 2.14, we consider a special case where $p = p_0$.
Let $F : G \to G$ be a non-split Frobenius map as in 2.14.
Assume that the order of $p$ in the multiplicative group $(\BZ/n'\BZ)^*$
is even.  Then there exists $y \in \BZ$ such that $p^y \equiv -1 \pmod {n'}$.
(Note that this gives a restriction on $p$.  If $p \equiv 1 \pmod{n'}$,
such a $p$ does not exist.)
We define $q_0$ by $q_0 = qp^y$, a power of $p$,
and let $F_0 : G \to G$ be a split Frobenius map with respect to the $\BF_{q_0}$-structure.
By a similar computation as in the proof of Lemma 2.15, we have, for $x \in \BZ_{n'}^F$,
\begin{equation*}
\tag{2.20.1}  
q_0x = qp^y x \equiv (-q)x \equiv x \pmod{n'}.
\end{equation*}  
It follows that $\BZ_{n'}^F$ coincides with $\BZ_{n'}^{F_0}$, and we have
$(Z_G\wg)^F \simeq (Z_G\wg)^{F_0}$.  
The previous discussion, such as Proposition 2.18, Theorem 2.19 holds in this case
with $G = G_0$. 
\par
We compare $(G, F)$ and $(G, F_0)$.  Let $C$ be a unipotent class in $G$.  Then
$C$ is $F$ and $F_0$-stable, and let $u \in C^F$ (resp. $v \in C^{F_0}$) be a split
element in $G^F$ (resp. in $G^{F_0}$).
It is easy to check that
\begin{equation*}
\tag{2.20.2}  
|Z_G(u)^F| = (-1)^{\dim Z_G(u)}|Z_G(v)^{F_0}|(-q).
\end{equation*}  
Also we note that
\par\medskip\noindent
(2.20.3) \ $Z_G(u) \simeq \BZ_{n'_{\la}}$, and the natural surjection
$\BZ_{n'} \to \BZ_{n'_{\la}}$ gives an isomorphism $A_G(u) \isom A_G(v)$.
This isomorphism is compatible with the action of $F$ on $A_G(u)$, and that of
$F_0$ on $A_G(v)$. In particular, we have an isomorphism $\wt A_G(u) \simeq \wt A_G(v)$.
\par\medskip
By (2.20.3), one can define a bijection $\xi$ from the set of
$G^F$-conjugacy classes in $G^F\uni$ to the set of $G^{F_0}$-conjugacy classes in
$G^{F_0}\uni$ by the condition that $v = \xi(u)$ is split if $u$ is split.
\par
Let $\eta = (L, C_0, \SE_0) \in (\SM_G)^F_{\x} = (\SM_G)^{F_0}_{\x}$. 
Recall (see 1.4) that, for $w \in \SW_L$, the generalized Green functions on $G^F\uni$
associated to $\vf_0 : F^*\SE_0 \isom \SE_0$ are
given by
\begin{align*}
\tag{2.20.4}
  Q^G_{L^w,C^w_0,\SE^w_0,\vf^{w}_0}(g)
   = \sum_{E_i \in (\SW_L)\wg\ex}\Tr(w\s_{E_i}, \wt E_i)X_i(g).
\end{align*}

We also consider the generalized Green functions on $G^{F_0}\uni$
associated to $\vf'_0: F_0^*\SE_0 \isom \SE_0$, which we denote by
$Q^0_{L^w,C^w_0,\SE^w_0,\vf'^{w}_0}$. 
Here we assume that $\vf_0: F^*\SE_0 \isom \SE_0$ (resp. $\vf_0': F_0^*\SE_0 \isom \SE_0$)
is defined by
the trivial extension $\wt\r_0$ of $\r_0 \in A_L(u_0)\wg$
(resp. $\wt\r'_0$ of $\r'_0 \in A_L(v_0)\wg$), where
$u_0 \in C_0^F$, $v_0 \in C_0^{F_0}$ are split elements,
and we identify them by the isomorphism $\wt A_L(u_0) \simeq \wt A_L(v_0)$.  
\par
For any $g \in C_{\la}$, we put $\nu_{g, \eta} = \nu_{\la,\eta}$, where
$\nu_{\la,\eta}$ is defined as in 2.3. 
\par
In the case of $E_6$, the formula connecting Green functions
for $(G,F)$ and $(G, F_0)$ was proved in \cite{BS}.   
The following is a generalization of their result.

\begin{thm}  
Assume that $G = SL_n$.  Let $F$ and $F_0$ be as in 2.20.
\begin{enumerate}
\item \    
There exists a bijection $\xi$ from the set of $G^F$-conjugacy classes of $G^F\uni$
to the set of $G^{F_0}$-conjugacy classes of $G^{F_0}\uni$ satisfying the
equalities
\begin{equation*}
  Q^G_{L^w, C_0^w, \SE_0^w, \vf_0^w}(g)
      = \nu_{g, \eta} Q^0_{L^{w_0w}, C_0^{w_0w}, \SE_0^{w_0w}, \vf'^{w_0w}_0}(\xi(g))(-q)
\end{equation*}
for any  $(L,C_0, \SE_0) \in (\SM_G)^F_{\x} = (\SM_G)^{F_0}_{\x}$. 
\par\smallskip
\item \  The bijection $\xi$ is characterized by the properties
\par\medskip
\begin{enumerate}
\item \ $|Z_G(g)^F| = (-1)^{\dim Z_G(g)}|Z_G(\xi(g))^{F_0}|(-q)$.
\\
\item \ 
  $\Tr(\Phi w_0, H^{a_0+r}_c(\SP_g, \dot\SE_0))
    = \nu_{g,\eta} (-1)^{(a_0+r)/2}\Tr(\Phi_0, H^{a_0 +r}_c(\SP_{\xi(g)}, \dot\SE_0))$.    
\end{enumerate}
\end{enumerate}  
\end{thm}  

\begin{proof}
We prove (i).  By Lemma 1.18, (ii) and Theorem 2.19, we have

\begin{align*}
X_{i,F}(g) &= \sum_{j \in I^F}p_{ji}Y_{j,F}(g) \\
       &= \sum_{j \in I^F}(-1)^{\d_i + \d_j}p_{ji}^0(-q)\nu_{g,\eta}
                             (-1)^{\d_j}Y_{j,F_0}(\xi(g))  \\ 
           &= \nu_{g,\eta} (-1)^{\d_i} X_{i,F_0}(\xi(g))(-q).
\end{align*}  
Then by (1.4.4), we have

\begin{align*}
Q^G_{L^w, C_0^w, \SE_0^w, \vf_0^w}(g)
     &= \sum_{i \in I^F}(-1)^{a_{E_i}}\Tr(ww_0, E_i)X_{i,F}(g)(-1)^{(a_0)_i}q^{(a_0+r)_i/2}  \\
&= \sum_{i \in I^F}(-1)^{d_g}\Tr(ww_0, E_i)\nu_{g,\eta} X_{i,F_0}(\xi(g))(-q)
              (-1)^{(a_0)_i}q^{(a_0+r)_i/2} \\
     &= \nu_{g,\eta} Q^0_{L^{ww_0}, C_0^{ww_0}, \SE_0^{ww_0}, \vf'^{ww_0}_0}(\xi(g))(-q)
\end{align*}  
since $d_g = (a_0+r)_i/2$ for $g \in C$ with $i = (C,\SE)$. Hence (i) holds. 

\par
Next we show (ii). Let $\xi$ be the bijection defined in 2.20.  For each class $C$, 
if $g = u$ is split in $C^F$, and $\xi(u) = v$ is split in $C^{F_0}$, then
(a) holds, and (b) was proved by Proposition 2.18.
Since $\wt A_G(u) \simeq \wt A_G(v)$, these properties also hold for $g \in C^F$ and
$\xi(g) \in C^{F_0}$. Hence the equality holds for our choice of $\xi$ given in 2.20. 

\par
Now assume that there exists a bijection $\xi_1$ satisfying the properties (a) and (b).
Take a split element $v \in C^{F_0}$, and let $u_1 \in C^F$ be
such that $\xi_1(u_1) = v$.
Then by (a), we have $\wt A_G(u_1) \simeq \wt A_G(v)$. 
By a similar discussion as in the proof of
Theorem 2.19, the equality (b) implies that
$\g(u_0, \wt\r_0, u_1, \wt\r_1) = \nu_{u,\eta} (-1)^{\d_i}$ for
$i = (C,\SE) \in I^F$. This implies that
$\g(u_0, \wt\r_0, u, \wt\r) = \g(u_0, \wt\r_0, u_1, \wt\r_1)$ for any
$\wt\r \simeq \wt\r_1$ under the identification $\wt A_G(u) \simeq \wt A_G(u_1)$.
Hence by Lemma 1.9, $u$ and $u_1$ are conjugate under $G^F$.  Thus $\xi$ is unique,
and (ii) is proved. 
\end{proof}  

\remark{2.22.}
A root of unity $\nu_{u, \eta}$ for $u \in C^F$ and
$\eta = (L, C_0, \SE_0) \in \SM^F_G$ is defined by $\nu_{u,\eta} = \r_0(c_{\la})$
in 2.3.  $c_{\la} \in A_L(u_0)$ was computed explicitly in \cite[I, Lemma 4.17]{Sh4}.
Thus $\nu_{u,\eta}$ can be computed explicitly. 

\par\bigskip
\section{ Classical groups}

\para{3.1.}
In this section, we assume that $G$ is a classical group, not of type $A_n$, and
$F$ is a Frobenius map on $G$, split or non-split, such that $F^2$ is split.
(The case where $G$ is of type $D_4$, such that $F$ is a non-split Frobenius map
with $F^3$ split, is discussed in Section 10.) 
First assume that $G = Sp_{2n}$ or $SO_N$.  The existence of split elements
was discussed in \cite[II]{Sh4}. Here we review those results.
Let $C$ be an $F$-stable unipotent class
in $G$. For each $u \in C^F$, $A_G(u)$ is abelian, and $F$ acts trivially on $A_G(u)$.
The ``split element'' $u \in C^F$ was defined in \cite[II, \S 2]{Sh4} in an explicit form.
Note that in the case where $G = SO_{2n}$ with $F$ non-split, $u$ is chosen 
in connection with preferred extensions of $\SW_L$-modules (see \cite[II, 2.10, 2.11]{Sh4}).
The following result was proved in Theorem 4.2 and Theorem 4.3 in \cite[II]{Sh4}.

\begin{thm}  
Let $G = Sp_{2n}$ or $SO_N$.  Assume that $(C, \SE) \in \SN_G^F$ belongs to
to the series $(L,C_0, \SE_0) \in \SM_G^F$.  Choose $(u_0, \r_0) \lra (C_0, \SE_0)$
with $u_0 \in C_0^F$ split.
Then for $(u,\r) \lra (C, \SE)$ with $u \in C^F$ split, we have
$\g(u_0, \r_0, u, \r) = 1$.  In particular, $u$ is a split element in the sense of
1.7. 
\end{thm}

\para{3.3.}
In the rest of this section, we assume that $G$ is a spin group $\Spin_N$ with $p \ne 2$. 
The existence of split elements in $G$ was discussed in \cite[III]{Sh4}.  However it
contains several errors, and in this occasion we shall correct them. 
The main points to be corrected are as follows. 
\par\medskip
(1) \ The classification of quadratic forms in \cite[III, 2.2]{Sh4} 
is not complete, and some cases are missing.  
Note that as stated in Remark 2.8 in \cite[III]{Sh4}, 
the case where $F$ is non-split was excluded from  the discussion there. 
However we need to consider the non-split case also, which comes from the above
missing case. 
\par\medskip
(2) \ The definition of split elements in the case of spin groups (for $F$ split) 
given in \cite[III, 3.2]{Sh4} is wrong. Although it was defined up to $GL_N$-conjugate, 
this should be corrected so that it is defined up to $G^F$-conjugate.  
Actually, the proof of Lemma 3.10 in \cite[III]{Sh4} contains
some errors.  Also, we need to define split elements in the non-split $F$ case, 
based on the above correction (1).
\par\medskip
In the discussion below, we attach ${}^*$ to indicate the corrected statement
of the original one, for example, Lemma $3.10^*$ is the corrected version of
Lemma 3.10 in \cite[III]{Sh4}.

\para{3.4.}
First we consider the problem (1).
We follow the notation in \cite[III, 2.1]{Sh4}.
Let $V$ be a vector space of dimension $N \ge 3$ over $\Bk$ with $\ch\Bk \ne 2$,
and $(\ ,\ )$ a non-degenerate symmetric bilinear form on $V$. Let $C(V)$ be the
corresponding Clifford algebra with an embedding $V \subset C(V)$, and $C^+(V)$
the subalgebra of $C(V)$ spanned by products of even number of elements in $V$.
The Spin group $\Spin(V) \simeq \Spin_N$ is defined as the subgroup of the group of
units in $C^+(V)$ consisting of all products $v_1v_2\cdots v_a$ (with $a$ even)
where $v_i \in V$ satisfy $(v_i,v_i) = 1$. 
We have a natural isogeny $\b : \Spin(V) \to SO(V)$.
Let $G = \Spin(V)$, and $Z_G$ the center of $G$.
If $N$ is odd, then $Z_G$ has order 2; it is generated by $\ve = (-1)$ times the
unit element of $C(V)$.  If $N$ is even, $Z_G$ has order 4;
it is generated by $\ve$ and by $\w = v_1v_2\cdots v_N$, where $v_1, \dots, v_N$ is an
orthonormal basis of $V$.  We have $\w^2 = \ve^{N/2}$, hence $Z_G$ is generated by
$\w$ if $N \equiv 2 \pmod 4$, and $Z_G \simeq \lp \w \rp \times \lp \ve \rp$ is
a product of two cyclic groups of order 2 if $N \equiv 0 \pmod 4$.
In any case, $\Ker \b = \{ 1, \ve\}$. 
\par
We consider the $\BF_q$-structure of $G$.  Assume that $V$ is defined over $\BF_q$
with Frobenius map $F : V \to V$, and that the form $(\ ,\ )$ is compatible with
$F$-action, then $C(V)$ has a natural $\BF_q$-structure, which induces an $\BF_q$-structure
on $G$. 
If $N$ is odd, $G$ has a unique $\BF_q$-structure where $F$ is split,
while if $N$ is even, it has two $\BF_q$-structures where $F$ is split or non-split. 
Assume that $N$ is even.  In the discussion below, we use the theory of quadratic forms 
over a finite field $\BF_q$ with $q$ : odd. See \cite{Gr} for details.    
In 2.2 in \cite[III]{Sh4}, 
the quadratic form on $V$ is defined only in the case where $N \equiv 0 \pmod 4$. 
We also need to consider the case where $N \equiv 2 \pmod 4$.  
In the full generality, they are given as follows.
Assume that $N$ is even. Let $e_1, \dots e_N$ be the orthogonal basis of $V$
such that $F(e_i) = e_i$. For $x = \sum_ix_ie_i$, we define the quadratic 
form on $V$ associated to $(\ ,\ )$ by the following formulas.  
Choose $\d \in \BF_q - \BF_q^2$.
\par\medskip\noindent
(3.4.1) \
(i) \ Assume that $N \equiv 0 \pmod 4$.  Then put 
\begin{equation*}
(x,x) = \begin{cases}
          x_1^2 + \cdots + x_N^2 &\quad\text{ split case, } \\
          x_1^2 + \cdots + x_{N-1}^2 + \d x_N^2  &\quad\text{ non-split case}
        \end{cases}
\end{equation*} 

\par
(ii) \ Assume that $N \equiv 2 \pmod 4$.  If $q\equiv 1 \pmod 4$, then put 
\begin{equation*}
(x,x) = \begin{cases}
          x_1^2 + \cdots + x_N^2 &\quad\text{ split case, } \\
          x_1^2 + \cdots + x_{N-1}^2 + \d x_N^2  &\quad\text{ non-split case,}
        \end{cases}
\end{equation*}
\hspace*{1.5cm} while if $q\equiv -1 \pmod 4$, we exchange the condition for splitness as 
\begin{equation*}
(x,x) = \begin{cases}
          x_1^2 + \cdots + x_{N-1}^2 + \d x_N^2  &\quad\text{ split case,} \\
          x_1^2 + \cdots + x_N^2 &\quad\text{ non-split case. } \\
        \end{cases}
\end{equation*}
The complexity in the case $N \equiv 2 \pmod 4$ comes from  the following fact.
The splitness of the $\BF_q$-structure of $SO(V)$ is defined by 
using the hyperbolic basis of $V$, for $N = 2n$, as

\begin{equation*}
\tag{3.4.2}
(x,x) = \begin{cases}
            x_1x_{n+1} + \cdots + x_nx_{2n}   &\quad\text{ split case, } \\
            x_1x_{n+1} + \cdots + x_{n-1}x_{2n-1} + x^2_n -\d x_{2n}^2 
                 &\quad\text{ non-split case,}
        \end{cases}
\end{equation*}
see, for example, \cite[II, 2.1]{Sh4}, where the case $p = 2$ is discussed. 
In studying $SO_N$, the expression in (3.4.2) is convenient, 
but in studying $\Spin_N$, 
the expression used in (3.4.1) is more convenient. 
Thus we need to change (3.4.2) to the form in (3.4.1). 
Note that the form $x_nx_{2n}$ is equivalent to $x_n^2 - x_{2n}^2$. 
We have $-1 \in \BF_q^2$ if and only if $q \equiv 1 \pmod 4$. 
It follows that
\begin{align*}
\tag{3.4.3}
x_n^2 - x_{2n}^2 &\simeq\begin{cases}
                     x_n^2 + x_{2n}^2 &\quad\text{ if } q \equiv 1 \pmod 4, \\
                     x_n^2 + \d x_{2n}^2 &\quad\text{ if } q\equiv -1 \pmod 4,
                \end{cases}  \\ 
x_n^2 - \d x_{2n}^2 &\simeq \begin{cases}
                     x_n^2 + \d x_{2n}^2 &\quad\text{ if } q\equiv 1 \pmod 4, \\
                     x_n^2 + x_{2n}^2    &\quad\text{ if } q\equiv -1 \pmod 4.
                    \end{cases}  
\end{align*}
(3.4.1) follows from (3.4.3).

\par
The discussion in \cite[III, 2.2]{Sh4} should be modified as follows. 
In the case where $N \equiv 0 \pmod 4$, $F(\w) = -\w$ if $(G,F)$ is non-split.
While in the case where $N \equiv 2 \pmod 4$, then 
$F(\w) = -\w$ if, either $q \equiv 1 \pmod 4$ and $(G,F)$ is non-split, or 
$q\equiv -1 \pmod 4$ and $(G,F)$ is split.  Otherwise $F(\w) = \w$.   
Thus (2.2.1) in \cite[III]{Sh4} should be replaced by the following statement. 
\par\medskip\noindent
(2.2.1)${}^*$ \ Let $Z_G$ be the center of $G$.
\begin{enumerate}
\item \ Assume that $N$ is odd, then $F$ acts trivially on $Z_G$.    
\item \ Assume that $N$ is even. If $N \equiv 0\pmod 4$, then 
$F$ acts trivially on $Z_G$ if $(G,F)$ is split.  If $N \equiv 2 \pmod 4$, 
then $F$ acts trivially on $Z_G$ if, either  $(G,F)$ is split and $q\equiv 1 \pmod 4$, 
or $(G,F)$ is non-split and $q \equiv -1 \pmod 4$.   
Otherwise, $F(\w) = -\w$, and $F$ acts non-trivially on $Z_G$. 
\end{enumerate}

\para{3.5.}
We write a partition $\la = (\la_1, \dots, \la_k)$ of $N$
as $\la = (1^{m_1}, 2^{m_2}, \dots)$, where $m_i$ is the number of $j$
such that $i = \la_j$. We define a set $X_N$ by
\begin{equation*}
X_N = \{ \la \mid \text{ $m_i$ : even if $i$ : even,} \text{ $m_i \le 1$ if $i$: odd}\}.
\end{equation*}
The generalized Springer correspondence pertaining to $\x \in Z_G\wg$ 
(see 1.2.1) is essentially the same as the case of
$SO_N$ if $\x(\ve) = 1$, and is described by \cite[Thm. 4.14]{LS2} as a bijection 
$\la : \SN_{\x} \simeq X_N$ (see also \cite[III, Prop. 2.5]{Sh4}) if $\x(\ve) = -1$.  
Remark 2.7 in \cite[III]{Sh4} should be modified as follows.
In the 4th line of Remark 2.7, it is written as 
``By Proposition 2.5, the Jordan type $\la$ of $\b(C)$ is not contained in $X_N$".
This is wrong, and actually we need to assume this. 
Namely, this part should be 
modified as follows (this is independent from (1), (2) in 3.3). 
\par\medskip
``Assume that the Jordan type $\la$ of $\b(C)$ is not contained in $X_N$.
Since $\b$ gives a bijection $C^F \simeq \b(C)^F$, one can find $u \in C^F$ 
such that $\b(u) = v^{\bullet}$. In this case, it is known by \cite[14.3]{L1} 
and by Proposition 2.5 that $A_G(u) \simeq A_{\ol G}(v^{\bullet})$.
Thus the verification of the theorem is reduced to the case where 
$\x(\ve) = -1$, and to the determination of the split class in $C^F$ 
for $C$ of Jordan type $\la \in X_N$''. 

\para{3.6.}
Remark 2.8 in \cite[III]{Sh4} should be modified as follows, 
according to the correction in 3.4.
\par\medskip\noindent
{\bf Remark $2.8^*$.} \ 
Assume that $(G, F)$ is non-split. Then $N$ is even, and 
$Z_G$ is generated by $\ve$ and $\w$. We consider the set $\SN_{\x}$ in 
Proposition 2.5.  Then $\x(\ve) = -1$. By $(2.2.1)^*$, we have the following. 
Assume that $N \equiv 0 \pmod 4$, then $F(\w) = \w$. 
Assume that $N \equiv 2 \pmod 4$.  Then $F(\w) = -\w$ if $q\equiv 1 \pmod 4$, 
and $F(\w) = \w$ if $q \equiv -1 \pmod 4$. 
If $F(\w) = -\w$, then 
\begin{equation*}
\x(F(\w)) = \x(-\w) = \x(\ve)\x(\w) = -\x(\w).
\end{equation*}
This shows that $\x$ is not $F$-stable, and so $\SN_{\x}^F = \emptyset$. 
Thus in the non-split case, we may assume that $N \equiv 2 \pmod 4$ and that 
$q \equiv -1 \pmod 4$. In this case, $F$ acts trivially on $Z_G$.  

\para{3.7.}
Lemma 3.5 in \cite[III]{Sh4} should be modified as follows 
(this is independent from the correction in 3.4.) 
\par\medskip\noindent
{\bf Lemma $3.5^*$.} 
{\it If $q \equiv 1 \pmod 4$, then $F$ acts trivially on $A_G(u)$.  While 
if $q \equiv -1 \pmod 4$, the action of $F$ on $A_G(u)$ is 
not necessarily trivial.} 
\par\medskip

In fact, the case where $q \equiv 1 \pmod 4$ is discussed in \cite[III]{Sh4}.
We consider the case where $q\equiv -1 \pmod 4$.  
Assume that $(G, F)$ is of split type, and $N \equiv 0 \pmod 4$. 
We follow the notation in \cite[III, 3.1]{Sh4}. 
We write $\la_j = 2m_j +1$.  
In the case where $|I| = 2$, we have $m_1 + m_2$ is odd.  
Then $F(x_1) = (-1)^{m_1 + 1+ 1}x_1$, and $F(x_2) = (-1)^{m_2 + 1 +2}x_2$. 
It follows that $F(x_1x_2) = x_1x_2$, $F$ acts trivially on $A_G(u)$.
But if $|I| = 3$, the condition is that $m_1 + m_2 + m_3$ is odd. 
In this case, $F(x_1) = (-1)^{m_1 + 1 + 1}x_1$ and $F(x_2) = (-1)^{m_2 + 1 + 2}x_2$. 
Since we can choose $m_1, m_2$ arbitrary, it occurs that 
$F(x_1x_2) = -x_1x_2$.  Hence $F$ acts on $A_G(u)$ non-trivially. 

\para{3.8.}
We now consider the problem (2) in 3.3.
We apply the restriction formula to the parabolic subgroup $Q = MU_Q$, and
a unipotent class $C'$ of $M$.  
Assume that $N \ge 7$.
Following \cite[III, 3.6]{Sh4}, we consider the case where $M$ is a Levi subgroup
of $Q$ such that $\b(M) \simeq SO_{N-4}\times GL_2$.
Take $(u,\r) \lra (C,\SE) \in (\SN_G)_{\x}, (u',\r') \lra (C',\SE') \in (\SN_M )_{\x}$.
If $X_{u,u'} \ne \emptyset$, then $u'$ satisfies the properties,

\begin{enumerate}
\item \ The projection of $\b(u')$ on $GL_2$ is regular.
\item \ The projection of $\b(u')$ on $SO_{N-4}$ corresponds to the unipotent elements
of Jordan type in $X_{N-4}$ under the correspondence
  $\la : (\SN_{\Spin_{N-4}})_{\x} \simeq X_{N-4}$.  
\end{enumerate}  
We write $\b(G) = \ol G$, and $\b(u) = \ol u$. 
Also put $\b(Q) = \ol Q, \b(M) = \ol M$, and put $\la = \la(u), \la' = \la(u')$.
Thus $\la = (\la_1 \le \la_2 \le \cdots\le \la_k)$ is a partition of $N$,
and $\la' = (\la'_1 \le \la'_2 \le \cdots \le \la'_{k'})$ is a partition of $N-4$.
It is known, by Lemma 4.5 and Lemma 4.8 in \cite{LS2}, 
that the Young diagram of $\la'$ is obtained from 
the Young diagram of $\la$ by removing 4 squares by the following rules.

\par\medskip
(I)  $\la_i$ is odd, $\la_i > \la_{i-1} + 4$, and 
\begin{equation*}
\la_j' = \begin{cases}
             \la_j - 4  &\quad\text{ if $j = i$, } \\
              \la_j     &\quad\text{ otherwise. }
         \end{cases}
\end{equation*}

(II) $\la_i = \la_{i+1} \ge \la_{i-1} + 2$ and 
\begin{equation*}
\la_j' = \begin{cases}
            \la_j-2 &\quad\text{ if $j = i, i+1$, }  \\
            \la_j   &\quad\text{ otherwise.} 
         \end{cases}
\end{equation*}

(III) $\la_i = \la_{i+1} \ge \la_{i-1}+ 4$ and
\begin{equation*}
\la_j' = \begin{cases}
            \la_j - 3  &\quad\text{ if $j = i$, }  \\
            \la_j - 1  &\quad\text{ if $j = i+1$, } \\
            \la_j      &\quad\text{ otherwise. }
         \end{cases}
\end{equation*}

(IV) $\la_{i+1} - 2 = \la_i \ge \la_{i-1} + 1$ and 
\begin{equation*}
\la'_j = \begin{cases}
           \la_j - 1 &\quad\text{ if $j = i$, }  \\
           \la_j - 3 &\quad\text{ if $j = i+1$, } \\
           \la_j     &\quad\text{ otherwise. }  
         \end{cases}
\end{equation*}

(V) $\la_{i+2} = \la_{i+1} = \la_i + 1$ and
\begin{equation*}
\la'_j = \begin{cases}
            \la_j - 1 &\quad\text{ if $j = i, i+2$, } \\
            \la_j - 2 &\quad\text{ if $j = i+1$, }  \\
            \la_j     &\quad\text{ otherwise. } 
         \end{cases} 
\end{equation*}

We have the following lemma.

\begin{lem}  
Let $\ol Y_{\ol u, \ol u'}$ be the variety defined  in (1.12.1) with respect to
$\ol G, \ol u = \b(u)$ and $\ol u' = \b(u')$. 
Assume that there exists $\ol y \in \ol Y_{\ol u,\ol u'}^F$, and
take $y \in \b\iv(\ol y)$. 
Then $y \in Y_{u,u'}$.  
If $y \notin Y_{u,u'}^F$, put $u_1 = huh\iv  \in C^F$ for 
$h \in G$ such that $h\iv f(h) = \ve$.
Then we have  $hy \in Y_{u_1, u'}^F$, and $\b(u_1)$ is $\ol G^F$-conjugate to $\ol u$.
Furthermore, the isomorphism $Z_G(u) \isom Z_G(u_1), x \mapsto hxh\iv$ induces
an $F$-equivariant isomoprhism $A_G(u) \simeq A_G(u_1)$. 
\end{lem}

\begin{proof}
Assume that $u \in X_N$.  Then $\ve \notin Z_G^0(u)$ by \cite[14.3]{L1}, and so 
the image of $\ve$ in $A_G(u)$ is not equal to 1. Take $h \in G$ such that 
$h\iv F(h) = \ve$, and put $u_1 = huh\iv$. 
Then $u_1$ is not necessarily $G^F$-conjugate to $u$.  
But since $\ol h\iv F(\ol h) = 1$, we have $\ol h = \b(h) \in \ol G^F$.  
Hence $\ol u$ and $\b(u_1)$ are $\ol G^F$-conjugate. 
Under the isomorphism $A_G(u) \isom A_G(u_1)$, the action of $F$ on $A_G(u_1)$
corresponds to the action of $F\ve$ on $A_G(u)$. 
Since $\ve$ is central in $A_G(u)$, this isomorphism is $F$-equivariant. 
\par
Assume that there exists $\ol y \in \ol G$ such that $\ol y \in \ol Y _{\ol u, \ol u'}^F$. 
Then there exists $y \in \b\iv(\ol y)$ such that $\b(yQ) = \ol y\ol Q$ and that
$y\iv uy \in u'_1U_Q$ for some $u_1' \in C'$.  This implies that 
$\ol y\iv \ol u \ol y \in \ol u_1'U_{\ol Q}$.  But since 
$\ol y\iv \ol u \ol y \in \ol u'U_{\ol Q}$, we must have $\ol u_1' = \ol u'$, 
and so $u_1' = u'$.  It follows that $y \in Y_{u,u'}$.   
Since $y \in \b\iv (\ol y)$, $F(y) = y$ or $F(y) = \ve y$.  
Assume that $F(y) \ne y$, then $F(y) = \ve y$. 
In this case, we consider $Y_{u_1, u'}$ instead of $Y_{u,u'}$.   
Then we have $y\iv h\iv u_1hy \in u'U_Q$. 
Since $F(hy) = F(h)F(y) = \ve h \cdot \ve y = hy$, this implies that  
$hy \in Y_{u_1, u'}^F$. The lemma is proved.
\end{proof}

\remark{3.10.} In the course of the proof of Lemma 3.10 in \cite[III]{Sh4}, a property 
``if $\ol y \in \ol Y_{\ol u, \ol u'}^F$, then there exists $y \in Y_{u,u'}^F$'' was used. 
But this is not always the case, and if this does not hold, one needs to
replace $u$ by $u_1$ as in the above lemma.  Thus the proof of Lemma 3.10 must be modified.
Also in Lemma 3.10, only the case where $F$ acts trivially on $A_G(u)$ and on $A_M(u')$
was discussed.  By using Lemma 1.15, we can relax this condition, and actually
Lemma 3.10 in [loc. cit.] is not used in later discussions.    

\para{3.11.}
Since Lemma 3.10 in \cite[III]{Sh4} does not hold in this form,
the definition of split elements given in \cite[III, 3.2]{Sh4} is not appropriate.
In view of Lemma 3.9, the choice of $u \in C^F$ for a split $\b(u)$ should be more
delicate. We consider the setup in 3.8, and let $\ve_{u,u'}(\r,\r')$ be as in 1.14.
In the following, by modifying the discussion 
in \cite[III, Prop. 3.11]{Sh4}, we shall prove that there 
exists $v \in C^F$ such that $F$ acts trivially on $\ve_{v,u'}(\r,\r')$.
\par
Now assume that $v \in C^F$, and that there exists $y \in Y_{v,u'}^F$.
Then $y\iv vy \in Q$, and we have natural maps
\begin{equation*}
\tag{3.11.1}  
  A_G(v) \simeq A_G(y\iv vy)  \leftarrow A_Q(y\iv vy) \to A_M(u').
\end{equation*}  
\par
We can identify $\CQ = G/Q$ with the variety $\SF$ of all two-dimensional totally
isotropic subspaces of $V$.  Then  as in \cite[4.8]{LS2},
$\CQ_{u, C'}$ in (1.12.1) can be identified with
the subvariety $\SF_{u, C'}$ of $\SF$ consisting of all $E \in \SF$ such that
\begin{enumerate}
\item \ $E$ is $\ol u$-stable.
\item \ $\ol u|_E \ne 1$.
\item \ The Jordan type of $\ol u|_{E^{\perp}/E}$ is $\la'$.
\end{enumerate}

In the case where $F$ is split, we can choose an $F$-stable basis of $V$
associated to $u$ as
in \cite[III, 3.1]{Sh4}. In particular, for each $\la_j$ odd, we have a
vector space $V_j$ of $\dim V_j = \la_j= h$ with basis $e^j_1, \dots, e^j_h$, and
for each $\la_j = \la_{j+1}$ even, we have a vector space $V_j$ of
$\dim V_j = 2\la_j = 2h$
with basis $e^j_1, \dots, e^j_h, e^{j+1}_1, \dots, e^{j+1}_h$.
$V$ is expressed as a direct sum of those subspaces $V_j$. 
\par
In the setup in 3.8, assume that $\ol u $ is split in $\b(C)^F$,
and that $\ol u'$ is split in $\b(C')^F$. 
Also assume that $\r \in A_G(u)\wg\ex, \r' \in A_M(u')\wg\ex$, and
that an extension $\wt\r'$ of $\r'$ is fixed.
We consider the following cases.  Note that as explained in \cite[III, Remark 3.12]{Sh4},
the case (III) is not used in later discussions, and so is omitted. 
\par\medskip
{\bf Case I.} Take $u \in C^F$ such that $\ol u = \b(u)$ is a split element in
$\b(C)^F$. 
First assume that $F$ is split, and let $e_1, \dots, e_h$ be the $F$-stable basis
of $V_i$ given as above.  Let $E = \lp e_1, e_2\rp$. Then it is easy to see
that $\SF_{u,C'} = \{ E \}$.  
Even if $F$ is non-split, the variety $\SF_{u,C'}$ is $F$-stable, hence
$E$ is an $F$-stable subspace of $V$ under the isomorphism $\SF \simeq G/Q$. 
Then $GL(E) \times SO(E^{\perp}/E)$ is $\ol G^F$-conjugate to $\ol M$.
Since $\ol u$ is split in $\ol G$, $(\ol u|_E, \ol u|_{E^{\perp}/E})$ is
$\ol G^F$-conjugate to a split element $\ol u'$ in $\b(C')^F$. 
Hence there exists $\ol y \in \ol Y_{\ol u, \ol u'}^F$ such that $\ol y\ol Q$
corresponds to $E$, 
and $\b\iv(\ol y)$ is written as $\b\iv(\ol y) = \{ y, \ve y\}$. 
If $y \in G^F$, we have $y \in Y_{u,u'}^F$, and we put $v = u$. 
But if $F(y) = \ve y$, Lemma 3.9 shows that $hy \in Y_{u_1, u'}^F$,
and we put $v = u_1$. 
In particular, for $v = u$ or $u_1$ in $C^F$, there exists 
$y \in Y_{v,u'}^F$ such that $\ol y \in Y_{\ol v,\ol u'}^F$. 
Since $\SF_{u,C'} = \{ E\}$, we have $Y_{v,u'} = yZ_M(u')U_Q$.
Hence $X_{v,u'} \simeq A_M(u')$.  Now the structure of $A_G(u)$ is
described in \cite[14.3]{L2} (see also \cite[III, 2.6]{Sh4}).
In this case, (3.11.1) induces an isomorphism $A_G(v) \isom  A_M(u')$,
and $A_G(v) \times A_M(u')$ acts on $\ve_{v,u'} = \Ql[A_M(u')]$ (the group algebra of
$A_M(u')$ over $\Ql$) as the
two-sided regular representation of $A_M(u')$.
In particular, $\ve_{v,u'}(\r,\r') \ne 0$ only when $\r = \r'$,
and in which case, $F$ acts trivially on $\ve_{v,u'}(\r, \r') = \wt\r\otimes \wt\r'^*$
if we choose an extension $\wt\r$ by $\wt\r = \wt\r'$. 

\par\medskip
{\bf Case II.} \
We write the basis of $V_i$ in 3.11 as $e_1, \dots, e_h, f_1, \dots, f_h$
with $\la_i = \la_{i+1} = h$.
By the discussion in \cite[III, 3.11]{Sh4}, $\SF_{u,C'}$ is expressed as
\begin{equation*}
\SF_{u,C'} = \bigcup_{g \in G_1}g\{ \lp e_1, e_2 + \a f_1\rp \mid \a \in \Bk\},
\end{equation*}
where $G_1$ is a subgroup of $Z_{\ol G}(\ol u)$ isomorphic to $Sp(M_i)$
($M_i$ is a subspace of $V$ which is the direct sum of $V_k$ such that $\la_k = h$). 
Put $E = \lp e_1, e_2\rp$.  Then $E \in \SF_{u,C'}^F$ (since $h$ is even,
$E$ is $F$-stable even if $F$ is non-split).  Since $\lp e_1, e_2 + \a f_1\rp$
is $Z_{\ol G}(\ol u)$-conjugate to $E$, $\SF_{u,C'}$
coincides with the $Z_{\ol G}(\ol u)$-orbit of $E$.
Then  there exists $\ol y \in \ol Y_{\ol u,\ol u'}^F$ such that $\ol y\ol Q $
corresponds to $E$. 
By a similar discussion as in Case I, by replacing $u$ by $v \in C^F$,
one can find $y \in Y_{v, u'}^F$ and $\b(y) \in \ol Y^F_{\ol v, \ol u'}$.
Since $Z_{\ol G}(\ol v)$ acts 
transitively on $\SF_{v,C'}$,
$Y_{v,u'} = Z_G(v)yZ_M(u')U_Q$.  Here the stabilizer of $E$ in $G$ is
equal to $Q' = yQy\iv$.  Since $E$ is contained in $V_i$ such that $\la_i$ is even,
we have $A_G(y\iv vy) = A_Q(y\iv vy)$. Then (3.11.1) induces an isomorphism
$A_G(v) \isom  A_M(u')$.
Thus we have $X_{v,u'} \simeq A_M(u')$, and $\ve_{v,u'}$ coincides with the regular
representation $\Ql[A_M(u')]$. Hence
$\ve_{v,u'}(\r, \r') \ne 0$ only when $\r = \r'$, and if we choose an extension
$\wt\r = \wt\r'$, 
$F$ acts trivially on
$\ve_{v,u'}(\r, \r') = \wt\r\otimes \wt\r'^*$. 

\par\medskip\noindent
{\bf Case IV.} \
First assume that $F$ is split. Put $M_i = V_i \oplus V_{i+1}$,
where $\dim V_{i+1} = h$ and $\dim V_{i+1} = h-2$ with $h$ odd.
Write the basis of $V_{i+1}, V_i$ in 3.11 as $e_1, \dots, e_h$ for $V_{i+1}$,
and $e_1', \dots, e_{h-2}'$ for $V_i$.
Then by \cite[III, 3.11]{Sh4}, we have
\begin{equation*}
\SF_{u,C'} = \{ E = \lp e_1, e_2 + e_1'\rp, E' = \lp e_1, e_2- e_1'\rp \}.
\end{equation*}
Here both of $E, E'$ are $F$-stable. Even in the case $F$ is non-split,
$\SF_{u,C'}$ is written as $\SF_{u,C'} = \{ E, E'\}$ with $E, E'$ : $F$-stable. 
\par
There exist $\ol y \in \ol Y^F_{\ol u, \ol u'}$ such that 
$\ol y\ol Q$ corresponds to $E$, and $\ol y' \in \ol Y^F_{\ol u, \ol u'}$
such that $\ol y'\ol Q$ corresponds to $E'$.  One can find 
$v_1 \in C^F$ such that $y \in Y_{v_1,u'}^F$, and $v_2 \in C^F$ such that 
$y' \in Y_{v_2, u'}^F$ for some $v_1, v_2 \in \{ u,u_1\}$.  We need to show that 
$v_1 = v_2$. For this, we argue as follows. 
We consider 
a transformation $z \in SO(V_i)$ such that $z(e_1') = -e_1', z(e_{h-2}) = -e_{h-2}$, 
and $z(e_j') = e_j'$ for any other $j$. Let $W$ be a subspace of
$V_i$ spanned by $e_1', e'_{h-2}$, then $W$ is $F$-stable, and the corresponding map  
$\b_1 : \Spin (W) \to SO(W)$ is $F$-equivariant.  Here $\Spin (W) = \Spin_2$, and 
the center $Z_1$ of $\Spin_2$ is $\lp \ve_1 \rp \times \lp \w_1 \rp$, 
where $\Ker\b_1 = \lp \ve_1\rp$. Thus we have $\b_1(\w_1) = -1 \in SO_2$.
Since $\Spin_2$ is a subgroup of $\Spin_N$, 
one can find $x \in G^F$ such that $\b(x)|_{V_i}$ coincides with $z$.  It follows that 
$x(E) = E'$. Since $x \in G^F$, we see that $v_1 = v_2$ as asserted.   
Put $v = v_1 = v_2$. 
Then $F$ acts trivially on $\SF_{v, C'} = \{ E, E'\}$. Hence we have
$X_{v,u'} \simeq A_M(u') \sqcup A_M(u')$ as in Case I.
\par
Let $I$ be the set of $i$ such that $\la_i$ is odd. 
We follow the notation in \cite[III, 2.6]{Sh4}. 
Then $A_Q(y\iv vy) \simeq A_M(u')$, and by (3.11.1) $A_M(u')$
is regarded as a subgroup of $A_G(v)$ such that
\begin{equation*}
\tag{3.11.2}  
  A_G(v) = A_M(u')\sqcup x_ix_{i+1}A_M(u') \sqcup x_jx_iA_M(u') \sqcup x_jx_{i+1}A_M(u')
\end{equation*}
for $x_j$ with $j \notin \{ i, i+1\}$
($\{ x_j \mid j \in I\}$ is the set of generators of $A_G(v)$, and
$\{ x_j \mid j \in I -\{i,i+1\} \}$ is that of $A_M (u')$).
Note that the action of $A_{G}(v)$ on $\SF_{v,C'} = \{ E, E'\}$ is
given as follows;  $A_M(u')$ leaves $E, E'$ stable,
$x_ix_{i+1}$ leaves $E, E'$ stable, and $x_jx_i, x_jx_{i+1}$
permute $E$ and $E'$. Thus
\begin{equation*}
  \ve_{u,u'} \simeq \Ql[A_M(u')]\oplus \Ql[A_M(u')],
\end{equation*}
and
$A_G(v)$ acts on $\ve_{u,u'}$ through the decomposition (3.11.2).
Also the  action of $F$ on $\ve_{v, u'}$ is given by the action of $F$ on $A_M(u')$
through (3.11.2). 
Let $\r$ be the unique irreducible representation of $A_G(v)$ pertaining to $\x$
as given in \cite[14.4]{L2} (see also \cite[III, 2.5]{Sh4}),
and define $\r'$ of $A_M(u')$ similarly.
Then $\r|_{A_M(u')} = \r'\oplus \r'$, and $\ve_{u,u'}(\r, \r')$ coincides with
$\r\otimes \r'^*$. 
Assume given an extension $\wt\r'$ of $\r' \in A_M(u')\wg$. 
There exists a unique extension
$\wt\r$ of $\r \in A_G(v)\wg$ such that $\wt\r|_{\wt A_M(u')} = \wt\r' \oplus \wt\r'$.
Then $F$ acts trivially on $\ve_{v,u'} = \wt\r \otimes \wt\r'^*$. 

\par\medskip\noindent
{\bf Case V. } \ 
Put $h = \la_{i+1} = \la_{i+2}$ with $h$ even. We write the basis of
$V_{i+1}$ as $e_1, \dots, e_h, f_1, \dots, f_h$.
Then it is known by the discussion in \cite[III, 3.11]{Sh4} that
$Z_{\ol G}(\ol u)$ acts transitively on $\SF_{u,C'}$,
and $E = \lp e_1, e_2\rp \in \SF_{u,C'}^F$.
Thus we may choose an element 
$v \in C^F$ so that $y \in Y_{v,u'}^F$ with $\ol y \in \ol Y_{\ol v,\ol u'}^F$, where
$\ol v$ is split in $\b(C)^F$.   
In a similar way as in Case II, we have $Y_{v,u'} = Z_G(v)yZ_M(u')U_Q$.
Since $E = \lp e_1, e_2\rp $ is contained in $V_{i+1}$ with $i+1 \notin I$,
$x_i \in A_{\ol G}(\ol v)$ stabilizes $E$.  It follows that
$A_G(y\iv vy) = A_Q(y\iv vy)$,
and we have a natural map $\pi : A_G(v) \to A_M(u')$ by (3.11.1).
Hence $X_{v,u'} \simeq A_M(u')$, and $A_G(v)$ acts on $X_{v,u'}$ through 
$\pi$.
But $\pi$ is not surjective
($A_M(u')$ is generated by $\{ x_j \mid j \in I' = I - \{ i\} \sqcup \{ i+2\}\}$).
Then $A_G(v) \times A_M(u')$ acts on $\Ql[A_M(u')]$ through the homomorphism
$A_G(v) \times A_M(u') \to A_M(u') \times A_M(u')$.
Note that for $\r \in A_G(v)\wg$, $\dim \r$ is determined by $|I|$
(\cite[Prop. 14.4]{L1}, see also [III, Prop. 2.5]). 
For $\r' \in A_M(u')\wg\ex$, the pullback $\r = \pi^*(\r')$ also pertains to $\x$.
Since $|I| = |I'|$, $\r$ gives an irreducible
representation of $A_G(v)$ with $\dim \r = \dim \r'$.
Let $\wt\r'$ be an extension of $\r' \in A_M(u')\wg$.  We define $\wt\r$ as
the pullback of $\wt\r'$, then $F$ acts trivially on
$\ve_{v,u'}(\r,\r') = \wt\r \otimes \wt\r'^*$.

\para{3.12.}
For an $F$-stable class $C$ in $G$ corresponding to $\la \in X_N$,  
we define a split element $v \in C^F$ as follows. 
For $\la \in X_N$, there exists $\la' \in X_{N-4}$ such that $\la'$ 
is obtained from $\la$ by the pattern in (I), (II), (IV) or (V) 
(note that $X_N = \emptyset$ if $N = 2$).
We choose $\la'$ such that $\la_i$ is as big as possible, first look for $\la'$ in (I),
then in (II), and so on until (V).
Let $\b(C)$ be the $F$-stable class in $\ol G = SO_N$. 
By Theorem 3.2, there exists a split element $\ol u \in \b(C)^F$ (in the sense of 1.7). 
Let $Q = MU_Q$ be an $F$-stable parabolic subgroup of $G$ such that
$M$ is an $F$-stable Levi subgroup as in 3.8, 
and $C'$ be an $F$-stable unipotent class in $M$ corresponding to $\la'$. By induction,
we may assume that there exists a split element $u' \in C'^F$.  
Then $\ol u' = \b(u')$ 
is a split element in $\b(C')^F$. 
By applying the discussion in 3.11, 
one can find $v \in C^F$ and $y \in Y_{v,u'}^F$ such that
$\ol v$ is $\ol G^F$-conjugate to $\ol u$.   
We define $v \in C^F$ as a split element in $C^F$. 
\par
The following result is a refinement of Theorem 3.13 in \cite[III]{Sh4}. 
Note that the case $\x(\ve) = 1$ is covered by Theorem 3.2.

\begin{thm} 
Let $G = \Spin_N$. Assume that $(C,\SE) \in \SN_{\x}^F$ belongs to
$(L, C_0, \SE_0) \in \SM_{\x}^F$ pertaining to $\x \in Z_G\wg$ such that
$\x(\ve) = -1$. 
Choose $(u_0, \r_0) \lra (C_0, \SE_0)$ with $u_0\in C_0^F$ split.
Assume that $(v,\r) \lra (C, \SE)$ with $v \in C^F$ split. 
\begin{enumerate}
\item  Recall that $\wt A_G(v) = \lp\tau \rp \ltimes A_G(v)$.
For $\r \in A_G(v)\wg\ex$, there exists a unique extension
$\wt\r$ such that $\tau$ acts trivially on $\wt\r$.   
\item  Let $\wt\r_0, \wt\r$ be extensions of $\r_0, \r$ as in (i).
Then we have $\g(u_0, \wt\r_0, v, \wt\r) = 1$.    
In particular, $v \in C^F$ is a split element in the sense of 1.7.
\end{enumerate}  
\end{thm}  

\begin{proof}
First assume that $N \ge 7$. Let $M$ be an $F$-stable Levi subgroup such that
$\b(M) \simeq GL_2 \times SO_{N-4}$. Let $\la \in X_N$ be the Jordan type of
$\ol v \in \b(C)$.  By 3.12, there exists $\la' \in X_{N-4}$ such that
$\la'$ is obtained from $\la$ by the pattern in (I), (II), (V) or (V),
and a split element $u' \in C'^F$, where $C'$ is a unipotent class in $M$,
and the Jordan type of $\ol u'$ is $\la'$.  Then by 3.11, there exists
$\r' \in A_M(u')\wg$ such that $\r \in A_G(u)\wg$ is obtained from $\r'$.
\par
First we show (i).  By induction, we may assume that the extension $\wt\r'$
of $\r'$ has the property that $\tau$ acts trivially on $\wt\r'$.  Then 
$\wt\r$ is determined from $\wt\r'$ by the procedure discussed in 3.11.
In particular, if $\tau$ acts trivially on $\wt\r'$, then $\tau$ acts trivially
on $\wt\r$. Thus (i) holds.
\par
Next we show (ii).  Again by induction, we may assume that
$\g(u_0, \wt\r_0, u', \wt\r') = 1$.  Then $\wt\r$ is determined from $\wt\r'$,
and the discussion in 3.11 shows that $F$ acts trivially on $\ve_{v,u'}(\r,\r')$.
By Lemma 1.15, we have $\g(u_0, \wt\r_0, v, \wt\r) = \d\iv$ under the notation there.
Here $\d$ is determined by comparing the preferred extensions of
$V_{(u,\r)} \in \SW_L\wg$ and
of $V_{(u',\r')} \in {\SW _L'}\wg$. But in this case $\SW_L, \SW_L'$ are isomorphic to the
Weyl groups of type $B_m$.  Hence the preferred extensions are trivial,
and we have $\d = 1$.  Thus $\g(u_0, \wt\r_0, v, \wt\r) = 1$, and (ii) follows. 
\par
It remains to consider the case where $N \le 6$. Assume that $N = 6$ (resp. $N = 3$).
In this case, $X_N = \{ (15)\}$ (resp, $X_N = \{ (3) \}$), and
the corresponding $(u,\r) \lra (C,\SE)$ is a cuspidal pair by \cite[Cor. 4.9]{LS2},
with $u$ regular unipotent in $G$. 
The natural map $Z_G \to A_G(u)$ gives an isomorphism $Z_G \simeq A_G(u)$, hence
$A_G(u)$ is abelian.  Moreover, $\r \in A_G(u)\wg$ coincides with $\x \in Z_G\wg$.
By the definition of split elements in 1.7, the condition that $u$ is split
(in the sense of 1.7) is given by the condition for $(u,\r')$ with $\r' \ne \r$.
Since $(u, \r')$ does not belong to $X_N$, it is only concerned with $\ol G$.
In particular, we may choose a split element
$\ol u \in \b(C)^F$, and define $v \in C^F$ by the condition that $\b(v) = \ol u$.
Since we may choose $\vf_0 : F^*\SE_0 \isom \SE_0$ for $\SE_0 = \SE$ freely,
we choose $\wt\r \in \wt A_G(v)\wg$ such that $\tau$ acts trivially on $\wt\r$
(note that $A_G(v)$ is abelian). Hereafter, we define $\Xi_G$ in (1.7.1) by using
those $\vf_0$ corresponding to $\wt\r$ for $\Spin_6$ (resp. $\Spin_3$). 
Then $v \in C^F$ is a split element in the sense of 1.7, and (i), (ii) hold.
\par
Next assume that $N = 5$ or 4.
In those cases, the discussion in 3.8 makes sense; we consider
$Q = MU_Q$, where $M$ is a Levi subgroup of $Q$ such that $\b(M) \simeq GL_2$
(we regard $SO_{N-4}$ as the unit group $\{ 1\}$). 
We have $X_5 = \{ (5), (122)\}$ and $X_4 = \{ (22), (13)\}$.
Let $u \in C^F$ corresponding to $X_N$.  If $N = 5$, then
$A_G(u) \simeq \BZ_2$, and
the natural map $Z_G \to A_G(u)$ gives an isomorphism $Z_G \simeq A_G(u)$. 
If $N = 4$, $A_G(u)$ is abelian.  
$\SN_{\x}$ corresponds to $(L, C_0, \SE_0) \in \SM_{\x}$, where 
$L$ is a Levi subgroup such that $\b(L) = GL_2$. $(u_0, \r_0) \lra (C_0, \SE_0)$
is given by a regular unipotent element $u_0$ in $L$, and the nontrivial character
$\r_0$ of $A_L(u_0) = Z_L = \BZ_2$. Hence $M = L$.  
If we regard $X_1 = \{ (1)\}$, $X_0 = \{(-)\}$ as the set corresponding to $u_0$,
we have a bijection $X_1 \simeq (\SN_L)_{\x}, X_0 \simeq (\SN_L)_{\x}$.
The variety $\SF_{u,C'}$ is defined similarly to 3.11, and 
the discussion there is also applied to the present situation;
the cases (I), (II) for $N = 5$, (II) and the cases (II), (IV) for $N = 4$. 
Hence the verification of (i), (ii) is reduced to the case where $G = SL_2$.
In this case, the existence of split elements
is given by Theorem 2.19.
\end{proof}  

\par\bigskip
\section{ Preliminaries on exceptional groups}

\para{4.1.}
In the rest of this paper, we assume that $G$ is almost simple, simply connected 
of exceptional type. 
The classification of unipotent classes in $G$ were done by 
Stuhler \cite{St} for type $G_2$, by 
Shinoda \cite{Shi} for type $F_4$ with $p = 2$, by \cite{Sh1} for 
type $F_4$ with $p \ne 2$, and by Mizuno \cite{M1},\cite{M2} for 
$E_6, E_7$ and $E_8$.  
In this paper, we follow the notation in \cite{Sp5}, which is based 
on Mizuno's notation.   
\par
The following facts are known. 
\par\medskip\noindent
(4.1.1) \ Assume that $A_G(u)$ is abelian.  
If $F$ acts trivially on $A_G(u)$ for $u \in C^F$, then $F$ acts trivially on 
$A_G(u_1)$ for any choice of $u_1 \in C^F$, namely, the property that 
$F$ acts trivially on $A_G(u)$ is independent from  the choice of $u \in C^F$.
\par\medskip
In fact, $u_1 \in C$ is expressed as $u_1 = u_c$, an element
obtained from $u$ by twisting by $c \in A_G(u)$.  Then the action of $F$ on $A_G(u_1)$
corresponds to the action of $cF$ on $A_G(u)$, where $c$ acts as an inner automorphism.
Since $A_G(u)$ is abelian, the action of $cF$ coincides with that of $F$ on $A_G(u)$. 
\par\medskip\noindent
(4.1.2) \ Any unipotent class $C$ in $G$ is $F$-stable. 
For any class $C$, there exists a representative $u \in C^F$. 
\par\medskip
In fact, the former statement follows from the labeling of 
unipotent classes of exceptional groups in terms of the 
subsystems of the root system $\vD$ of $G$ (see 4.3 below).
If $C$ is written as $C = C_X$ with $X \subset \vD$, then 
$F(C_X) = C_{F(X)}$ gives the same unipotent class. 
The latter statement follows from the general theory. 
\par\medskip\noindent
(4.1.3) \ Unless $G$ is $E_6$, any class $C$ has 
a representative $u \in C^F$ such that $F$ acts trivially on $A_G(u)$. 
\par\medskip
In fact, it is checked, for all $A_G(u)$ occurring in 
exceptional groups not of type $E_6$, that the inequality  
\begin{equation*}
\tag{*}
|A_G(u)/\!\!\sim_F| < |A_G(u)/\!\!\sim|
\end{equation*}
holds unless the action of $F$ on $A_G(u)$ is given by an inner automorphism, 
where $A_G(u)/\!\!\sim_F\!$ is the set of $F$-twisted conjugacy classes of $A_G(u)$, and  
$A_G(u)/\!\!\sim$ is the set of conjugacy classes in $A_G(u)$. 
But $|A_G(u)/\!\!\sim_F|$ coincides with the number of $G^F$-orbits in $C^F$, and 
$|A_G(u)/\!\!\sim|$ coincides with the number of $G^{F^m}$-orbits in $C^{F^m}$ for a sufficiently 
large $m$ since $F^m$ acts trivially on $A_G(u)$.
Hence (*) does not occur. It follows that $F$ acts as an inner automorphism, 
and there exists $\da  \in Z_G(u)$ such that $\da F$ acts trivially on $A_G(u)$. 
Then $F$ acts trivially on $A_G(u_1)$, where $u_1 = u_a \in C^F$ is a twisted element 
of $u$ by $a \in A_G(u)$.   
\par\medskip\noindent
(4.1.4) \ Assume that $G = E_6$.  Then $Z_G \simeq \BZ_3$ (resp. $Z_G \simeq \{1\}$) 
if $p \ne 3$ (resp. $p = 3$). 
Assume that $p \ne 3$, and that $F$ is a split Frobenius map (resp. non-split Frobenius map), 
then $F$ acts trivially on $\BZ_3$ if and only $q \equiv 1 \pmod 3$ 
(resp. $q \equiv -1 \pmod 3$). 
If $F$ acts trivially on $Z_G$, then for any $C$, there exists $u \in C^F$ such that
$F$ acts trivially on $A_G(u)$, otherwise, there exists a class $C$ such that 
$F$ acts non-trivially on $A_G(u)$ for any choice of $u \in C^F$. 
\par\medskip 
In fact, a similar argument as in (4.1.3) works also for this case.
In the case where $G = E_6$ of adjoint type, (4.1.4) was given in 
\cite{M1}, \cite{BS} (see also the book of Liebeck-Seitz \cite[22.2.3]{LiS}).
The simply connected case is easily obtained from the above results. 

\para{4.2.}
For type $E_6, E_7, E_8$, the following labeling of the Dynkin diagram is used.

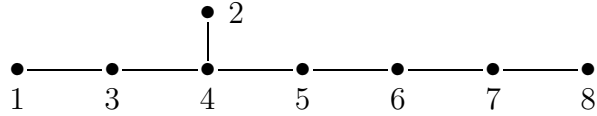
\SelectTips{cm}{12}
\objectmargin= {1pt}

\begin{equation*}
\begin{aligned}
\xygraph{
\bullet([]!{+(0,-.3)}{1}) - [r]
\bullet([]!{+(0,-.3)}{3}) - [r] 
\bullet([]!{+(0,-.3)}{4}) (
 - []!{+(0,.6)} \bullet([]!{+(.3, 0)} {2}), 
 - [r] \bullet([]!{+(0,-.3)}{5})
 - [r] \bullet([]!{+(0,-.3)}{6}) 
-  [r] \bullet([]!{+(0,-.3)}{7})
-  [r] \bullet([]!{+(0,-.3)}{8})
)} \\
\end{aligned}
\end{equation*}


In later discussion, we use a precise description of the root systems
of type $E_6, E_7, E_8$, which we follow the notation in \cite{Bo}.
Let $V = \BR^8$ be the euclidean space with the inner product $(\ , \ )$,
and $\ve_1, \dots, \ve_8$ an orthonormal basis of $V$. The set of
simple roots $\Pi = \{ \a_i \mid i \in I\}$ of type $E_8$ is given by
\begin{align*}
  \Pi &= \{ \a_1 = \frac{1}{2}(\ve_1 - \ve_2 -\cdots -\ve_7 + \ve_8),
             \quad \a_2 = \ve_2 + \ve_1, \quad \a_3 = \ve_2-\ve_1, \\
      &\quad \a_4 = \ve_3-\ve_2, 
                    \quad \a_5 = \ve_4-\ve_3,  \quad \a_6 = \ve_5-\ve_4,
                    \quad \a_7 = \ve_6-\ve_5, \quad \a_8 = \ve_7-\ve_6 \},   
\end{align*}
where $I = \{ 1, 2, \dots, 8\}$ corresponds to the labeling of the Dynkin diagram. 
The set $\vD^+$ of positive roots of type $E_8$ is given by
\begin{align*}
  \vD^+ = \bigl\{ \pm \ve_i + \ve_j\ (1 \le i < j \le 8), \quad
                \frac{1}{2}(\ve_8 + \sum_{i=1}^7(-1)^{\nu(i)}\ve_i) \
                  \text{ with  $\ \sum_{i=1}^7\nu(i)$ even} \bigr\}.  
\end{align*}
The set of simple roots $\Pi_7$ (resp. $\Pi_6$) of type $E_7$ (resp. $E_6$)
is given by $\Pi_7 = \{ \a_i \mid 1 \le i \le 7\}$
(resp. $\Pi_6 = \{ \a_i \mid 1 \le i \le 6\}$),
and the set of positive roots $\vD^+_7$ of type $E_7$, and $\vD^+_6$ of type $E_6$
are given as subsets of $\vD^+$ by

\begin{align*}
  \vD_7^+ &= \bigl\{ \pm \ve_i + \ve_j \ (1 \le i< j \le 6), \quad \ve_8 - \ve_7,
                   \quad \frac{1}{2}(\ve_8 - \ve_7 + \sum_{i=1}^6(-1)^{\nu(i)}\ve_i)\bigr\},  \\
  \vD_6^+ &= \bigl\{ \pm  \ve_i + \ve_j \ (1 \le i < j \le 5), \quad
                   \frac{1}{2}(\ve_8 - \ve_7-\ve_6 + \sum_{i=1}^5(-1)^{\nu(i)}\ve_i)\bigr\}.
\end{align*}  

Next consider the case of $F_4$.  The labeling of the Dynkin diagram is given by
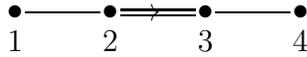
\begin{equation*}
\xygraph{!~:{@ 2{-}|@{>}}
\bullet([]!{+(0,-.3)}{1}) - [r]
\bullet([]!{+(0,-.3)}{2}) : [r]
\bullet([]!{+(0,-.3)}{3})  - [r]
\bullet([]!{+(0,-.3)}{4}) 
} 
\end{equation*}
Following \cite{Bo}, we give the root system  of type $F_4$ as follows.
Let $V = \BR^4$ be the euclidean space with the inner product $(\ \ )$,
and $\ve_1, \dots, \ve_4$ an orthonormal basis of $V$. 
The set of simple roots $\Pi = \{ \a_i \mid i \in I\}$ is given by
\begin{equation*}
  \Pi = \{ \a_1 = \ve_2 -\ve_3, \quad \a_2 = \ve_3 - \ve_4, \quad \a_3 = \ve_4, \quad
  \a_4 = \frac{1}{2}(\ve_1-\ve_2-\ve_3-\ve_4)\},
\end{equation*}
where $I = \{ 1,2,3,4\}$ corresponds to the labeling of the Dynkin diagram.
The set $\vD^+$ of positive roots of type $F_4$ is given by
\begin{equation*}
  \vD^+ = \{\ve_i, \quad  \ve_i \pm \ve_j \ (1 \le i < j \le 4), \quad 
           \frac{1}{2}(\ve_1 \pm \ve_2 \pm \ve_3 \pm \ve_4)\}.
\end{equation*}  

\para{4.3.}
Here we explain the labeling of unipotent classes for exceptional groups 
used in \cite{Sp5}, which depends on \cite[p.26]{Sp2}. 
Let $G$ be a simply connected, almost simple group of exceptional type, 
and $\vD$ the root system associated to $G$ and $T$, with 
the vertex set $I$ corresponding to the set of simple roots $\{\a_i \mid i \in I\}$.  
For any subsystem $X$ of 
$\vD$, we denote by $C_X$ the unipotent class obtained from the regular 
unipotent class of a reductive subgroup $G_X$ of $G$. The unipotent class 
$C$ in $G$ is called type $X$ if $C = C_X$.  
The conjugacy class of the subsystem $X$ in $\vD$ determines a unique 
unipotent class $C = C_X$ except the following cases.
\par\medskip
(a) \ In the case where $G$ is not simply laced, we denote by $\wt A_i$ the 
subsystem $A_i$, where all vertices correspond to short roots, and use the notation 
$A_i$ where all vertices correspond to long roots. The notation of this kind 
is also used for the case of $F_4$, such as $A_1 + \wt A_2$, etc. 
\par\medskip
(b) \ Assume that $G$ is $E_7$ (resp. $E_8$). In this case, there exist 
two families $X = 3A_1, A_3 + A_1, 4A_1, A_3 + 2A_1, A_5, A_5 + A_1$ 
(resp. $X = 4A_1, A_3 + 2A_1, 2A_3, A_5 + A_1, A_7$) 
such that $C_X$ represent two different classes in $G$. 
In this case, we denote by $X'$ such as $(3A_1)'$, etc., 
the subsystem of $\vD$ which is contained in 
the subsystem $A_7 \subset E_7$ (resp. the subsystem $A_8 \subset E_8$), 
and denote by $X''$ otherwise, such as $(3A_1)''$, etc..  
Here the subsystem $A_7$ (resp. $A_8$) is obtained from the extended Dynkin diagram 
$\wt E_7$ (resp. $\wt E_8$) by using the negative of highest root. 
\par
Moreover, for simplifying  the labeling of unipotent classes, 
$(4A_1)'', (A_3 + 2A_1)''$ are written as
$4A_1$, $A_3 + 2A_1$ in the case of $E_7$, and $(4A_1)', (A_3 + 2A_1)', (2A_3)', (A_7)'$ 
are written as $4A_1, A_3 + 2A_1, 2A_3, A_7$ in the case $E_8$ 
(since other cases do not appear in the classification of unipotent classes).

\par\medskip
(c) \ In the case where $p$ is a bad prime, we use a notation such as 
$X_2$ (for $p = 2$) or $X_3$ (for $p = 3$), in addition to $X$, 
such as $(A_3 + A_2)_2, A_3 + A_2$  for $E_7$ with $p = 2$. 
In this case, $X_2$ (or $X_3$) represents the class which has a representative 
in the regular unipotent class in the subgroup $G_X$, and $X$ represents 
another class which has no intersection with $G_X$. 
The class $A_3 + A_2$ has a representative in the semi-regular unipotent 
class in the Levi subgroup of type $D_6$ in the group $E_7$.    
\par\medskip
(d) We denote by $X(a_i)$ the subregular classes in the reductive group $G_X$.
\par\medskip
This fits to the notation of the tables in \cite{M2}. 

\para{4.4.}
For a unipotent element $u$ in $G$, the possibilities for $A_G(u)$ are given as
\begin{equation*}
  1, \BZ_2, \BZ_3, \BZ_4, \BZ_5, \BZ_6 \simeq \BZ_2 \times \BZ_3,
  \quad \BZ_2 \times \BZ_2, \quad S_3, S_4, S_5, \quad \BZ_2 \times S_3, \quad D_8,
\end{equation*}
where $\BZ_n$ is the cyclic group of order $n$, $S_n$ is the symmetric
group of degree $n$, and $D_8$ is the dihedral group of order 8.
\par
Concerning with $A_G(u)\wg$, we follow the notation in \cite{Sp5}. In particular,
$\BZ_2\wg = \{ 1, -1\}$, where $-1$ is the non-trivial representation of $\BZ_2$.
$S_3\wg = \{ 1, \th, \ve\}$, $D_8\wg = \{ 1, \ve, \ve',\ve'',\th\}$,
where $\ve$ is the sign representation, $\th$
is the unique irreducible representation such that $\deg \th = 2$,
and $\ve', \ve''$ (in the case $D_8$) are two additional representations
of degree 1.  
$(\BZ_2 \times S_3)\wg = \{ 1, \th, \ve, -1, -\th, -\ve\}$ where
$-\r$ denotes $(-1)\otimes \r$ for $\r \in S_3\wg$. 
\par
Concerning with $\SW_L\wg$, we follow the notation in \cite{Sp5}.
In particular, for $\SW_L = W(G_2)$, we write it as $\{ 1, \ve, \ve_l, \ve_c, \th', \th''\}$,
where $\ve$ is the sign representation, and $\ve_l$ (resp. $\ve_c$) is the
long sign representation (resp. short sign representation).  $\th', \th''$ are irreducible
representations of degree 2, where $\th'$ is the reflection representation of $W(G_2)$
($\th' = \f_{2,1}$ in the table of Carter's book \cite[p.412]{C}). Note that in \cite{Sp5},
$\th'$ and $\th''$ are not distinguished.  Here in order to distinguish them,
we use the results of \cite{L5}, \cite{H} (see Introduction). 

\para{4.5.}
By induction, we may assume that Conjecture 1.10 holds for any reductive 
subgroup of $G$ with smaller semisimple rank.
In particular, for any $(L, C_0, \SE_0) \in \SM_G^F$ such that $L \ne G$, 
a split element $u_0 \in C_0^F$, and the split extension 
$\wt\r_0$ corresponding to $\vf_0: F^*\SE_0 \isom \SE_0$ is already determined.  
\par
Let $T$ be an $F$-stable maximal torus of $G$.
For a unipotent class $C$ in $G$, let $\SM_G(C)$ be the subset of $\SM_G$ 
consisting of $(L,C_0,\SE_0)$ with $T \subset L \ne G$, such that 
$(C,\SE)$ belongs to $(L,C_0,\SE_0)$ for some $\SE$.
Note that $\SM_G(C)$ always contains $(T,\{1\},1)$, 
which corresponds to the part of the Springer correspondence among the generalized 
Springer correspondence (1.1.4).  
Also note that we have excluded $(G, C, \SE) \in \SM_G$ for a cuspidal pair $(C, \SE)$ of $G$.  
\par
We consider the following unipotent class. 
\par\medskip\noindent
(4.5.1) \ $G = E_7$ and $C$ is a unipotent class
of type $D_6(a_1) + A_1$ in $G$, with $p \ne 2$. 
\par\medskip

We shall prove the following result, which asserts
that Conjecture 1.10 holds if $G$ is an almost simple group of
exceptional type, possibly except one class in (4.5.1). 

\begin{thm} 
Let $G$ be an almost simple, simply connected group of exceptional type,
and $C$ an ($F$-stable) unipotent class in $G$.
Assume that $C$ is not the class in (4.5.1). 
Then there exists a split element $u \in C^F$ in the sense of 1.7.
More precisely, 
\begin{enumerate}
\item \ Assume that $F$ is split on $G$.
  \begin{enumerate}
    \item 
  There exists an element in $C^F$,  unless 
  $G$ is $E_8$ and $C = D_8(a_3)$ with $q \equiv -1 \pmod 3$, such that
  $\g(u_0, \wt\r_0, u, \r) = 1$ for any $\r \in A_G(u)\wg$. 
     \item 
  Assume that $G$ is $E_8$ and $C = D_8(a_3)$ with $q \equiv -1 \pmod 3$.
  Then $A_G(u) = S_3$, and there exists an element $u \in C^F$ such that
      \begin{equation*}
          \g(u_0, \r_0, u, \r) = \begin{cases}
                                 1   &\quad\text{ if } \r \ne \ve,   \\
                                 -1  &\quad\text{ if } \r = \ve,
                              \end{cases}
       \end{equation*}
  \end{enumerate}
where $\SM_G(C) = \{ (T, \{1\}, 1)\}$, and $(u_0, \r_0) = (1,1)$.   
\item \ Assume that $G$ is of type $E_6$, and $F$ is non-split 
on $G$. There exists an element $u \in C^F$ such that 
\begin{equation*}
\tag{4.6.1}
\g(u_0, \wt\r_0, u, \wt\r) = \nu_{u,\eta}(-1)^{\d_E},
\end{equation*}
where $E = V_{(u,\r)} \in \SW_L\wg$, and $\d_E$ is defined as in (1.7.4).
$\nu_{u,\eta}$ is a root of unity determined by the class $C$ and by $\eta = (L,C_0, \SE_0)$.
If $L = T$ or $L$ is of type $D_4$, then $\nu_{u,\eta} = 1$, while if $L$ is of type
$2A_2$ with $p \ne 3$, then $\nu_{u,\eta}$ is determined from the discussion
in Section 2 (for $G$ of type $A_5$ and $L$ of type $2A_2$). 
\end{enumerate}
\end{thm}

\para{4.7.}
The computation of Green functions is a part of the computation of
generalized Green functions.  As explained in the introduction,
Green functions of exceptional groups were computed for any $p$.  
In computing Green functions, Beynon-Spaltenstein [BS]
gave a notion of split elements (a similar kind of notion already appears
in \cite{Sh2}, \cite{Sh-cl},
where those elements are called distinguished elements).  Here we discuss the relation
between their definition of split elements and ours.
\par
Let $(u,\r) \lra (C,\SE)$ be the pair belonging to the series $(T, \{1\}, 1) \in \SM_G$.
Let $Y_{u,u'}$ and $\CQ_{u,C'}$ be the varieties defined in (1.12.1).
We follow the setup in 1.13.  In particular, $F$ acts naturally on $Y_{u,u'}$ and on $\CQ_{u,C'}$.  
We consider the case where $Q = MU_Q $ is $B = TU$, and $u' \in T$ is 1.
Then $Y_{u,1} = \{ g \in G \mid g\iv ug \in B \}$, and
$\CQ_{u,\{1\}} = \{ gB \in G/B \mid g\iv ug \in B\}$ coincides
with the fixed point subvariety $\SB_u$ of the flag variety $\SB = G/B$.
Thus the set $X_{u,1}$ coincides with the set of irreducible components
(they have the same dimension) of $\SB_u$, and so $\ve_{u,1}$ can be identified
with $H^{2d_u}(\SB_u) = H^{2d_u}(\SB_u,\Ql)$,
with $d_u = \dim \SB_u$,  and $X_{u,1}$ corresponds to the natural basis
of $H^{2d_u}(\SB_u)$.  
By the Springer correspondence, $H^{2d_u}(\SB_u)$ is decomposed as
$W \times A_G(u)$-module as
\begin{equation*}
H^{2d_u}(\SB_u) \simeq \bigoplus_{\r \in A_G(u)\wg}V_{(u,\r)}\otimes \r.
\end{equation*}
This gives the $A_G(u)$-module structure of $\ve_{u,1}$, and the action of
$F$ on $X_{u,1}$ coincides with the multiplication by $q^{d_u}$ followed by
the permutation of the basis. 
\par
Assume that $(G,F)$ is split.  
Beynon-Spaltenstein defined a split element $u \in C^F$ by the condition that
$F$ acts trivially on $X_{u,1}$, and on $A_G(u)$. This condition determines
the $G^F$-class of split elements uniquely if it exists, in the case where
$G$ is adjoint type, but not necessarily so in the simply connected case.
(Note that $\SB_u$ depends only on its image to the adjoint group.)
Assume that $F$ acts trivially on $A_G(u)$.
Then $V_{(u,\r)}\otimes \r$
is an irreducible $W \times A_G(u)$-module, and $F$ acts trivially on
$W \times A_G(u)$, hence $F$ acts on
the isotypic component $V_{(u,\r)}\otimes \r$ in $H^{2d_u}(\SB_u)$
by a scalar multiplication $cq^{d_u}$ for some $c \in \Ql^*$.
By applying Lemma 1.15 to the case where $(u_0, \r_0) = (1,1)$ is the cuspidal pair
in $T$, we have the following. 
\par\medskip\noindent
(4.7.1) \ Assume that $F$ acts trivially on $A_G(u)$, and
that $(u,\r)$ appears in the Springer correspondence. Let $c$ be as above.
Then we have $\g(1,1,u,\r) = c$. 
In particular, if $u$ is split in the sense of Beynon-Spaltenstein,
then we have $\g(1,1,u,\r) = 1$ for any $(u,\r)$ belonging to $(T, \{1\},1)$.  
\par\medskip
In the case where $|\SM_G(C)| = 1$, generalized
Green functions associated to $(L, C_0, \SE_0)$ for $L \ne G$
coincide with Green functions.  Hence if $u$ is split in the
sense of \cite{BS}, it is split in our sense.  In the case where $p$ is good,
they proved the existence of split elements $u \in C^F$, unless
$G = E_8$ and $C = D_8(a_3)$. In the exceptional case, $|\SM_G(C)| =1$,
and $A_G(u) \simeq S_3$. Thus there exists $u \in C^F$, unique up to $G^F$-conjugate,
such that $F$ acts trivially on $A_G(u)$. For this $u$, they have computed
$\g(1,1,u,\r) = \pm 1$ explicitly, which is given as in
Theorem 4.6, (i), (b).  Thus $u$ is split in our sense.
\par
In the case where $p$ is bad, Green functions were computed by Geck \cite{G2} and
L\"ubeck \cite{Lu}. From those computations, there exists $u \in C^F$ such that
$\g(1,1,u,\r) = 1$ if $F$ is split on $G$,
and that $\g(1,1,u,\r) = \pm 1$ as in Theorem 4.6 (ii) for $G = E_6$ with non split $F$. 
Thus we have
\par\medskip\noindent
(4.7.2) \ Assume that $\SM_G(C) = \{ (T, \{1\}, 1)\}$.  Then there exists
a split element $u \in C^F$.  

\para{4.8.}
In view of Lemma 1.15, it is important to know the structure of $X_{u,u'}$ and 
the Frobenius action on it. (Note that if $F$ acts trivially on $X_{u,u'}$, then 
$F$ acts trivially on $\ve_{u,u'}(\r,\r')$ for any $\r,\r'$.  This condition is
more convenient than considering $\ve_{u,u'}(\r,\r')$ separately.) 
\par 
In general, it is not so easy to determine the structure of $X_{u,u'}$.
But as discussed in \cite{Sp2}, there are two cases,  (III), and (IV) in \cite{Sp2}, 
where $X_{u,u'}$ can be described explicitly. Actually, the analysis of those cases
played a  crucial role when Spaltenstein determined the generalized 
Springer correspondence for exceptional groups (\cite{Sp2}). 
In those cases, we can also 
determine the Frobenius action on $X_{u,u'}$.
\par
From now on, in the remainder of the paper, we follow the setup in 1.13.
In particular, $L \subset P$ are $F$-stable, and
$M \subset Q$ are also $F$-stable such that $L \subset M$ and $P \subset Q$.
Assume that $C$ (resp. $C'$) is a unipotent class in $G$ (resp. in $M$).
Recall that the class $C$ in $G$ is induced from 
the class $C'$ in $M$ if $C \cap u'U_Q$ is open dense 
in $u'U_Q$ for $u' \in C'$ (\cite{LS1}). In that case, we write it as $C = \Ind_M^G C'$. 
\par
The following lemma corresponds to the property (III) in \cite{Sp2}.

\begin{lem}  
Let $C$ (resp. $C'$) be an $F$-stable unipotent class in $G$ (resp. in $M$).
Assume that $C$ is induced from $C'$.  Take $u \in C^F$ and $u' \in C'^F$ 
such that $u \in u'U_Q^F$.   
\begin{enumerate}
\item \ We have $Z_Q^0(u) = Z_G^0(u) \subset Z_M(u')U_Q$.
Let $N = Z_Q(u)/Z_G^0(u) \subset A_G(u)$, and 
$H = (Z_M^0(u')U_Q \cap Z_G(u))/Z_G^0(u) \subset N$.  Then 
$H$ is normal in $N$, and $N/H$ is naturally isomorphic to $A_M(u')$.
  Moreover, $X_{u,u'} \simeq A_G(u)/H$, and under this identification,
the action of $A_G(u) \times A_M(u')$ on $X_{u,u'}$ corresponds to 
the action of $A_G(u) \times N/H$ on $A_G(u)/H$;
$(a, nH) : xH \mapsto axn\iv H$ ($a,x \in A_G(u), n \in N$). 
\item \ 
Assume that $F$ acts trivially on $A_G(u)$. 
Then $F$ acts trivially on $X_{u,u'}$.
\end{enumerate}
\end{lem}

\begin{proof}
The proof is based on the discussion in \cite[1.3 (III)]{Sp2}.  
First we show (i).  
From the discussion in \cite{LS1}, the first formula follows, and  
$Z_Q(u)$ meets all the irreducible components of $Z_M(u')U_Q$. 
Then $H$ is normal in $N$, and $N/H$ is naturally isomorphic to $A_M(u')$. 
The latter statement of (i) was proved in \cite{Sp2} by using
\cite[Thm. 1.3 (d)]{LS1}. 
\par
If $F$ acts trivially on $A_G(u)$, then $F$ acts trivially on $H$. 
Then (ii) follows from (i) since the isomoprhism $X_{u,u'} \simeq A_G(u)/H$ is
compatible with $F$-actions.   
\end{proof}

\para{4.10.}
Next we consider the situation where (IV) in \cite{Sp2} can be applied.
Let $M \subset G$ be as in 4.8.  
Let $C$ (resp. $C'$) be an $F$-stable unipotent class in $G$ (resp. in $M$).  
Take $u \in C^F$ and $u' \in C'^F$.  Here we consider the special situation 
that $u = u' \in M^F$. 
The following lemma is an $F$-version of (IV) in \cite[1.4]{Sp2}.

\begin{lem}   
Assume that $u = u' \in C'^F$, and that   
$F$ acts trivially on $A_G(u)$.  
\begin{enumerate}
\item \  
$A_M(u)$ is naturally regarded as a subgroup of $A_G(u)$, hence 
$F$ acts trivially on $A_M(u)$. 
\item \ $X_{u,u'}$ is isomorphic to $A_G(u)$, where $A_G(u) \times A_M(u')$
acts on $A_G(u)$ by $(a,b) : x \mapsto axb\iv$, ($a,x \in A_G(u), b \in A_M(u')$).
In particular, $F$ acts trivially on $X_{u,u'}$. 
\end{enumerate}
\end{lem}

\begin{proof}
As discussed in \cite[1.4]{Sp2}, $A_M(u)$ can be naturally embedded in $A_G(u)$; 
let $S$ be the connected center of $M$.
Then $M = Z_G(S)$, and so $Z_M(u) = Z_G(S, u) = Z_{Z_G(u)}(S)$. 
On the other hand, $S \subset Z_G^0(u)$ and since the centralizer
of a torus in the connected group is connected, $Z_{Z_G^0(u)}(S)$ is connected.
Hence we have $Z_M^0(u) = Z_{Z_G^0(u)}(S)$.
Thus 
\begin{equation*}
A_M(u) = Z_{Z_G(u)}(S)/ Z_{Z_G^0(u)}(S) = 
(Z_G(u) \cap M)/(Z_G^0(u) \cap M ) \hra Z_G(u)/Z_G^0(u).
\end{equation*} 
Hence $A_M(u) \hra A_G(u)$, and (i) follows.  
It is known by \cite[Lemma 1.4 (i)]{Sp2} that $Y_{u,u'} = Z_G(u)U_Q$. Thus 
\begin{equation*}
Y_{u,u'} = \bigcup_{\a \in A_G(u)}\a Z_G^0(u)U_Q, 
\end{equation*}
and irreducible components of $Y_{u,u'}$ are all $F$-stable. Hence $F$ acts 
trivially on $X_{u,u'}$.
Moreover, it was proved in \cite[1.4]{Sp2}, that
$X_{u,u'}$ and $A_G(u)$ are isomorphic, compatible with $F$-actions,
as stated in (ii).
\end{proof}

By Lemma 4.11, we obtain the following.
\begin{prop}  
Assume that $L \subset M \ne G$, 
and choose $u = u'$ a split element in $C'^F$. 
Assume that $F$ acts trivially on $A_G(u)$. 
Let $(u,\r)$ be a pair belonging to $(L,u_0, \r_0) \in \SM^F_G$.
Then there exists an extension $\wt\r$ of $\r$ such that
$\g(u_0, \wt\r_0, u,\wt\r)$ is explicitly computable. 
In particular if $F$ is split on $G$, we have 
$\g(u_0, \wt\r_0, u, \wt\r) = 1$.
\end{prop}

\begin{proof}
  Since $\ve_{u,u'}$ is isomorphic to the regular representation of $A_G(u)$ by
Lemma 4.11 (ii), 
there exists $(u', \r') \in \SN^F_M$ such that 
$\r\otimes \r'^*$ appears in $\ve_{u,u'}$.
By assumption, $F$ acts trivially on $A_G(u)$, hence $F$ acts trivially on $A_M(u')$.
Since $u'$ is split in $C'^F$, we have 
$\g(u_0, \wt\r_0, u', \wt\r') = \nu'$ for a certain extension $\wt\r'$ of $\r'$,
where $\nu'$ is explicitly computable. 
Since $F$ acts trivially on $X_{u,u'}$ by Lemma 4.11, 
there exists an extension $\wt\r$, determined uniquely from $\wt\r'$ as in 1.14, 
and we have $\g(u_0, \wt\r_0, u, \wt\r) = \d\iv\nu'$ by Lemma 1.15.  
In the case where $F$ is split on $G$, then $F$ is split on $M$.
Since $\g(u_0, \wt\r_0, u', \wt\r') = 1$ and $\d = 1$, we obtain
$\g(u_0, \wt\r_0, u, \wt\r) = 1$. 
\end{proof}

Let $\SM_G(C)$ be as in 4.5.
As a corollary to Proposition 4.12, we have 

\begin{lem} 
Assume that $(G,F)$ is split. Assume that $\SM_G(C)$ is contained in
$\{ (T, \{ 1\}, 1), (L, C_0,\SE_0)\}$. Let $M$ be an $F$-stable Levi subgroup 
such that $T \subset L \subset M \subsetneq G$. 
Assume that for some $u_1 \in C^F$, $A_G(u_1)$ is abelian, and $F$ acts trivially
on $A_G(u_1)$.  
Let $u = u'$ be a split element in $M^F$.   
Then $u$ is a split element in $G^F$. 
\end{lem}

\begin{proof}
Since $A_G(u)$ is abelian, by our assumption on $u_1$,
$F$ acts trivially on $A_G(u)$ (see (4.1.1)).  Thus Proposition 4.12
can be applied, and we see that $\g(u_0, \wt\r_0, u, \r) = 1$ for any
$(u,\r)$ belonging to $(L, C_0, \SE_0)$. But this argument simultaneously
works also for $(u,\r)$ belonging to $(T, \{1\},1)$.  Hence
$\g(1,1,u, \r) = 1$ for $(u,\r)$ as above, and $u$ is a split element in $C^F$. 
\end{proof}  

\remarks{4.14.}
(i) \ 
For almost all cases 
$\SM_G(C) \subset \{ (T,\{ 1\}, 1), (L, C_0, \SE_0)\}$  (here $(L, C_0)$ is fixed, but
$\SE_0$ on $C_0$ is not unique in general). 
The exception is the case where $G = E_6$ and $C = E_6$, with  $p = 2$,   
or the case where $G = E_8$ and 
$C = E_7, E_8$,  with $p = 2$. In those cases, two types of Levi subgroups $L$ and $L'$ containing $T$
appear in $\SM_G(C)$. 
\par
(ii) \ Assume that $\SM_G(C) = \{ (T, \{1\},1), (L, C_0, \SE_0), (L', C_0', \SE_0') \}$
(here $\SE_0, \SE'_0$ are not unique in general).
If $T \subset L \subset L' \ne G$, then the discussion of Lemma 4.13 can be applied, and
one can find a split element in $C^F$.
This is applied for the case $G = E_8$.
But for the case $G = E_6$ this does not hold, and we need an additional argument
to show  the existence of split elements. 
\par\medskip

\para{4.15.}
In the case where there does not exist $M \ne G$ such that
$C \cap M \ne \emptyset$, Lemma 4.13 can not be
applied, and we need a more precise computation.  The following lemma
is useful for such a situation.
Let $Q = MU_Q$ be as in 1.12 with $M \ne G$, and $C$ (resp. $C'$) be an $F$-stable unipotent
class in $G$ (resp. in $M$).  Take $u \in C^F$, and $u' \in C'^F$, and assume that
$(u,\r)$ belongs to $(L, C_0, \SE_0)$. Recall that $\SW_L = N_G(L)/L$ and
$\SW_L' = N_M(L)/L$. We consider $V_{(u,\r)} \in \SW_L\wg$, and
$V_{(u',\r')} \in (\SW_L')\wg$ for $\r' \in A_M(u')\wg$. 

\begin{lem} 
Under the notation in 4.15, 
assume that the following conditions are satisfied.

\begin{enumerate}
\item  \ $V_{(u',\r')}$ appears in the restriction of $V_{(u,\r)}$ on $\SW_L'$.
\item \ $1\otimes 1$ appears in $\ve_{u,u'}$ with multiplicity 1.
\item  \ $\dim Z_G(u) = \dim Z_M(u')$.
\end{enumerate}
Assume that $u' \in C'^F$ is a split element such that $F$ acts trivially on
$A_M(u')$. Take $u \in G^F$ such that $Y_{u,u'}^F \ne \emptyset$.
If $F$ acts trivially on $A_G(u)$, then 
$\g(u_0, \wt\r_0, u, \wt\r) = \d\iv \g(u_0, \wt\r_0, u', \wt\r')$,
where $\d$ is as in Lemma 1.15. 
\end{lem}  

\begin{proof}
Since $V_{(u',\r')}$ appears in the restriction of $V_{(u,\r)}$,
by the restriction formula (1.12.3), $\r\otimes\r'^*$ appears in $\ve_{u,u'}$
with non-zero multiplicity. Since $1\otimes 1$ appears in $\ve_{u,u'}$ with multiplicity 1,
$A_G(u) \times A_M(u')$ acts transitively on $X_{u,u'}$.
By (1.12.2), $\dim \CQ_{u,C'} \le (\dim Z_G(u) - \dim Z_M(u')/2$.  Thus
the condition $\dim Z_G(u) = \dim Z_M(u')$ implies that $\dim \CQ_{u,C'} = 0$. 
Since $Y_{u,u'}/Z_M(u')U_Q \simeq \CQ_{u,C'}$, $Z_M^0(u')U_Q$ gives an irreducible
component of $Y_{u,u'}$ of dimension $d_{u,u'}$.
Since $Y_{u,u'}^F \ne \emptyset$, there exists $g \in G^F$ such that
$g\iv ug \in u'U_Q$. Then  $gZ_M^0(u')U_Q  \in  X_{u,u'}$, and 
\begin{equation*}
\tag{4.16.1}  
Y_{u,u'} = \bigcup_{a \in A_G(u), b \in A_M(u')}agbZ_M^0(u')U_Q
\end{equation*}
gives a decomposition of $Y_{u,u'}$ into irreducible components, all of which
have the same dimension $d_{u,u'}$.
Here $gZ_M^0(u')U_Q \in X_{u,u'}$ is $F$-stable. 
Since $F$ acts trivially on $A_G(u)$ and on $A_M(u')$, $F$ acts trivially on $X_{u,u'}$
by (4.16.1). Then Lemma 1.15 can be applied, and the lemma follows. 
\end{proof}  

\para{4.17.}
In later sections, we discuss each case of exceptional groups separately. 
First we treat the case where $G$ is not $E_6$, and after that
discuss the case $G = E_6$.  This is because 
in the case of $E_6$, the Frobenius actions on $A_G(u)$ are more complicated 
than in other cases (see 4.1). 

\par\bigskip
\section {The case $E_7$ } 

\para{5.1.} 
In this section, we assume that $G$ is simply connected of type $E_7$. 
In this case, $A_G(u) \simeq S_3 \times \BZ_{(2,p-1)}$ if 
$C$ is $D_6(a_2) + A_1$, and $A_G(u) \simeq S_3$ if $C = D_4(a_1)$. 
In all other cases, $A_G(u)$ is abelian. 
\par
Let $(3A_1)''$ be the Dynkin diagram of the following type
(corresponding to white nodes),

\SelectTips{cm}{12}
\objectmargin= {1pt}

\begin{equation*}
\tag{5.1.1}
\begin{aligned}
&\xygraph{
\bullet([]!{+(0,-.3)}) - [r]
\bullet([]!{+(0,-.3)}) - [r] 
\bullet([]!{+(0,-.3)}) (
 - []!{+(0,.6)} \circ([]!{+(.3, 0)}), 
 - [r] \circ([]!{+(0,-.3)})
 - [r] \bullet([]!{+(0,-.3)}) 
-  [r] \circ([]!{+(0,-.3)})
)} \\
\end{aligned}
\end{equation*}

The set $\SM_G$ consists of 
$(T,\{1\}, \Ql)$ and $(L, C_0, \SE_0)$, together with cuspidal ones, where
$L$ is a Levi subgroup of type $(3A_1)''$ if $p \ne 2$, and of type $D_4$ if $p = 2$.

\para{5.2.}
For a given class $C$, we shall find a split element $u \in C^F$. 
First consider the cases where Lemma 4.13 can be applied.
Consider the unipotent class $C$ in the list 
\begin{align*}
\tag{5.2.1}
E_7, \quad E_7(a_1), \quad E_7(a_2),  \quad D_6 + A_1, \quad D_6(a_1) + A_1, 
\quad D_6(a_2) + A_1, \quad D_4(a_1).
\end{align*}
In those cases, Lemma 4.13 cannot be applied; 
in the former four cases, there does not exist a proper Levi subgroup
$M$ such that $C \cap M \ne \emptyset$, and 
if $C = D_6(a_2) + A_1$ or $D_4(a_1)$, then $A_G(u)$ is not abelian.

\par
Assume that $C$ is in the remaining cases.
First assume that $p \ne 2$. In this case, $L$ is of type $(3A_1)''$.
If $|\SM_G(C)| = 2$, then $C$ is in the list
\begin{align*}
  &D_6, \quad D_6(a_1), \quad D_5 + A_1,  \quad
  D_6(a_2), \quad (A_5 + A_1)'', \\
  &D_5(a_1) + A_1, \quad A_5'', \quad D_4 + A_1, \quad 
  A_3 + A_2 + A_1, \quad D_4(a_1) + A_1, \\
  &A_3 + 2A_1, \quad (A_3 + A_1)'', \quad A_2 + 3A_1, \quad 4A_1, \quad 3A_1''. 
\end{align*}
In those cases, one can find $M \supset L$ such that $M \cap C \ne \emptyset$,
hence Lemma 4.13 can be applied. For other cases, since $|\SM_G(C)| = 1$, 
there is no restriction on $M$, and so Lemma 4.13 can be applied (or (4.7.2) can be
applied).  
\par
Next assume that $p = 2$.  In this case $L$ is of type $D_4$.
If $|\SM_G(C)| = 2$, then $C$ is in the list

\begin{align*}
E_6, \quad D_6, \quad D_6(a_1), \quad D_5 + A_1, \quad D_5, \quad D_4 + A_1, \quad D_4. 
\end{align*}
In those cases, one can find $M \supset L$ such that $M \cap C \ne \emptyset$,
hence Lemma 4.13 can be applied.  For other cases , since $|\SM_G(C)| = 1$,
Lemma 4.13 or (4.7.2) can be applied. 
\par
Thus except the classes listed in (5.2.1), Lemma 4.13 can be applied for $C$,
and one can find a split element in $C^F$. 

\para{5.3.}
We consider the case where $C = E_7, E_7(a_1), E_7(a_2), D_6 + A_1$ or $D_6(a_1) + A_1$,
and apply Lemma 4.9 or Lemma 4.16 for it.
\par\medskip
(1) \ Assume that $C = E_7$.  Then $C$ is a regular unipotent class in $G$,
and for $u \in C$, $A_G(u) \simeq \BZ_{2(6,p)}$.
Assume that $(u,\r)$ belongs to $(L, u_0, \r_0) \in \SM_G$ with $L \ne G$.
Then $V_{(u, \r)}$ is the trivial representation of $\SW_L$.  
Here $C = \Ind_T^G1$.  Let $M$ be a Levi subgroup of type $D_6$.  
and put $C' = \Ind_T^M 1$.  Then $C'$ is a regular unipotent class
in $M$.  By Lemma 4.9, $A_M(u')$ is a subquotient of $A_G(u)$ for $u' \in C'$.
Assume that $(u',\r')$ belongs $(L, u_0, \r_0) \in \SM_M$.
Then $V_{(u',\r')}$ is the trivial
representation of $\SW_L'$. In particular, $V_{(u',\r')}$ appears in the restriction
of $V_{(u,\r)}$ to $\SW_L'$. 
\par
We choose a split element $u' \in C'^F$, and take $u \in C^F$ such that
$u \in u'U_Q$.  Then $F$ acts trivially on $A_G(u)$ (note that $A_G(u)$ is abelian)
and on $A_M(u')$. By applying Lemma 4.9, $F$ acts trivially on $X_{u,u'}$.
Thus by Lemma 1.15, $\g(u_0, \wt\r_0, u, \r) = 1$ (in the notation
of Lemma 1.15, we have $\d = 1, \nu' = 1$ since $F$ is split on $G$ and on $M$).
This holds simultaneously for $L$ and $T$. Thus $u \in C^F$ is a split element. 
\par\medskip
(2) \ Assume that $C = E_7(a_1)$.  
Then $A_G(u) \simeq \BZ_2$, and $\dim Z_G(u) = 9$.
Assume that $p \ne 2$. Then $L$ is of type $(3A_1)''$, and $\SW_L = W(F_4)$. 
We have $V_{(u,-1)} = \x_{2,1} \in \SW_L\wg$. Consider the Levi subgroup $M$ of type
$D_5 + A_1$.  Then $\SW_L' = N_M(L)/L$ corresponds to the subgroup of
$W(F_4)$ of type $C_3$.  The restriction $\x_{2,1}|_{\SW_L'}$ coincides with
the irreducible representation $\x'$ of type $(12;-)$ of $W(C_3)$.  Choose $u' \in M$ such that
$V_{(u',\r')}$ coincides with $\x'$.
Then the Jordan type of $u'$ in $M$ is equal to $((37), (2))$.
We have $A_M(u') \simeq  \BZ_2$, and $\dim Z_M(u') = 9$.  
The $A_G(u)\times A_M(u')$-module $\ve_{u,u'}$ is decomposed as
\begin{equation*}
\tag{5.3.1}  
\ve_{u,u'} = (1\otimes 1) \oplus (-1\otimes -1).
\end{equation*}  
Let $C'$ be the unipotent class in $M$ containing $u'$. We choose a split element
$u' \in C'^F$, and take $u \in C^F$ such that $Y_{u,u'}^F \ne \emptyset$.  
Then by Lemma 4.16, $\g(u_0, \wt\r_0, u, -1) = 1$ (note that $F$ acts trivially on $A_G(u)$
since $A_G(u)$ is abelian). We also have
$\g(1,1,u, 1) = 1$ by (5.3.1).
Hence $u$ is a split element.
\par
In the case where $p = 2$, $L$ is of type $D_4$, and $\SW_L = W(B_3)$.
We consider $V_{(u,-1)} = (-;3) \in \SW_L\wg$.
Let $M$ be a Levi subgroup of type $D_6$.  Then $\SW_L' = W(B_2)$.
Let $C' \subset M$ be the class $((48), (11))$.
Then for $u' \in C'$, we have $A_M(u') \simeq \BZ_2$, and $\dim Z_M(u') = 9$. 
The restriction of $(-;3)$ on $\SW_L'$ coincides with $(-;2) = V_{(u',\r')}$,
and $\ve_{u,u'}$ can be decomposed as in (5.3.1).
Hence by a similar discussion using Lemma 4.16 as above,
one can find a split element $u \in C^F$. 
\par\medskip
(3) \ Assume that $C = E_7(a_2)$.  Then $A_G(u) \simeq \BZ_2$, and $\dim Z_G(u) = 11$. 
Assume that $p \ne 2$. Then $V_{(u,-1)} \in \SW_L\wg$ corresponds to
$\x_{4,1} \in W(F_4)\wg$.
Let $M$ be a Levi subgroup of type $D_5 + A_1$ as in (2).  Then
$\SW_L' = W(C_3)$, and we have $\x_{4,1}|_{\SW_L'} = (3;-) + (2;1)$.
We choose $u' \in M$ of Jordan type $((1117), (2))$.  Then $A_M(u') \simeq \BZ_2$, and
$\dim Z_M(u') = 11$.  We have $V_{(u',-1)} = (2;1)$ in $\SW_L'$, and
$\ve_{u,u'}$ is decomposed as
in (5.3.1).
Hence by Lemma 4.16, a split element $u \in C^F$ exists.
\par
Next assume that $p = 2$ with $C = E_7(a_2)$.
Then $L$ is of type $D_4$, and $\SW_L = W(B_3)$. We have $A_G(u) \simeq \BZ_2$,
and $V_{(u,-1)} = (2;1) \in \SW_L\wg$. Let $M$ be a Levi subgroup of type $D_5 + A_1$.
Then $\SW_L' = W(\wt A_1 + A_1)$. 
Consider $u'$ in $M$ of type $((1126), (11)), (2))$.
(Here we follow the notation of \cite[2.1]{LS2} for describing the
unipotent classes in $SO_{10}$ with $p = 2$; $(1126)$ is a partition of 10,
and $(11)$ denotes $\ve(2) = \ve(6) = 1$.)
Then $A_M(u') \simeq \BZ_2$, and $V_{(u',-1)}$ is contained in the restriction of
$V_{(u,-1)}$. We have $\dim Z_M(u') = 11$, and $\ve_{u,u'}$ is decomposed as
in (5.3.1).
Hence by Lemma 4.16, a split element $u \in C^F$ exists. 
\par\medskip
(4) \ Assume that $C = D_6 + A_1$ with $p \ne 2$. Then $A_G(u) \simeq \BZ_2 \times \BZ_2$,
and $\dim Z_G(u) = 13$. Here $L$ is of type $(3A_1)''$ and $\SW_L = W(F_4)$.
We have $V_{(u,\ve')} = \x_{8,1}$ and $V_{(u,\ve'')} = \x_{1,2}$. 
Let $M$ be a Levi subgroup of type $D_5 + A_1$,  then $\SW_L' = W(C_3)$.
Take $C' \subset M$ of Jordan type $((1135), (2))$.  Then for $u' \in C'$,
we have $A_M(u') \simeq \BZ_2 \times \BZ_2$, and $\dim Z_M(u') = 13$. 
We have a decomposition
\begin{equation*}
\tag{5.3.2}
  \ve_{u,u'} = (1\otimes 1) \oplus (\ve \otimes \ve) \oplus (\ve' \otimes \ve')
               \oplus (\ve'' \otimes \ve''). 
\end{equation*}
Then by Lemma 4.16, a split element $u \in C^F$ exits.
\par\medskip
(5) \ Assume that $C = D_6(a_1) + A_1$ with $p \ne 2$.
Note that this is the excluded case (4.5.1). As noticed below,
Lemma 4.16 can not be applied for this case. 
\par
$A_G(u) \simeq \BZ_2 \times \BZ_2$, and $\dim Z_G(u) = 17$.   
Here $L$ is of type $(3A_1)''$, and $\SW_L = W(F_4)$.
We have $V_{(u,\ve')} = \x_{4,2}$ and $V_{(u,\ve'')} = x_{2,3}$ in $W(F_4)$.
(Note that the table of the generalized Springer correspondence for $G = E_7$ in
\cite{Sp2} has an error, which was corrected by Geck; in \cite{Sp2} it
was given as $V_{(u,\ve'')} = \x_{8,3}$ and $V_{(u_1,-1)} = \x_{2,3}$ for
$u_1 \in C_1 = D_5 + A_1$ in $M$.  This should be switched so that
$V_{(u,\ve'')} = \x_{2,3}$ and $V_{(u_1, -1)} = \x_{8,3}$.) 
\par
Let $M$ be a Levi subgroup of type $D_6$. Then $\SW_L' = W(B_3)$.
We have $\x_{4,2}|_{W(B_3)} = (1;2) + (-;3)$, and $\x_{2,3}|_{W(B_3)} = (12;-)$.
In the Springer correspondence in $M$, $(1;2)$ occurs in the class $C_1' = (66;+)$,
$(-;3)$ occurs in $C_2' = (39)$, and $(12;-)$ occurs in $C_3' = (1227)$.
Take $u_1'\in C_1', u_2' \in C_2', u_3' \in C_3'$.  Then for $u' = u_1', u_2'$ or $u_3'$,
we have $A_M(u')\simeq \BZ_2$,
and $\dim Z_M(u_1') = 13, \dim Z_M(u_2') = 9, \dim Z_M(u_3') = 15$.
We have $V_{(u_1',-1)}= (1;2)$, $V_{(u_2',-1)} = (-;3)$ and $V_{(u'_3, -1)} = (21;-)$.
Moreover, 
\begin{equation*}
\ve_{u,u_1'} = (1\otimes 1) \oplus (\ve'\otimes -1), 
\quad
\ve_{u,u_2'} = (1\otimes 1) \oplus (\ve'\otimes -1),
\end{equation*}
\begin{equation*}
\ve_{u,u_3'} = (1\otimes 1) \oplus (\ve''\otimes -1).
\end{equation*}  
But since $\dim Z_G(u) > \dim Z_M(u')$, Lemma 4.16 cannot be applied.
Actually, for any unipotent class $C'$ of any Levi subgroup $M$ such that
$V_{(u',\r')}$ appears in the restriction of $V_{(u,\r)}$, we have
$\dim Z_G(u) > \dim Z_M(u')$. Hence there does not exist a pair $(u',\r')$ such that
Lemma 4.16 can be applied.
\par
Later in 6.3 (3), (6), we discuss the case where Lemma 4.16 can not be applied,
but still the existence of split elements is proved. 
The discussion there does not hold in the present case, since we need to compare
$A_G(u) \times A_M(u_1')$-module $\ve_{u,u_1'}$ and $A_G(u) \times A_M(u_3')$-module
$\ve_{u,u_3'}$. We can only show that there exists $u \in C^F$ such that
\begin{align*}
  \g(1,1,u,1) = 1, \quad \g(1,1,u,\ve) = 1, \quad \g(u_0, \r_0, u, \ve') = 1,
  \quad \g(u_0, \r_0, u, \ve'') = \pm 1. 
\end{align*}  

\par\medskip
From the above discussion, for $C = E_7, E_7(a_1), E_7(a_2), D_6 + A_1, D_6(a_1) + A_1$
(possibly except $D_6(a_1) + A_1$), a split element $u \in C^F$ exists. 

\para{5.4.}
We consider the case where $C = D_4(a_1)$.  In this case, $A_G(u) \simeq S_3$
for $u \in C$.  Since $|\SM_G(C)| = 1$, there exists a split element $u \in C^F$
by (4.7.2).

\para{5.5.}
Finally consider the case where $C = D_6(a_2) + A_1$.  
In this case, $A_G(u) = S_3 \times \BZ_{(2,p-1)}$. 
Then $\SM_G(C) = \{ (T, \{1\}, 1), (L, C_0, \SE_0)\}$ 
with $L$ of type $(3A_1)''$ if $p \ne 2$, and 
$\SM_G(C) = \{ (T, \{ 1\}, \Ql)\}$ if $p = 2$. 
Thus by (4.7.2), we have only to consider the case where $p \ne 2$.
\par
In this case, it is also possible to find a suitable Levi subgroup $M$
and a unipotent class $C'$ in $M$, with $u' \in C$, such that the condition
of Lemma 4.16 is satisfied for $\ve_{u,u'}$.  But since $A_G(u)$ is not abelian,
if we choose a split element $u' \in C'^F$, and take an element $u \in u'U_Q$
such that $Y_{u,u'}^F \ne \emptyset$, one cannot check whether $F$ acts trivially
on $A_G(u)$, and the discussion as in 5.3 breaks.
Thus in this case, we take the other way round; first we choose $u \in C^F$
such that $F$ acts trivially on $A_G(u)$, then define $u' \in C'^F$ from $u$
for a certain Levi subgroup $M$.  We check that $u'$ is a split element in $C'^F$,
and verify that $\ve_{u,u'}$ satisfies the condition in Lemma 4.16.
This is a similar way as Beynon-Spaltenstein determined $\g(1,1,u, \r)$ in the
case where $C$ is the exceptional class in $G = E_8$. 
\par
We choose $u \in C^F$ following \cite[Table 2]{M2} as follows. 

\begin{align*}
\tag{5.5.1}
u &= x_{14}(1)x_{15}(1)x_{27}(1)x_{17}(1)x_{18}(1)x_{19}(1)x_{23}(1) \\
  &= x_{\a}(1)x_{3+2}(1)x_{\b}(1)x_{4+1}(1)x_{5-2}(1)x_{6-3}(1)x_{5-1}(1),
\end{align*}
where $\a = -1-2+3-4-5-6-7+8, \b = -1-2-3-4+5-6-7+8$. 
Here the first expression is due to Mizuno ($u = y_{47}$ in Lemma 21 in \cite{M2}), 
and the second one is given by the root subgroup expression $x_{\a}(1)$ for 
$\a \in \vD^+$ in 4.2,  
where we use a convention such as
$\frac{1}{2}(\ve_1-\ve_2-\cdots-\ve_7 + \ve_8) = 1-2- \cdots - 7+8$, $\ve_1 + \ve_2 = 1+2$,
and so on.
\par
We have $A_G(u) \simeq S_3 \times \BZ_2$, and 
it is checked by [M2] that $F$ acts trivially on $A_G(u)$.
Let $Q = MU_Q$ be a parabolic subgroup of $G$ such that
$M$ is a Levi subgroup of type $D_6$. 
Then $M$ contains $L$ of type $(3A_1)''$.   
Put 
\begin{equation*}
\tag{5.5.2}
u' = x_{3+2}(1)x_{4+1}(1)x_{5-2}(1)x_{6-3}(1)x_{5-1}(1).
\end{equation*}

Then $u' \in M^F$, and $u$ is written as $u \in u'U_Q$. 
Let $C'$ be the unipotent class in $M$ containing $u'$. 

\par
In the discussion below, we follow the notation in Section 3. 
Evaluating $u'$ in the natural 12 dimensional representation of $M$ 
through the computation by CHEVIE, 
we obtain a matrix of Jordan type $(2244)$. Now $M/Z^0M$ is a quotient of $\wt M = \Spin_{12}$,
and for the element $v \in \Spin_{12}$ 
such that $\b(v) = \ol v \in SO_{12}$ has Jordan type $(2244)$, we have
$A_{\wt M }(v) \simeq \BZ_2$ by \cite[14.3]{L1} (since $I= \emptyset$ in the notation 
there). It follows that 
$A_M(u') \simeq \BZ_2$ or $A_M(u') = \{1\}$.  
We consider the variety $Y_{u, u'}$, and the set $X_{u, u'}$.
We show a lemma.

\begin{lem}  
Under the notation above, 
\begin{enumerate}
\item  \ $A_G(u) \times A_M(u')$ acts transitively on $X_{u, u'}$. 
\item \ $A_M(u') \simeq \BZ_2$.
\item \ The decomposition of $\ve_{u,u'}$ is given by
\begin{align*}
\ve_{u,u'} &=  (1 \otimes 1) 
       \oplus (\th \otimes 1) \oplus (-1 \otimes -1)
          \oplus (-\th \otimes-1) 
\end{align*}
for $A_G(u)\wg = (S_3 \times \BZ_2)\wg = \{ 1, \th, \ve, -1, -\th,-\ve \}$
and $A_M(u')\wg = \{ 1, -1\}$.
\end{enumerate}
\end{lem}

\begin{proof}
We have $\dim Z_G(u) = 21$ by \cite{M2}, and $\dim Z_M(u') = 21$.
It follows that $\dim \CQ_{u,C'} = 0$. 
We have $V_{(u, 1)} = 315_7$.  On the other hand, 
for $W(D_6)$, we have $V_{(u',1)} = (\la; \la)_+$ or $(\la;\la)_-$, 
where $\la = (12)$ is a partition of $3$, and $\pm $ correspond to the 
irreducible representation of $W(D_6)$ appearing in the restriction of 
the irreducible representation $(\la;\la) \in W(B_6)\wg$.   
By a direct computation, we see that the condition
$\lp V_{(u,1)}, V_{(u',1)} \rp_{\, W(D_6)} \ne 0$ 
determines $V_{(u',1)}$ uniquely, and we have
$\lp V_{(u,1)}, V_{(u',1)}\rp_{\, W(D_6)} = 1$. 
Hence the multiplicity of 
$1\otimes 1$ in $\ve_{u, u'}$ is equal to 1, and as in the proof of Lemma 4.16,
we see that $A_G(u) \times A_M(u')$ acts transitively on $X_{u, u'}$.
Hence (i) holds. 
\par
Next we show (ii).  By (i), $Z_G(u)$ acts transitively on $\CQ_{u, C'}$. 
Since the stabilizer of $Q$ in $Z_G(u)$ coincides with $Z_Q(u)$, we have
$\CQ_{u,C'} \simeq Z_G(u)/Z_Q(u)$ as finite sets. 
By \cite[Lemma 21]{M2}, $Z_G(u) \subset B \cup BwB$, where $w = s_1s_2s_6$.  
Since $BwB \cap M = \emptyset$, we have $Z_Q(u) = Z_B(u)$. 
By a direct computation, we see that $Z_B(u) = Z_GZ_U(u)$.
By [M2], $Z_G(u) \simeq S_3 \times \BZ_2$ and 
$Z_G^0(u) = Z_U(u)$. 
Hence $Z_Q(u) = Z_GZ^0_G(u)$, and we have
$Z_G(u)/Z_Q(u) \simeq S_3$.  It follows, that 
\begin{equation*}
\tag{5.6.1}
\dim \ve_{u, u'} = |Z_G(u)Q| = |S_3| = 6. 
\end{equation*}

Now by assuming that $A_M(u') = \{1 \}$, we will have a contradiction. 
By \cite{Sp2}, the generalized Springer correspondence for $C$ is 
given as follows; 
\begin{align*}
\tag{5.6.2}
V_{(u,\r)} = \begin{cases} 315_7 &\quad\text{ for } \r = 1,  \\
                          280_9  &\quad\text{ for } \r = \th, \\ 
                          35_{13} &\quad\text{ for } \r = \ve, \\
             \end{cases}  \qquad 
V_{(u,\r)} = \begin{cases}            
               \x_{12}  &\quad\text{ for } \r = -1, \\
               \x_{6,2} & \quad\text{ for } \r = -\th, 
              \end{cases}
\end{align*}                 
where the former is related to $\SW_T = W(E_7)$, 
and the latter is related to $\SW_L = W(F_4)$.
Note that $(u, -\ve)$ is a cuspidal pair for $G$. 
\par
Suppose that $A_M(u') = \{1\}$.  Then the decomposition of $\ve_{u, u'}$
is only concerned with the Springer correspondence, and we have

\begin{equation*}
\tag{5.6.3}
\ve_{u,u'} = \lp V_{(u,1)}, V_{(u',1)}\rp\, (1 \otimes 1)
              \oplus \lp V_{(u, \th)}, V_{(u',1)}\rp \,(\th\otimes 1)
              \oplus \lp V_{(u, \ve)}, V_{(u',1)}\rp \,(\ve \otimes1). 
\end{equation*}
Since 
$\lp 315_7, V_{(u',1)}\rp = 1, \ \lp 280_9, V_{(u',1)}\rp = 1, 
\ \lp 35_{13}, V_{(u',1)}\rp = 0$, 
we have $\dim \ve_{u, u'} = \dim (1 \oplus \th) = 3$. This contradicts (5.6.1), 
and we obtain $A_M(u') \simeq \BZ_2$. (ii) is proved.
\par
Finally we show (iii). 
The subgroup $\SW'_L = N_M(L)/L$ of $\SW_L$ is of type $B_3$. 
Hence the generalized Springer correspondence for $u' \in C'$ is given by 
\begin{equation*}
V_{(u',1)} \in W(D_6)\wg, \qquad  V_{(u',-1)} \in W(B_3)\wg, 
\end{equation*}
where $\r' = \pm 1 \in A_M(u')\wg$. 
Now (iii) follows easily from the restriction formula by computing the
multiplicity of $V_{(u',\r')}$ in the restriction of $V_{(u,\r)}$.
\end{proof}

\para{5.7.}
Note that the center of the spin group $\Spin_{12}$ 
is isomorphic to $\BZ_2 \times \BZ_2$, generated by $\w$ and $\ve$
(see Section 3). The half spin group $\frac{1}{2}\Spin_{12}$ 
is defined as the quotient group $\Spin_{12}/\lp \w\rp$. 
As $A_M(u') \simeq \BZ_2$ by Lemma 5.6, the almost simple group 
$\ol M = M/Z^0M$ is isomorphic to 
$\Spin_{12}$ or $\frac{1}{2}\Spin_{12}$ (since $A_M(u') \simeq A_{\ol M }(u')$, and 
the centralizer of the element with Jordan type $(2244)$ in $SO_{12}$ is connected).   
On the other hand, since the center of $\ol M$ is isomorphic to $\BZ_2$, we see
that $\ol M \simeq \frac{1}{2}\Spin_{12}$. 
\par
Since $A_M(u') = \BZ_2$, $C'^F$ splits into two $M^F$-orbits.
Exactly one of them is a split class in $C'^F$.  
In the case where $u'$ is split, we put $u_1' = u'$ and $u_1 = u$.
Assume that $u'$ is not split.  In this case, 
following the discussion in Section 3, we modify $u$ and $u'$ as follows.   
Now the center $Z_G$ of $G$ is $\BZ_2$. We denote by $z$ 
the generator of $Z_G$. Then $z \in Z_M(u')$, and its image 
gives the generator of $A_M(u')$.  Choose $h \in M$ such that 
$h\iv F(h) = z$, and put $u_1' = hu'h\iv$.  Then $u'$ and $u_1'$ 
give representatives of $M^F$-orbits in $C'^F$, and so $u_1' \in C'^F$
is a split element. We define $u_1$ by $u_1 = huh\iv$. Since $z \in Z_G(u)$, 
we have $u_1 \in C^F$. Since $z$ is a central element, $F$ acts trivially 
on $A_G(u_1)$.  
Define $Y_{u_1, u_1'}$ similarly to $Y_{u,u'}$.  Then $g \mapsto hgh\iv$ 
gives an isomorphism $Y_{u, u'} \isom Y_{u_1, u_1'}$, which is compatible with 
the actions of $A_G(u) \times A_M(u')$ and $A_G(u_1) \times A_M(u_1')$. 
The following result proves the existence of split elements for
$C = D_6(a_2) + A_1$.  

\begin{prop}  
Assume that $p \ne 2$, and $C = D_6(a_2) + A_1$.
Let $u_1, u_1'$ be as in 5.7, which we denote by
$u,u'$. 
Then 
\begin{enumerate}
\item 
$F$ acts trivially on $A_G(u)$ and $A_M(u')$.  Furthermore,
$F$ acts trivially on $X_{u, u'}$.   
\item 
$u \in C^F$ is a split element. 
\end{enumerate}
\end{prop}

\begin{proof}
By the construction of $u$ and $u'$, $F$ acts trivially on 
$A_G(u)$ and on $A_M(u')$.  
By a similar discussion as in the proof of Lemma 4.16,
Lemma 5.6 implies that $F$ acts trivially on $X_{u,u'}$.
Thus (i) holds. 
\par
We show (ii). 
By our construction, $u'$ is a split element in $C'^F$. 
By Lemma 4.16, we see that 
$\g(u_0, \wt\r_0, u, \r) = 1$ for $(u,\r)$ belonging to 
$(L, C_0, \SE_0)$ or to $(T, \{1\},1)$, whenever $\r\otimes \r'^*$
appears in $\ve_{u, u'}$.
Hence by Lemma 5.6 (iii), we have
\begin{equation*}
\g(1,1, u, 1) = \g(1,1,u,\th) = 1, \quad
\g(u_0, \wt\r_0, u, -1) = \g(u_0,\wt\r_0, u, -\th) = 1. 
\end{equation*}

It remains to show that $\g(1,1, u, \ve) = 1$.
But since this is only related to the Springer correspondence,
we may replace $G$ by its adjoint group $\ol G$, and $u$ by its image $\ol u \in \ol G$.
Then $F$ acts trivially on $A_{\ol G}(\ol u) \simeq S_3$, and $\ol u$ belongs to the
unique $\ol G^F$-orbit satisfying this property.
It is proved in \cite{BS} that $\ol u$ is a split element in the sense of
Beynon-Spaltenstein. Hence by (4.7.1), we have $\g(1,1,u, \ve) = 1$.
Thus $u \in C^F$ is a split element.
\end{proof}

\para{5.9.}
Summing up the above discussion, we have proved that Theorem 4.6 holds for 
$G = E_7$. 

\par\bigskip
\section{ The case $E_8$}

\para{6.1.}
In this section, we assume that $G$ is of type $E_8$. 
In this case, $A_G(u)$ is isomorphic to  $D_8$ for $C = D_8(a_1)$ (the case $p = 2$),
to $S_5$ for $C = 2A_4$, 
to $S_3 \times Z_{(2,p)}$ for $C = E_7(a_2) + A_1$, and to $S_3$ for 
$C = A_8, D_8(a_3)$ (the case $p \ne 3$), $A_5 + A_2, D_4(a_1) + A_1, D_4(a_1)$.   
$A_G(u)$ is abelian for all other cases. 
The set $\SM_G$ is given, except for cuspidal ones, by 
\begin{align*}
&\{ (T, \{1\},1), (L, C_0, \SE_0), (L', C_0', \SE_0')\}  
    \quad \text{ for $L = D_4, L' = E_7$ \text{ with } $p = 2$}, \\
&\{ (T,\{1\},1), (L, C_0, \SE_0)\} \quad \text{ for $L = E_6$ with $p = 3$}, \\
&\{ (T,\{1\},1)\} \quad \text{ with $p \ne 2,3$},
\end{align*}
where in the first case, there exist two $\SE_0'$ for a given class $C_0'$ of
$L' = E_7$, and in the second case, there exist two $\SE_0$ for a given class
$C_0$ of $L = E_6$. 
\par
In the case where $p$ is good, since $|\SM_G(C)| = 1$ for any class $C$, the results
of Beynon-Spaltenstein can be applied, and the existence of split elements
is shown by (4.7.2).  

\par
Consider the class $C$ in the list
\begin{align*}
\tag{6.1.1}  
&E_8, \quad E_8(a_1), \quad E_8(a_2), \quad E_7 + A_1\ (p = 2), \quad  E_7(a_1) + A_1 \ (p = 2), \\
&D_8(a_1) \ (p = 2), \quad  E_7(a_2) + A_1, \quad  A_8, \quad D_8(a_3) \ (p \ne 3),
   \quad 2A_4, \\   
&A_5 + A_2, \quad D_4(a_1) + A_1, \quad D_4(a_1).
\end{align*}  
In those cases, Lemma 4.13 cannot be applied; if
$C = E_8, E_8(a_1), E_8(a_2), E_7 + A_1, E_7(a_1) + A_1$,
there does not exist a proper Levi subgroup $M$ such that $M \cap C \ne \emptyset$,
and $A_G(u)$ is not abelian for the latter $C$. 

\para{6.2.}
Let $C$ be the class which is not in the list (6.1.1).
First assume that $p = 2$.  In this case $L$ is of type $D_4$ and $L'$ is of type $E_7$.  
If $|\SM_G(C)| = 2$, then
$C$ is in the list

\begin{align*}
\tag{6.2.1}
&E_7(a_1), \quad D_7, \quad E_7(a_2), \quad E_6 + A_1, \quad (D_7(a_1))_2, \\
&E_6, \quad D_6, \quad  (D_5 + A_2)_2, \quad D_6(a_1), \quad  D_5 + A_1, \\
&D_5, \quad (D_4 + A_2)_2, \quad
   D_4 + A_1, \quad D_4. 
\end{align*}
In those cases, one can find a proper subgroup
$M \supset L$ such that $M \cap C \ne \emptyset$, hence
Lemma 4.13 can be applied. If $|\SM_G(C)| = 1$, Lemma 4.13 can be applied since
there is no restriction on $M$.
If $|\SM_G(C)|\ge 3$, then $|\SM_G(C)| = 4$ for $C = E_7$ with $L = D_4$ and $L' = E_7$.
But in this case (an extension of) Lemma 4.13 can be applied for $M = E_7$
since $T \subset L \subset L'$ (see Remark 4.14).
\par
Next assume that $p = 3$. 
If $|\SM_G(C)| = 3$, then $C = E_7, E_6 + A_1$ or $E_6$  with $L = E_6$,
and $\SW_L= W(G_2)$.  In those cases, Lemma 4.13 can be applied for
$M  = E_7 \supset L = E_6$. 
For other $C$, Lemma 4.13 can be applied since $|\SM_G(C)| = 1$.
\par
Finally assume that $p \ne 2,3$.   In this case we have $|\SM_G(C)| = 1$ for any class $C$,
and Lemma 4.13 can be applied without referring $L$. 

\para{6.3.}
We consider $C$ in the list in (6.1.1).
If $C$ is one of $A_8, D_8(a_3) \ (p \ne 3), 2A_4, A_5 + A_2,D_4(a_1) + A_1, D_4(a_1)$,
then $A_G(u) \simeq S_3$ or $S_5$, and, $|\SM_G(C)| = 1$.
In those cases, (4.7.2) can be applied,
and there exist split elements for $C$.
\par
We consider the case where $C = E_8, E_8(a_1), E_8(a_2), E_7 + A_1$, or $E_7(a_1) + A_1$. 
Since we may assume that $|\SM_G(C)| \ne 1$, we have only to consider 
the case where $C = E_8\ (p = 2,3), E_8(a_1) \ (p = 2,3)$,  $E_8(a_2) \ (p = 2)$
or $E_7 + A_1 \ (p =2,3)$, $E_7(a_1) + A_1 \ (p = 2)$.

\par\medskip
(1) \ First assume that $C = E_8$. Then $C$ is a regular unipotent class in $G$,
and $A_G(u) \simeq \BZ_{(60, p^2)}$ for $u \in C$.
Here $C = \Ind_T^G1$.  Let $M$ be a Levi subgroup
of type $E_7$, and put $C' = \Ind_T^M1$. Then $C'$ is a regular unipotent class in $M$.
A similar discussion as in the case of 5.3 (1), by using Lemma 4.9, shows that
splits elements exist in $C^F$.

\par\medskip
(2) \ Assume that $C = E_8(a_1)$ with $p = 2$. Then $A_G(u) \simeq \BZ_4$, and
the characters of order 4 correspond to cuspidal pairs. Hence we need to consider
$-1 \in A_G(u)\wg$, where $(u,-1)$ belongs to the series $(L, C_0, \SE_0)$
with $L$ of type $D_4$. Let $M$ be of type $E_7$, and $C' = E_7(a_1)$.
Then for $u' \in C'$, we have $A_M(u') \simeq \BZ_2$, and $\dim Z_G(u) = \dim Z_M(u') = 10$.
We have
\begin{equation*}
\tag{6.3.1}  
\ve_{u,u'} = (1\otimes 1) \oplus (-1\otimes -1). 
\end{equation*}  
By Lemma 4.16, there exists $u \in C^F$ such that
$\g(1,1, u, 1) = 1, \g(u_0, \wt\r_0, u, -1) = 1$. 
Thus $u$ is a split element in $C^F$. 

\par\medskip
(3) \ Assume that $C = E_8(a_1)$ with $p = 3$.
We have $A_G(u) \simeq \BZ_3$,
and for a non-trivial character $\z$ of $\BZ_3$, $(u, \z)$ belongs to the
series $(L,C_0,\SE_0)$ with $L$ of type $E_6$.
In this case, one cannot find $C' \subset M$ such that
Lemma 4.16 can be applied. However, by modifying the discussion in the proof of
Lemma 4.16, we can
argue as follows. Let  $M$ be of type $E_7$, and $C'= E_7$.
Then $A_M(u') \simeq \BZ_3$ for $u' \in C'$. $\ve_{u,u'}$ can be decomposed as
\begin{equation*}
\tag{6.3.2}  
\ve_{u,u'} = (1\otimes 1) \oplus (\z \otimes \z^*) \oplus (\z^2\otimes \z^{2*}).
\end{equation*}
But since
$\dim Z_G(u) = 10 >  \dim Z_M(u') = 8$, 
Lemma 4.16 cannot be applied. 
\par
Now $F$ acts on the set $X_{u,u'}$ as a permutation.
By (6.3.2), $A_G(u) \times A_M(u')$ acts transitively on $X_{u,u'}$.
If $X_{u,u'}^F \ne \emptyset$, since $F$ acts trivially on $A_G(u)$ and $A_M(u')$,
all the elements in $X_{u,u'}$ are $F$-stable, and $F$ acts trivially on $\ve_{u,u'}$.
Then the conclusion of Lemma 4.16 holds, and $u \in C^F$ is a split element.
Thus we assume that $X_{u,u'}^F = \emptyset$. 
By (6.3.2), $X_{u,u'}$ is a three point set, and $F$ acts on $X_{u,u'} = \{ x_1, x_2, x_3\}$
as a permutation $F: x_1 \mapsto x_2 \mapsto x_3 \mapsto x_1$.
On the other hand, since $\ve_{u,u'}$ is a regular $A_G(u)$-module by (6.3.2), 
there exists a generator $a$ of $A_G(u) \simeq \BZ_3$ such that the action of $a$ on $X_{u,u'}$
coincides with that of $F$ on it.
Let $\w$ be a cubic root of unity in $\Ql^*$ such that $\z(a) = \w$.
Then $\z \otimes \z^*$ and $\z^2 \otimes \z^{2*}$ are both $F$-stable, and
the action of $F$ on $\z\otimes\z^*$ (resp. on $\z^2\otimes \z^{2*}$) is given by a scalar
multiplication by $\w$ (resp. $\w^2$). It follows, by a similar discussion as in the proof of
Lemma 4.16, that
\begin{equation*}
\g(1,1,u,1) = 1, \quad \g(u_0, \r_0, u, \z) = \w, \quad \g(u_0, \r_0, u, \z^2) = \w^2.
\end{equation*}  

We consider $u_z \in C^F$ obtained from $u$ by twisting with $z \in A_G(u)$.
Here we choose $z$ so that $\z(z) = \w\iv$. Then
\begin{align*}
\g(u_0, \r_0, u_z, \z) &= \z(z)\g(u_0, \r_0, u, \z) = \z(z)\w = 1, \\
\g(u_0, \r_0, u_z, \z^2) &= \z^2(z)\g(u_0, \r_0, u, \z^2) = \z^2(z)\w^2 = 1.  
\end{align*}
Since $\g(1,1,u_z,1) = \g(1,1,u,1) = 1$, $u_z$ is a split element in $C^F$. 

\par\medskip
(4) \ Assume that $C = E_8(a_2)$ with $p = 2$. Then $A_G(u) \simeq \BZ_2$,
and for $-1 \in A_G(u)\wg$, $(u, -1)$ belongs to the series $(L, C_0, \SE_0)$
with $L$ of type $D_4$. .
We consider $M$ of type $E_7$, and let $C' = E_7(a_2)$.  For $u' \in C'$,
we have $A_M(u') \simeq \BZ_2$, and $\dim Z_G(u) = \dim Z_M(u') = 12$.  
$\ve_{u,u'}$ is decomposed in a similar way as in (6.3.1).
Hence Lemma 4.16 can be applied to show the existence of
split elements in $C^F$. 

\par\medskip
(5) \ Assume that $C = E_7 + A_1$.  First assume that $p = 2$.
Then $A_G(u) \simeq \BZ_2 \times \BZ_2$.  $L$ is of type $D_4$,
and $\SW_L = W(F_4)$. We have $V_{(u,\ve')} = \x_{8,1}$ and
$V_{(u,\ve'')} = \x_{1,2}$. Let $M$ be a Levi subgroup of type $D_7$.
Then $\SW_L' = W(C_3)$. Note that these two characters $\x_{8,1}$ and $\x_{1,2}$
in $W(F_4)$ are completely the same as those appeared in the case (4) in 5.3
for $G = E_7$.
Let $u' \in C' \subset M$ of type $((1148), (11))$. 
We have $\dim Z_G(u) = \dim Z_M(u') = 14$, and
$A_M(u') \simeq \BZ_2 \times \BZ_2$.
Then $\ve_{u,u'}$ is decomposed as
\begin{equation*}
\tag{6.3.3}  
  \ve_{u,u'} = (1\otimes 1) \oplus (\ve \otimes \ve)\oplus (\ve' \otimes \ve')
                  \oplus (\ve'' \otimes \ve'').
\end{equation*}  
Hence a split element exists by Lemma 4.16.
\par
Next assume that $p = 3$.  Then $A_G(u) \simeq \BZ_2 \times \BZ_3$. 
$L$ is of type $E_6$, and $\SW_L = W(G_2)$. We have $V_{(u,\z)} = \th'$,
where $\z$ is an irreducible character of $\BZ_3$ of order 3.
Note that $(u,-\z), (u, -\z^2)$ are cuspidal pairs in $G$.
We consider a Levi subgroup $M$ of type $E_7$, then $\SW_L' = W(A_1)$.
Take $u' \in C' \subset M$
such that $C' = E_6$. Then $A_M(u') \simeq \BZ_3$, and we have
$\dim Z_G(u) = \dim Z_M(u') = 14$.  $\ve_{u,u'}$ is decomposed as

\begin{equation*}
\tag{6.3.4}
\ve_{u,u'} = (1\otimes 1) \oplus (\z\otimes \z^*) \oplus (\z^2 \otimes \z^{*2}). 
\end{equation*}  
It follows, by Lemma 4.16, that there exists $u \in C^F$ such that
\begin{equation*}
\g(1,1, u, 1) = 1, \quad \g(u_0, \wt\r_0, u, \z) = 1, \quad \g(u_0, \wt\r_0, u, \z^2) = 1.
\end{equation*}

Now assume that $\g(1,1, u, -1) = \nu$ for a root of unity $\nu$. It follows from the
computation of Green functions by \cite{Lu}, there exists $u_1 \in C^F$ such that
$\g(1,1, u_1, 1) = 1$ and that $\g(1,1,u_1, -1) = 1$.  $u_1$ is written
as $u_1 = u_z$ for some $z \in A_G(u)$ (a twisted element by $z$).
Since $\g(1,1, u_1, -1) = -1(z)\g(1,1, u,-1)$, 
we have $\nu = \pm 1$.  If $\nu = 1$, put $v = u$, then $v \in C^F$ is a split
element.  If $\nu = -1$, put $v = u_{z'}$ for $z' = (-1,1) \in \BZ_2 \times \BZ_2$.
We have $\g(1,1, u_{z'}, -1) = -1(z')\g(1,1,u,-1) = 1$, and since 
$\g(u_0,\wt\r_0, u_{z'}, \r) = \r(z')\g(u_0, \wt\r_0, u,\r)$ for $\r = \z$ or $\z^2$,
we have $\g(u_0, \wt\r_0, u_{z'}, \r) = 1$.  Thus $v = u_{z'}$ is a split element.
In any case, a split element $v \in C^F$ exists. 

\par\medskip
(6) \ Assume that $C = E_7(a_1) + A_1$ with $p = 2$.
We have $A_G(u) \simeq \BZ_2$. $L$ is of type $D_4$, and $\SW_L = W(F_4)$.
In this case, there does not exist a class $C' \subset M$
such that Lemma 4.16 can be applied. However, we can argue as in the case (3). 
We have $V_{(u,-1)} = \x_{4,2}$.
Let $M$ be a Levi subgroup of type $E_7$.
Then $\SW_L' = W(B_3)$.
We have $\x_{4,2}|_{W(B_3)} = (1;2) + (-;3)$. In the Springer correspondence,
$(1;2)$ corresponds to the class $C' = E_6$ in $M $. Take $u' \in C'$.
Then $A_M(u') \simeq \BZ_2$.
We have $V_{(u', -1)} = (1;2)$.
The $A_G'u) \times A_M(u')$-module $\ve_{u,u'}$ is decomposed as 
\begin{equation*}
\tag{6.3.5}  
\ve_{u,u'} = (1\otimes 1) \oplus (-1 \otimes -1).
\end{equation*}  
In this case, $\dim Z_M(u') = 14 < \dim Z_G(u) = 18$, and
Lemma 4.16 can not be applied.
But $X_{u,u'}$ is a two point set, and $F$ acts on it as a permutation.
If $F$ acts trivially on $X_{u,u'}$, then $u \in C^F$ is a split element.
If $X_{u,u'}^F = \emptyset$, then the action of $F$ on $X_{u,u'}$ coincides
with the action of $A_G(u) \simeq \BZ_2$ on $X_{u,u'}$. Thus by a similar
discussion as in (3), a twisted element $u_z \in C^F$ ($z$ is a generator of $A_G(u)$)
gives a split element in $C^F$. 

\para{6.4.}
The remaining cases are $C = D_8(a_1)$ and  $C = E_7(a_2) + A_1$.
Since we may assume that $|\SM_G(C)| \ne 1$,  it is enough to consider
the case where $p = 2$. 
For these cases, we consider $Q = MU_Q$, where $M$ is a Levi subgroup of type $E_7$.
In the discussion below, we use a notation of root systems as given in 4.2.
In particular, by comparing $\vD^+$ and $\vD^+_7$, we see that $U_Q$ is generated
by $x_{\a}(t)$ where $\a$ is contained in

\begin{align*}
\tag{6.4.1}
\begin{cases}
  7 \pm  1, \ 7 \pm 2, \ 7 \pm 3, \ 7 \pm 4, \ 7 \pm 5, \ 7 \pm 6, \\
  8 \pm 1, \ 8 \pm 2, \ 8 \pm 3, \ 8 \pm 4, \ 8 \pm 5, \ 8 \pm 6, \ 8 + 7, \\
  \pm 1 \pm 2 \pm 3 \pm 4 \pm 5 \pm 6 + 7 + 8, \quad\text{ with $-$ : even. } 
\end{cases}
\end{align*}

\par
First assume that $C = D_8(a_1)$ with $p = 2$. Following Mizuno's table \cite{M2},
we take $u = z_{44} \in C$ as 

\begin{align*}
\tag{6.4.2}
  u &= x_8(1)x_{15}(1)x_{16}(1)x_{17}(1)x_{18}(1)x_{24}(1)x_{13}(1)x_{102}(1) \\ 
    &= x_{\a}(1)x_{3+2}(1)x_{4-1}(1)x_{4+1}(1)x_{5-2}(1)x_{5+1}(1)
                       x_{6-4}(1)x_{7-5}(1), 
\end{align*}
where $\a = -1+2-3-4-5-6-7+8$.
Here $A_G(u) \simeq  D_8$, and we have $A_G(u)\wg = \{ 1, \ve, \ve', \ve'', \th\}$ (see 4.4).
Note that $(u,\ve)$ is a cuspidal pair for $G$. 
Let $C'$ be the class of $M$ such that
$C' = D_5 + A_1$, and take $u' \in C'$.  Then we have $A_M(u') \simeq \BZ_2$,
and $\ve_{u,u'}$ is decomposed as
\begin{equation*}
\tag{6.4.3}  
\ve_{u,u'} = (1\otimes 1) \oplus (\ve'\otimes 1) \oplus (\th \otimes -1).
\end{equation*}
In particular, $\th \in A_G(u)\wg$ appears in (6.5.3).
We have $\dim Z_G(u) = \dim Z_M(u') = 20$. 
By the restriction formula, (6.5.3) implies that there exists $g \in G$
such that $v = g\iv ug \in Q$  and that the projection of $v$ on $M$ coincides with $u'$.  
\par
For  a given $u \in C^F$ as in (6.5.2), we determine $g \in G^F$ and $u' \in C'^F$ explicitly.
Let $v$ be the conjugate of $u$ by $g = x_{2+1}(1)x_{4-3}(1)\w_{7-6} \in G^F$.
Then we have 
\begin{align*}
  v &= gug\iv =  x_{\a}(1)x_{3+2}(1)x_{4-1}(1)x_{4+1}(1)
                 x_{5-2}(1)x_{6-5}(1)x_{7-4}(1)x_{7-3}(1).
\end{align*}  
Thus $v \in Q$, and let $u'$ be the projection of $v$ to $M$.  We have

\begin{align*}
\tag{6.4.4}  
  u' &= x_{\a}(1)x_{3+2}(1)x_{4-1}(1)x_{4+1}(1)x_{5-2}(1)x_{6-5}(1)  \\
     &= x_{\a}(1)x_{3+2}(1)x_{4+1}(1)x_{5-2}(1)x_{6-5}(1) \times x_{4-1}(1).
\end{align*}  
Since $(\a, 5-2) = -1, (\a, 4+1) = -1, (\a, 3+2) = 0, (\a, 6-5) = 0$, 
we see that $\{ 3+2, 6-5, 5-2, \a, 4+1 \}$ gives a subsystem of
$\vD(E_7)$ of type $D_5$,
\begin{align*}
&\xymatrix@C=7pt@ R=7pt@ M =3pt{
                      &    (3+2) \ar@{-}[d]  &   &      \\
    (6-5) \ar@{-}[r]  &  (5-2) \ar@{-}[r]   &   \a \ar@{-}[r]  &  (4+1)
} 
\end{align*}
and $4-1$ is orthogonal to all the above roots. 
Thus $u'$ has type $D_5 + A_1$, and $u' \in C'^F$.
It is verified that $u'$ is a split element in $C'^F$.  Since
$F$ acts trivially on $A_G(u)$ and on $A_M(u')$,
Lemma 4.16 can be applied, and by (6.4.3), we have
\begin{equation*}
\tag{6.4.5}  
\g(1,1, u, 1) = \g(1,1,u, \ve') = 1, \quad  \g(u_0, \wt\r_0, u, \th) = 1.
\end{equation*}  

On the other hand, by the computation of Green functions by \cite{Lu},
there exists $u_1 \in C^F$ such that $F$ acts trivially on $A_G(u_1)$, and that
$\g(1,1, u_1, \r) = 1$ for any $\r \in A_G(u_1)\wg$ such that $(u,\r)$ belongs to
$(T, \{1\},1)$.
Since $A_G(u) \simeq D_8$, there exists exactly two $G^F$-classes in $C^F$
such that $F$ acts trivially on $A_G(v)$ for $v$ in such a class. 
Thus $u_1$ is obtained from $u$
by twisting with $z \in A_G(u)$ as $u_1 = u_z$, where $z$ is a central element.  Then we have
$\g(1,1, u, \ve) = \ve(z)$ and $\g(1,1,u, \ve'') = \ve''(z)$.
We have $Z(D_8) = \{ \pm 1\}$, and $\ve(z) = \ve(z'') = 1$ ($-1$ is a product of
two reflections corresponding to two short roots or two long roots).   Hence
\begin{equation*}
\tag{6.4.6}
\g(1,1,u,\ve) = 1, \quad \g(1,1, u, \ve'') = 1.
\end{equation*}  
By (6.4.5) and (6.4.6), $u \in C^F$ is a split element. 

\para{6.5.}
Next consider the case where $C = E_7(a_2) + A_1$ with $p = 2$.
Then $A_G(u) \simeq S_3 \times \BZ_2$ and $A_G(u)\wg = \{ 1,\ve, \th, -1, -\ve, -\th\}$.
$(u, -\ve)$ is a cuspidal pair for $G$. 
Let $u = z_{50} \in C^F$ be a representative of Mizuno \cite{M2}.  
We have
\begin{align*}
\tag{6.5.1}  
u &= x_{14}(1)x_{15}(1)x_{21}(1)x_{17}(1)x_{18}(1)x_{29}(1)x_{19}(1)x_{101}(1) \\
  &= x_{\a}(1)x_{3+2}(1)x_{\b}(1)x_{4+1}(1)x_{5-2}(1)x_{5+2}(1)x_{6-3}(1)x_{7-2}(1), 
\end{align*}
where $\a = -1-2+3-4-5-6-7+8, \b = -1-2-3+4-5-6-7+8$. 

Let $C'$ be the class in $M$ of type $D_5$.  Take $u' \in C'$.
Then $A_M(u') \simeq \BZ_2$, and we have the decomposition of $\ve_{u,u'}$ as

\begin{equation*}
\tag{6.5.2}  
  \ve_{u,u'} = (1\otimes 1) \oplus (\th\otimes 1) \oplus (-1\otimes -1)
                   \oplus (-\th \otimes -1).
\end{equation*}
In particular, $-1, -\th \in A_G(u)\wg$ appear in (6.6.2). 
We have $\dim Z_G(u) = \dim Z_M(u') = 22$.
We shall find $g \in G^F$ such that $v = g\iv ug \in Q$ and that
the projection $u' \in C'^F$ of $v$ to $M$, explicitly. 
\par
Let $g_1 = x_{-7+5}(1)x_{4-3}(1)$.  Then we have
\begin{align*}
g_1ug_1\iv  &=
      x_{\a}(1)x_{3+2}(1)x_{4+2}(1)x_{4+1}(1)x_{5+2}(1)x_{6-3}(1)x_{7-2}(1)
\end{align*}

Now put $\g = 1-2-3-4+5-6+7+8$.  Then
\begin{align*}
  w_{\g}(3+2) &= 1+2+3-4+5-6+7+8,  \quad w_{\g}(4+2) = 1+2-3+4+5-6+7+8,  \\
  w_{\g}(\a)  &= \a, \quad w_{\g}(4+1) = 4+1,  \quad w_{\g}(5+2) = 5+2, 
  \quad w_{\g}(6-3) = 6-3,  \\
  w_{\g}(7-2) &= -1-2+3+4-5+6+7-8 = -(\a_1 + (5+2)).
\end{align*}

Note that ${}^{\w_{\g}}x_{3+2}(1), {}^{\w_{\g}}x_{4+2}(1) \in U_Q$,
and ${}^{\w_{\g}}x_{7-2}(1) \in M$.  
Put $g = \w_{\g}g_1$, and let 
$u'$ be the projection of $v = gug\iv  \in Q$ to $M$.
Then we have

\begin{align*}
\tag{6.5.3}  
u' &= x_{\a}(1)x_{4+1}(1)x_{5+2}(1)x_{6-3}(1)x_{\d}(1), 
\end{align*}
where $\a$ is as before, and $\d = -1-2+3+4-5+6+7-8 \in \vD_7^-$.
Thus $u' \in M^F$.  Since $u'$ is conjugate to
$x_{6-3}(1)x_{4+1}(1)x_{\a}(1)x_{5+2}(1)x_{7-2}(1)$, the type of $u'$ is $D_5$.  
But $u'$ is not contained in $U_M$ (the unipotent radical of $B \cap M$) since $\d < 0$.
We want to find a conjugate of $u'$ contained
in $U_M$. 
If we put 
$\xi = 1 + 2 + 3 - 4 -5 -6 + 7- 8 \in \vD_7$,
we obtain a diagram of $E_7$; 

\begin{align*}
&\xymatrix@C=7pt@ R=7pt@ M =5pt{
                 &     &   (4+1) \ar@{-}[d]  &   &   & \\
    \xi\phantom{*}  \ar@{-}[r]  &  (6-3) \ar@{-}[r]  & \phantom{**} \a \phantom{**} \ar@{-}[r]
                                         &   (5+2) \ar@{-}[r]
          & \phantom{*} \d \phantom{*}  \ar@{-}[r]  & (-6-3)       
} 
\end{align*}

Since this is a simple system of type $E_7$, there exists $w \in W(E_7)$
such that $w : \{ \xi, 4+1, 6-3, \a, 5+2, \d, -6-3\} \to \{ \a_1,\a_2,\dots, \a_7\}$
in this order. 
In particular $w$ maps $\{ 4+1, 6-3, \a, 5+2, \d \} \to \{ 2+1, 2-1, 3-2,4-3, 5-4 \}$.
Hence we have
\begin{equation*}
\w u'\w\iv = x_{2+1}(1)x_{2-1}(1)x_{3-2}(1)x_{4-3}(1)x_{5-4}(1),
\end{equation*}
where $\w$ is an element in $M$ corresponding to $w \in W(E_7)$,
which is a product of elements $\w_i \in M^F$,
hence $\w \in M^F$. Note that since $p = 2$, we can ignore the signs in the computation. 
\par
Now it is verified that $\w u' \w\iv$ is a split element in $C'^F$.
Since $F$ acts trivially on $A_G(u)$, and on $A_M(u')$, by applying Lemma 4.16,
(6.5.2) implies that
\begin{equation*}
  \g(1,1, u,1) = \g(1,1,u,\th) = 1, \quad
     \g(u_0, \wt\r_0, u, -1) = \g(u_0, \wt\r_0, u, -\th) = 1.
\end{equation*}  

On the other hand, by the computation of Green functions due to \cite{Lu},
there exists $u_1 \in C^F$ such that $F$ acts trivially on $A_G(u_1)$ and
that $\g(1,1, u_1, \r) = 1$ for $\r = 1, \ve, \th$. As in the discussion
in 6.5, $u_1$ is obtained from $u$ as $u_1 = u_{z}$,
by twisting by a central element $z \in A_G(u)$.  Then we have
$\g(1,1, u, \ve) = \ve(z)\g(1,1,u_1, \ve) = 1$.  Hence $u \in C^F$ is a split element. 
\par\medskip
Summing up the above discussion, we have proved that Theorem 4.6 holds
for $G = E_8$.  

\par\bigskip
\section{ The case $F_4$ }

\para{7.1.}
In  this section, we  assume that $G$ is of type $F_4$. 
Then $A_G(u)$ is isomorphic to $D_8$ ($p = 2$) if $C = F_4(a_2)$, 
to $S_4$ ($p \ne 2$), $S_3$ ($p = 2$) if $C = F_4(a_3)$. In all other cases, 
$A_G(u)$ is abelian.  The set $\SM_G$ is given, except for cuspidal ones, by 
$\{ (T,\{1\},1), (L,C_0,\SE_0)\}$ with $L$ of type $B_2$ in the case $p = 2$, 
and $\{ (T,\{1\},1)\}$ in the case $p \ne 2$.  
\par
Consider a list of unipotent classes $C$,
\begin{equation*}
\tag{7.1.1}  
F_4, \quad F_4(a_1), \quad F_4(a_2), \quad F_4(a_3). 
\end{equation*}
For those $C$ in the list, Lemma 4.13 cannot be applied since
there does not exist a proper Levi subgroup $M$ such that $M \cap C \ne\emptyset$.

\para{7.2.}
If $p \ne 2$, since $|\SM_G(C)| = 1$ for any class $C$, (4.7.2) implies
the existence of split elements. Thus in the discussion below, we assume that $p = 2$.
In this case, $L$ is of type $B_2$.
Let $C$ be the class which is not in the list (7.1.1).
If $|\SM_G(C)| = 2$,
then $C$ is in the list
\begin{align*}
\tag{7.1.2}  
B_3, \quad C_3, \quad (B_2)_2.
\end{align*}  
In those cases,  Lemma 4.13 can be applied with $M  \supset L = B_2$. 
For other $C$, Lemma 4.13 can be applied since $|\SM_G(C)| = 1$, and $A_G(u)$ is
abelian.

\para{7.3.}
Assume that $C$ is in the list (7.1.1).  
For $C = F_4$,
since $C$ is regular unipotent, a similar argument as in the case of $E_7$
(see 5.3 (1)) can be applied.  For $C = F_4(a_1), F_4(a_3)$, since $|\SM_G(C)| = 1$,
again (4.7.2) can be applied. 
\par
Now assume that $C = F_4(a_2)$.  In this case, $A_G(u) \simeq D_8$ and  
$|\SM_G(C)| = 2$, where $L = B_2$ and $\SW_L = W(B_2)$.
$(u,\ve)$ is a cuspidal pair for $G$. 
In the discussion below, we follow the notation in 4.2.  In particular,
we use a convention similar to 5.5, for example, 
$\frac{1}{2}(\ve_1-\ve_2-\ve_3-\ve_4)$ is written as $1-2-3-4$.
We choose $u = x_{24} \in C^F$ from Shinoda \cite{Shi}, where 
\begin{align*}
  u &= x_{2-4}(1)x_{3+4}(1)x_{1-2-3-4}(1)x_{1-2-3+4}(1). 
\end{align*}
Then $F$ acts trivially on $A_G(u)$. 
Let $Q = MU_Q$ be a parabolic subgroup of $G$ such that $M$ is a Levi subgroup
of type $C_3$. For $g = x_4(1) \in G^F$, we have 
\begin{equation*}
gug\iv = x_{2-4}(1)x_2(1)x_{3+4}(1)x_{1-2-3-4}(1) \in MU_Q.
\end{equation*}
Let $u'$ be the projection of $gug\iv$ on $M$.
Then we have
\begin{equation*}
u' = x_{3+4}(1)x_{1-2-3-4}(1) \in M^F.
\end{equation*}

Let $C'$ be the class in $M$ containing $u'$. Then $u'$ is $M^F$-conjugate
to $x_{3-4}(1)x_4(1)$, which is a regular element of the subgroup of $M$ of
type $C_2$.  Hence $u'$ is of type $((4), (1))$. 
$A_M(u') \simeq \BZ_2$, and one can check that $u'$ is a split element in $C'^F$.
Furthermore we have $\dim Z_G(u) = \dim Z_M(u') = 8$. 
$\ve_{u,u'}$ is decomposed as 
\begin{equation*}
\tag{7.3.1}
\ve_{u,u'} \simeq (1\otimes 1) \oplus (\ve''\otimes 1) \oplus (\th \otimes 1).
\end{equation*}
Now Lemma 4.16 can be applied, and by (7.3.1) we have
\begin{equation*}
\g(1,1, u, 1) = \g(1,1, u, \ve'') = 1, \quad \g(u_0, \wt\r_0, u, \th) = 1.
\end{equation*}  

On the other hand, from the computation of Green functions due to Geck \cite{G2},
there exists $u_1 \in C^F$ such that $\g(1,1,u_1, \r) = 1$ for
$\r = 1, \ve',\ve''$.
By a similar discussion as in the case $G = E_8$ with $C = D_8(a_1)$ (see 6.5),
we have $\g(1,1,u, \ve') = \g(1,1, u_1, \ve') = 1$. Hence $u$ is a split element.  

\par\bigskip
\section{ The case $G_2$ }

\para{8.1.}
Assume that $G = G_2$.
In this case, $|\SM_G(C)| = 1$ for any class $C$, which works for any $p$.
Thus (4.7.2) is applied to show the existence of split elements. 

\par\bigskip
\section{ The case $E_6$ }

\para{9.1.}
In this section, we assume that $G = E_6$, where $F$ is split or non-split. 
\par
Note that any unipotent class $C$ is $F$-stable (see (4.1.2)). 
\par
For $u \in C$, $A_G(u)$ is abelian unless $C = D_4(a_1)$, in which
case, $A_G(u) \simeq S_3$.
\par
The set $\SM_G$ is given, except for cuspidal ones,  by
\begin{align*}
     &\{ (T, \{1\}, 1), (L, C_0, \SE_0), (L', C_0',\SE_0')\}
             \quad\text{ for $L = 2A_2, L' = D_4$ with $p = 2$,}  \\  
    &\{ (T, \{1\}, 1), (L, C_0, \SE_0)\} \quad\text{ for $L = 2A_2$ with $p \ne 2,3$,} \\
    & \{ (T, \{1\}, 1) \}    \quad\text{ with } p = 3,
\end{align*}
where in the first and second case, there exist two $\SE_0$ for the class $C_0$ of $L$. 
\par
Note that if three types of Levi subgroups $T, L$ and $L'$ are involved in $\SM_G(C)$,
then it is only the case where $C = E_6$. 

\para{9.2.}
Take $u \in C^F$. The action of $F$ on $A_G(u)$ was explained 
in (4.1.4).  More precisely, we have the following. 
$Z_G \simeq \BZ_3$ (resp. $Z_G = 1$) if $p \ne 3$ (resp. $p = 3$).
In the case where $Z_G = 1$, for any unipotent class $C$, there exists
a representative $u \in C^F$ such that $F$ acts trivially on $A_G(u)$. So 
the situation is similar to other types.  Now assume that $p \ne 3$, and 
$Z_G \simeq \BZ_3$. We have a natural map $\vf : Z_G \to A_G(u)$. If $\vf$ 
is non-trivial, then $\vf$ gives an isomorphism $Z_G \simeq A_G(u)$
except when 
$C = A_5 + A_1$, and $A_G(u) \simeq \BZ_3 \times \BZ_2$, in which case
$\vf$ gives an isomorphism from $\BZ_3$ to the direct factor $\BZ_3$ in 
$A_G(u)$. (Note that in the case where $C = D_4(a_1)$, $\vf : Z_G \to S_3$
is the trivial map.) 
Now the action of $F$ on $A_G(u)$ is given by the action of $F$ on 
the image of $Z_G$ in $A_G(u)$.   
\par
In the case where $p \ne 3$, $Z_G$ is identified
with the subgroup $\{ z \in \Bk^* \mid z^3 = 1\}$ of $\Bk^*$, 
and the action of $F$ on $Z_G$ is given by $z \mapsto z^q$ 
(resp. $z \mapsto z^{-q}$) if $F$ is split (resp. non-split). 
Hence in the split case 
$F$ acts trivially on $\BZ_3$ if and only if 
$q \equiv 1 \pmod 3$, and in the non-split case 
 $F$ acts trivially on $\BZ_3$ if and only if $q \equiv -1 \pmod 3$. 

\para{9.3.}
Assume that $p \ne 3$, then $Z_G \simeq \BZ_3$.
For $\x \in Z_G\wg$, we consider the generalized Springer correspondence
(1.2.1) pertaining to $\x$.  The generalized Springer correspondence (1.1.4)
and its partition to $\x$ parts in (1.2.1) are compatible with the actions of $F$.
Since $F$ permutes $Z_G\wg$, $F$ permutes the sets
$\SN_{\x}$. In particular, in considering $\SN_G^F$ and $\SM_G^F$,
it is enough to focus on $\SN_{\x}$ and $\SM_{\x}$ for $F$-stable $\x$. 
In our case, $Z_G \simeq \BZ_3$.  If $F$ acts non-trivially on $Z_G$, then
$F$ permutes two non-trivial characters $\x, \x'$ of $Z_G$, hence 
permutes $\SN_{\x}$ and $\SN_{\x'}$.  Thus as for $\SN_G^F$,
it is enough to consider the correspondence in (1.2.1) for $\x = 1$.  
Note that this is essentially the same as the generalized
Springer correspondence for the adjoint group of $G$. 

\para{9.4.}
We shall prove Theorem 4.6 for $G$ in a similar 
way as in the previous sections.  But 
in the case where $F$ is non-split, we need to use  
Lemma 1.15 rather than lemmas in Section 4. Also we need a special care 
if there does not exist $u \in C^F$ 
such that $F$ acts trivially on $A_G(u)$.

\para{9.5.}
First assume that $F$ is split, and $q \equiv 1 \pmod 3$ if $p \ne 3$.  
Hence by 9.2, for each class $C$, there exists a representative 
$u \in C^F$ such that $F$ acts trivially on $A_G(u)$. 
\par
Consider the class $C$ in the list

\begin{align*}
\tag{9.5.1}
E_6, \quad E_6(a_1), \quad A_5 + A_1, \quad D_4(a_1). 
\end{align*}
In those cases, Lemma 4.13 cannot be applied; if $C = E_6, E_6(a_1), A_5 + A_1$, there
does not exist a proper Levi subgroup $M$ such that $C \cap M \ne \emptyset$,
and if $C = D_4(a_1)$, then $A_G(u)$ is not abelian. 

\par 
Assume that $C$ is in the remaining cases.  First assume that $p \ne 3$.  
In this case, $L$ is of type $2A_2$ or $D_4$.  If $|\SM_G(C)| \ge 2$, then $C$ is in the list
\begin{align*}
\tag{9.5.2}
A_5, \quad 2A_2 + A_1, \quad 2A_2, \quad D_5, \quad D_4,
\end{align*}
where the former three cases are with respect to $L = 2A_2$, and the latter two cases
are with respect to $L = D_4$. 
In those cases, one can find $M \supset L$ such that $M  \cap C \ne\emptyset$,
hence Lemma 4.13 can be applied.  If $|\SM_G(C)| = 1$, (4.7.2) is applied. 
\par
In the case where $p = 3$, we have $|\SM_G (C)| = 1$ for any $C$,
and (4.7.2) can be applied.
\par
Thus unless $C$ is in the list (9.5.1), split elements
$u \in C^F$ exist. 

\para{9.6.}
Under the assumption in 9.5, we assume that $C$ is in the list (9.5.1).
\par\medskip
(1) \ First assume that $C = E_6(a_1)$. In this case, for $u \in C^F$,
$A_G(u) = 1$ if $p = 3$, and $A_G(u) \simeq \BZ_3$ if $p \ne 3$. So we assume that
$p \ne 3$.  Then $L$ is of type $2A_2$, and $\SW_L$ is of type $G_2$.
Let $\z = \x \in \BZ_3\wg$ be a non-trivial character.  Then $V_{(u,\z)} \in \SW_L\wg$
coincides with the short sign character $\ve_c \in W(G_2)\wg$.  
Let $M$ be a Levi subgroup of type $2A_2 + A_1$, and $C'$ the class
of type $2A_2$. In this case $A_M(u') \simeq \BZ_3$ for $u' \in C'$, 
and 
$\ve_{u,u'}$ is decomposed as
\begin{equation*}
\tag{9.6.1}  
\ve_{u,u'} = (1\otimes 1) \oplus (\z \otimes \z^*) \oplus (\z^2\otimes {\z^{2*}}).
\end{equation*}  
Here $\dim Z_G(u) = \dim Z_M(u') = 8$.
Thus by Lemma 4.16, there exists a split element $u \in C^F$. 
\par\medskip
(2) \  Next assume that $C = A_5 + A_1$.  In this case, $|\SM_G(C)| = 1$ if
$p = 3$.  Hence we assume that $p \ne 3$. Then $A_G(u) \simeq \BZ_2 \times \BZ_3$,
and $\dim Z_G(u) = 12$.
Let $\z$ be an irreducible character of $A_G(u)$ of order 3.
Note that $(u, -\z), (u, -\z^2)$ are cuspidal pairs for $G$.  Let $M$ be a Levi subgroup
of type $A_5$, and choose $C' \subset M$ such that the Jordan type of
$C'$ is $(33)$. Then for $u' \in C'$, we have $A_M(u') \simeq \BZ_3$, and
$\dim Z_M(u') = 12$. $\ve_{u,u'}$ is decomposed as
\begin{equation*}
\tag{9.6.2}  
\ve_{u,u'} = (1\otimes 1)\oplus (\z \otimes \z^*) \oplus (\z^2\otimes \z^{2*}).
\end{equation*}  
From the computation of Green functions, there exists a split element $\ol u$ in
the image of $C$ for the adjoint group $\ol G$ of $G$.  We choose $u \in C^F$
so that the image of $u$ coincides with $\ol u$, and that $u \in u'U_Q$, where
$u' \in C'^F$ is a split element. 
Then by Lemma 4.16, we have
\begin{equation*}
\g(1,1, u, 1) = 1, \quad \g(u_0, \r_0, u, \z) = 1, \quad \g(u_0, \r_0, u, \z^2) = 1.
\end{equation*}
Moreover, since $\ol u$ is split, we have
\begin{equation*}
\g(1,1,u, -1) = 1.
\end{equation*}  
Hence $u \in C^F$ is a split element.
\par\medskip
(3) \ 
Next assume that $C = D_4(a_1)$. Since $|\SM_G(C)| = 1$,
a split element $u \in C^F$ exists by (4.7.2).

\par\medskip
(4) \ Finally assume that $C = E_6$. 
If $|\SM_G(C)| = 1$, (4.7.2) can be applied.  If only $L, T$ or $L', T$ is involved in
the set $\SM_G(C)$, 
a similar discussion as in 5.3 (1) (the case $G = E_7, C = E_7$) can be applied.
Hence we may assume that the three types $T, L, L'$ of Levi subgroups are
involved in $\SM_G(C)$.  In this case $p = 2$. 
\par
Let $P = LU_P$ (resp. $P' = L'U_{P'}$)  be a parabolic subgroup of $G$ as in 9.1, and   
$C_0$ (resp. $C_0'$) the regular unipotent class in $L$ (resp. $L'$).
Let $P'' = L''U_{P''}$ 
be the parabolic subgroup of $G$ such that
$L'' = L \cap L'$ is the Levi subgroup of $P''$ 
with type $2A_1$.  Let $C_0''$ be the regular unipotent class in $L''$. 
We choose a split element $u_0'' \in C''^F_0$. 
Since $C = \Ind_{L''}^GC_0''$,   
there exists $u \in u_0''U^F_{P''}$ such that $u \in C^F$.
One can find $u_0 \in C_0^F$ 
such that $u \in u_0U_Q^F$. 
Since $F$ acts trivially on $A_G(u)$, $F$ acts trivially on $A_L(u_0)$. 
Moreover, $A_L(u_0)$ is abelian. 
By the transitivity of induction, we have $C_0 = \Ind_{L''}^LC_0''$. 
Since $u_0$ is regular unipotent, by a similar discussion as in 5.3 (1),
one can apply Lemma 4.9 for $L$, we see that
$u_0$ is a split element in $C_0^F$.  A similar discussion works also for $L'$,
namely, one can find $u_0' \in C_0'^F$ such that $u \in u_0'U_{P'}^F$, and
$u_0'$ becomes a split element in $C_0'^F$. 
Then by applying Lemma 4.9 (and Lemma 1.15) for $G$ and $L$, we have 
$\g(u_0, \r_0, u, \r) = 1$ for any $(u, \r)$ belonging to $(L,C_0, \SE_0)$.
Similarly, by applying Lemma 4.9 for $G$ and $L'$, 
we have $\g(u_0', \r_0', u,\r) = 1$ for any pair $(u,\r)$ belonging to 
$(L',C_0', \SE_0')$. 
Lemma 4.9 can be applied also for $G$ and $T$, and we have
$\g(1,1,u,\r) = 1$ for $(u,\r)$ belonging to $(T, \{ 1\}, 1)$.
Thus $u$ is a split element in $C^F$. 
\par\medskip
Summing up the above discussion, we have 

\begin{prop}  
Assume that $G = E_6$, and that $F$ is split on $G$.
Further assume that $q \equiv 1 \pmod 3$ if $p \ne 3$. Then a split 
element $u \in C^F$ exists for any class $C$.  
\end{prop}

\para{9.8.}
Next consider the case where $F$ is split on $G$, 
and $q \equiv -1 \pmod 3$. 
In this case, there exists
a class $C$ such that $F$ acts non-trivially on $A_G(u)$ for any 
choice of $u \in C^F$. 
If $C$ has a representative $u \in C^F$ such that $F$ acts trivially on 
$A_G(u)$, then the discussion in the proof of Proposition 9.7 can be applied, 
and we obtain a split element in $C^F$.
\par
We now consider $C$ such that $F$ acts non-trivially on 
$A_G(u)$ for any choice of $u \in C^F$.  Those $C$ are given as follows. 

\begin{align*}
\tag{9.8.1}
C : \quad E_6, \quad E_6(a_1), \quad A_5 + A_1, \quad A_5, \quad 2A_2 + A_1, \quad 2A_2.
\end{align*}
The corresponding $A_G(u)$ are given as follows
\begin{align*}
\tag{9.8.2}
A_G(u) : \quad  \BZ_{3(2,p)}, \quad \BZ_3, \quad \BZ_3 \times \BZ_2, \quad \BZ_3, 
             \quad \BZ_3, \quad \BZ_3,  
\end{align*}
where the action of $F$ on $A_G(u)$ is given by the non-trivial action
of $F$ on $\BZ_3$ (see 9.2).  By 9.2, this action corresponds to the
action of $F$ on $Z_G$.
Hence if $(u,\r) \in \SN_G^F$, then $\r \in A_G(u)\wg_{\x}$ for the trivial
character $\x \in Z_G\wg$ (see 9.3). 
It follows that $\r$ must be
$1 \in \BZ_3\wg$ or $\pm 1 \in (\BZ_3 \times \BZ_2)\wg$.
\par\medskip
(1) \
If $C \ne E_6, E_6(a_1), A_5 + A_1$, there exists $M \ne G$ such that $M \cap C \ne\emptyset$.
Let $u = u' \in C'^F \subset M^F$ for $u \in C^F$, and 
consider the set $X_{u,u'}$.  Note that the statement of Lemma 4.11 holds
if we ignore the action of $F$, even if $F$ does not act trivially on $A_G(u)$.
Hence $A_M(u')$ is regarded as a subgroup of $A_G(u)$, compatible with $F$ action,
and $X_{u,u'}$ is isomorphic to the set $A_G(u)$ on which $A_G(u) \times A_M(u')$
acts from left and right.  This isomorphism is compatible with $F$.  
\par
Now take a split element   
$u'$ in $C'^F$ and let $\wt \r'$ be the split extension of
$\r' \in A_M(u')\wg$ to $\wt A_M(u')$. In this case, $A_M(u')$ is abelian,
and $\wt\r'$ is the trivial
extension of $\r'$.   
Then by the description of $X_{u,u'}$ given above, we have the following.
Let $\r = 1$ (for $A_G(u) \simeq \BZ_3$), or $\r = \pm 1$  (for
$A_G(u) \simeq \BZ_3 \times \BZ_2$), then there exists a unique extension
$\wt\r$ to $\wt A_G(u)$ such that $F$ acts trivially on the
isotypic subspace $\ve_{u,u'}(\r,\r')$ in $\ve_{u,u'}$.
In fact, $\wt\r$ is given by the trivial extension of $\r$ to $\wt A_G(u)$. 
Since $F$ is split on $M$, we have $\g(u_0,\wt\r_0, u', \wt\r') = 1$. 
Since $\d = 1$, by applying Lemma 1.15, we obtain 
$\g(u_0, \wt\r_0, u, \wt\r) = 1$.   Hence $u$ is a split element in $C^F$. 
\par\medskip
(2) \ Next assume that $C = E_6(a_1)$ or $A_5 + A_1$.
The discussion in 9.6 (1), (2)  can be applied
similarly, except the action of $F$.
We follow the notation in 9.6.  Then $X_{u,u'} \simeq A_G(u)$, and
$F$ acts non-trivially on it.  As in the previous discussion, for a split
element $u' \in C'^F$ and the split extension $\wt\r'$ of $\r'$,
$F$ acts trivially on $\ve_{u,u'}(\r,\r')$ for the trivial extension $\wt\r$
of $\r = 1$ (resp. $\r = \pm 1$) to $\wt A_G(u)$ if $C = E_6(a_1)$ (resp.
if $C = A_5 + A_1$). 
Thus $u$ is a split element in $C^F$ by Lemma 1.15. 
\par\medskip
(3) \ Finally assume that $C = E_6$.  Assume that $p = 2$.
Then $A_G(u) \simeq \BZ_3 \times \BZ_2$.
Following the notation in 9.7, we consider $C = \Ind_{D_5}^{E_6}C'$.
In this case, we only consider $\r = \pm 1 \in A_G(u)\wg$, thus no 
need to consider the part belonging to $(L, C_0, \SE_0)$ for $L = 2A_2$.
Hence the discussion in 9.7 is unnecessary,
and just by a similar discussion as in the case $C = E_6(a_1)$,
we see that there exists a split element in $C^F$. 
The case $p \ne 2$ is dealt similarly. 
\par\medskip
Summing up the above discussion, we have

\begin{prop} 
Assume that $G = E_6$, and that $F$ is split on $G$. 
Further assume that $q \equiv -1 \pmod 3$. Then for any class $C$, there exists
a split element $u \in C^F$. For each $\r \in A_G(u)\wg$, the split extension
$\wt\r$ is given by the trivial extension. 
\end{prop}     

\para{9.10.}
We now assume that $G = E_6$ and $F$ is non-split. Assume that $p$ is good. 
In \cite{BS}, they have proved the existence of $u \in C^F$ such that
$F$ acts on $H^{2d_u}(\SB_u)$ as a scalar multiplication by $(-q)^{d_u}$. 
This means that $\g(1,1, u,\r) = (-1)^{\d_E}$ for $E = V_{(u,\r)} \in W(E_6)\wg$,
where $\d_E = a_E - d_u$ as in (1.7.4), whenever $(u,\r)$ belongs to $(T, \{1\},1)$. 
By using the table of $a$-functions given in \cite[4.11]{L1}, $a_E$ coincides with $d_u$
for $E = V_{(u,\r)}$ unless
\begin{equation*}
\tag{9.10.1}  
E : \quad 15_4\ (C = A_5), \quad 10_9\ (C = 2A_2 + A_1), \quad 15_{16}\ (C = 3A_1).
\end{equation*}
In those cases, we have
$(d_u, a_E) = (4,3), (9,7), (16,15)$, accordingly.
It follows that

\begin{equation*}
\tag{9.10.2}  
\g(1,1, u, \r) = \begin{cases}
                        -1   &\quad\text{ if $V_{(u,\r)} = 15_4$ or $15_{16}$, } \\
                         1   &\quad\text{ otherwise.}
                   \end{cases}
\end{equation*}  

This formula makes sense even if $p$ is bad, and the existence of $u \in C^F$
satisfying (9.10.2) was verified through the computation of Green functions due to
Geck \cite{G2}. 
\par
In the case of generalized Springer correspondence, $L = 2A_2\ (p \ne 3)$ or
$D_4 \ (p = 2)$ appears. If $L = 2A_2$, then $W_L \simeq W(G_2)$ and $\d_E = 1$
for any $E \in \SW_L\wg$. While if $L = D_4$, then $\SW_L \simeq A_2$.  In this case,
it is checked that $\d_E = 1$ for any $E \in \SW_L\wg$. 

\para{9.11.}
In the case where $|\SM_G(C)| = 1$, the results in 9.10 fit
to (4.7.2), and split elements exist. Thus we may assume that
$|\SM_G(C)| \ge 2$, and $C$ is in the following list.

\begin{align*}
\tag{9.11.1}  
E_6, \ E_6(a_1), \ D_5, \ A_5 + A_1, \ A_5, \ D_4, \ 2A_2 + A_1, \ 2A_2. 
\end{align*}  

Since $F$ is non-split, if $q \equiv -1 \pmod 3$, for any class $C$ there 
exists a representative $u$ such that $F$ acts trivially on $A_G(u)$.  While
if $q \equiv 1\pmod 3$, there exists a class $C$ such that $F$ acts non-trivially 
on $A_G(u)$ for any choice of $u \in C^F$ (see 9.2). 
The discussion in the proof of Proposition 9.7 and Proposition 9.9 
basically works for almost all $C$ by switching the condition 
$q \equiv \pm 1 \pmod 3$.

\par
First assume that $C$ is not in the list (9.5.1).  If $p \ne 3$, then $C$ is
in the list (9.5.2), and  
following the discussion in 9.5, we find an $F$-stable $M$ such that $M  \cap C \ne\emptyset$.
However, in the list (9.5.2),  we have to
exclude the case  $C = D_5$ 
since $F$-stable $M$ does not exist in this case.
We discuss the case $C = D_5$, together with $C = A_5 + A_1$, later in 9.13. 

\par 
For the remaining cases, the discussion in 9.5 can be applied also
for the case $F$ non-split as follows. Assume that  $(u,\r)$ belongs to
$(L, C_0, \SE_0)$ with $L \ne T$, where $L = 2A_2$ or $D_4$. First assume that
$L = 2A_2$. We can choose a Levi subgroup $M$ of type $A_5$ or $2A_2 + A_1$.
By applying the discussion in Section 2 for $M$, one can find
a split element $u' \in C'^F$, and the split extension
$\wt\r'$ for $\r' \in A_M(u')\wg\ex$
such that $\g(u_0,\wt\r_0, u', \wt\r') = \nu'$, where $\nu'$ is given by
$\nu' = \nu_{\la, \eta}(-1)^{\d_i}$ as in Theorem 2.19, (i) for a certain
root of unity $\nu_{\la, \eta}$, and $(-1)^{\d_i} = \pm 1$.
In particular, if $M = 2A_2 + A_1$, we have $\nu_{\la,\eta} = 1$. 
By modifying the discussion in the proof of Proposition 4.12, together with
Lemma 1.15, we see that there exists $u \in C^F$ and an extension $\wt\r$
for each $\r \in A_G(u)\wg\ex$ such that $\g(u_0, \wt\r_0, u, \wt\r) = \nu'$.  
Here $\wt\r$ coincides with the trivial extension of
$\r$ (note that $A_G(u)$ is abelian). 
Thus one can write as $\g(u_0, \wt\r_0, u, \wt\r) = \nu_{u, \eta}(-1)^{\d_E}$,
where $\nu_{u,\eta}$ coincides with $\nu_{\la, \eta}$ for $M$, and $u$ is a split element
as given in Theorem 4.6 (ii).
Next assume that $L = D_4$. Since $C = D_4$, we can choose $M$ of type $D_4$.
By applying the discussion in Section 3 for $u = u' = u_0$, we see
that $\g(u_0, \wt\r_0, u, \wt\r) = 1$. Thus $u$ is a split element. 

\para{9.12.} \ Assume that $C$ is in the list (9.5.1), namely,
$C = E_6, E_6(a_1)$ or $D_4(a_1)$ (the case $C = A_5 + A_1$ will be discussed
in 9.13.)  For $C = E_6, E_6(a_1)$, the discussion
in 9.6 (1) and (3) can be applied also for the case $F$ is non-split. Note that
in 9.6 (1), we have chosen $M = 2A_2 + A_1$ so that $M$ is $F$-stable.
The case where $C = D_4(a_1)$ is already done since $|\SM_G(C)| = 1$.
\par
Thus one can find a split element $u \in C^F$ for $C$ in the list (9.5.1).

\para{9.13.}
We consider the case where $C = D_5$ or $A_5 + A_1$.
These cases are dealt with by applying
(a variant of) Lemma 4.9.
\par
For example, assume that $C = A_5 + A_1$ with $q \equiv -1 \pmod 3$.
Then $A_G(u) \simeq \BZ_3 \times \BZ_2$.
We have $A_G(u)\wg = \{ 1,\z, \z^2, -1, -\z, -\z^2\}$, where $\z$ is a non-trivial
character of $\BZ_3$. 
We have $C = \Ind_{3A_1}^{E_6}1$.  Let $Q = MU_Q$ be such that 
$M$ is of type $A_5$, and put $C' = \Ind_{3A_1}^{A_5}1$.
Then $C = \Ind_{A_5}^{E_6}C'$. The Jordan type of $u' \in C'$ is given
by $(33)$ as elements in $SL_6$, namely $C'$ is of type $2A_2$.
Then $Z_M(u') \simeq \BZ_3$.
We have $N = A_G(u)$, and so $H  = \BZ_2$. Hence
$X_{u,u'} \isom A_G(u)/H \simeq \BZ_3$ with $A_G(u) \times A_M(u')$-actions. 
We choose a split element $u' \in C'^F$ as in Theorem 2.19.
By applying the discussion in Section 2 for $M $ of type $A_5$, one can write as
$\g(u_0, \wt\r_0, u', \wt\r') = \nu'$, where $\nu' = \nu_{\la,\eta}(-1)^{\d_i}$
as in Theorem 2.19. 
By applying Lemma 1.15,
there exists an element $u \in C^F$ such that $\g(u_0, \wt\r_0, u, \wt\r) = \nu'$
for $\r = \z$ or $\z^2$. 
(Note that $(u,-\z), (u, -\z^2)$ are cuspidal pairs for $G$.) 
In particular, $\g(u_0, \wt\r_0, u, \wt\r)$ is expressed as in the form (4.6.1). 
\par
On the other hand, by 9.10, there exists $u_1 \in C^F$ such that
$\g(1,1, u_1, \r_1) = 1$ for $\r_1 = \pm 1$.
Here $u_1$ is obtained from $u$ as $u_1 = u_z$ by twisting by $z \in \BZ_2$.
It follows that $\g(1,1, u, \pm 1) = \g(1,1,u_1, \r_1) = 1$.
Hence $u$ is a split element.
\par
For $C = A_5 + A_1$ with $p = 3$ or $C = D_5$, we have $A_G(u) \simeq \BZ_2$,
and a similar discussion as above can be applied (in fact, it is simpler). 
\par
The case where $q \equiv 1 \pmod 3$ is discussed in a similar way 
as in 9.8.
\par\medskip
Summing up the above discussion, we have the following.

\begin{prop}
Assume that $G = E_6$, and that $F$ is non-split on $G$.
Then for any class $C$, there exists a split element $u \in C^F$.
For each $\r \in A_G(u)\wg$, the split extension $\wt\r$ is given by
the trivial extension of $\r$. 
\end{prop}  

Combining Proposition 9.7, Proposition 9.9 and Proposition 9.14,
we have proved that Theorem 4.6 holds for $G = E_6$.

\para{9.15.}
In the case where $G = E_6$, we shall prove a similar result as
Theorem 2.21.  Recall that $Z_G \simeq \BZ_3$ (resp. $Z_G = 1$) if $p \ne 3$
(resp. if $p = 3$). We define $F$ and $F_0$ as
in 2.20 by applying $n' = 3$ as follows. Assume that $p \not\equiv 1 \pmod 3$,
and that $F$ is a non-split
Frobenius with respect to the $\BF_q$-structure.
We put $q_0 = qp$ if $p \equiv -1 \pmod 3$, and $q_0 = q$ if $p = 3$.
(Note that if $p \equiv 1 \pmod 3$, the following discussion does not hold.)
Let $F_0$ be the split Frobenius map with respect to the $\BF_{q_0}$-structure.
Then $q \equiv \pm 1\pmod 3$ if and only if $q_0 \equiv \mp 1\pmod 3$.
\par
Here $L$ is of type $2A_2$ or $D_4$, and $\SW_L \simeq W(G_2)$ or $S_3$,
accordingly. It is easy to see that the formulas in (1.17.1) hold for
$F$ and $F_0$. Hence Lemma 1.18 holds for $G = E_6$.
\par
For each class $C$ in $G$, choose a split element $u \in C^F$, and $v \in C^{F_0}$.
From the construction of split elements for $C^F$ and for $C^{F_0}$, it is easy
to see that the action of $F$ on $A_G(u)$ corresponds to the action of $F_0$
on $A_G(v)$. Hence we have $\wt A_G(u) \simeq \wt A_G(v)$.
Thus one can define a bijection $\xi$ from the set of $G^F$-conjugacy classes in
$G\uni^F$ to the set of $G^{F_0}$-conjugacy classes in $G\uni^{F_0}$, by
attaching a split element $v = \xi(u) \in C^{F_0}$ to a split element $u \in C^F$.
\par
Note that $\SM_G^F$ and $\SM_G^{F_0}$ are naturally identified
(the action of $F$ on $Z_G$ coincides with that of $F_0$).
The following result is a generalization of \cite[4, Thm.]{BS} in the case
where $p$ is good.  

\begin{thm}  
Assume that $G = E_6$.  Let $F$ and $F_0$ be as in 9.15.
Then similar properties as in Theorem 2.21 hold also for $G$.  
\end{thm}

\begin{proof}
We prove the statements (i) and (ii) in Theorem 2.21.  
We follow the notation in Theorem 2.19 and Theorem 2.21.  As in (2.19.3),
we have
\begin{equation*}
\tag{9.16.1}  
  Y_{j,F} = \nu_{u,\eta}(-1)^{\d_j}Y_{j, F_0}
\end{equation*}
for any $j \in I^F = I^{F_0}$.
Since Lemma 1.18 holds for $G$, the discussion of the proof of Theorem 2.21 (i)
also holds for $G$.  Thus (i) is proved. 
\par
Next we show  (ii).  We prove that the bijection $\xi$ satisfies the properties (a) and (b).
Since $(a_0 + r)/2 = d_u$, the relation (b) in (ii)
is nothing but the relation (9.16.1). Thus (b) holds. 
For (a), since $\wt A_G(u) = \wt A_G(\xi(u))$ for a split $u$, it is enough to show
that
\begin{equation*}
\tag{9.16.2}
  |Z_G(u)^F| = (-1)^{\dim Z_G(u)}|Z_G(\xi(u))^{F_0}|(-q)
\end{equation*}
for a split $u \in C^F$.
\par
In the case where $C = D_4(a_1)$, then $Z_G(u_1) \simeq S_3$ for $u_1 \in C$,
and a split element $u \in C^F$ is determined by the condition that
$F$ acts trivially on $A_G(u)$.  Similarly, a split element $v \in C^{F_0}$
is determined by the condition that $F_0$ acts trivially on $A_G(v)$.
Then (9.16.2) is verified directly.  If $C \ne D_4(a_1)$, then $A_G(u_1)$ is
abelian.
\par
If $|Z_G(u_1)^F|$ have a common value for any choice of $u_1 \in C^F$, 
then $|Z_G(v_1)^{F_0}|$ have a common value for any choice of $v_1 \in C^{F_0}$.
Thus in this case, (9.16.2) is also checked directly. 
\par
We now consider the case where  $C \ne D_4(a_1)$, and $|Z_G(u_1)^F|$
(resp. $|Z_G(v_1)^{F_0}|$) have different values for $u_1 \in C^F$
(resp. for $v_1 \in C^{F_0}$).
By using the tables in \cite{M1}, one can check that 
such a class $C$ is uniquely determined, namely $C = A_2$.
Since $A_G(u) \simeq \BZ_2$, there exists two elements 
$v_1, v_2 \in C^{F_0}$ such that
\begin{equation*}
  |Z_G(v_1)^{F_0}| = 2q_0^{26}(q_0^2-1)^2(q_0^3-1)^2, \quad
  |Z_G(v_2)^{F_0}| = 2q_0^{26}(q_0^4-1)(q_0^6-1).
\end{equation*}
Similarly, there exists two elements $u_1, u_2 \in G^F$ such that
$|Z_G(u_1)^F|$ and $|Z_G(u_2)^F|$ are obtained from them (or by switching from them)
by substituting $-q$ to $q_0$.  
Since $|\SM_G(C)| = 1$, by using the computation of Green functions,
one can show that $v_1 \in C^{F_0}$ is a split element, 
and the split element $u_1 \in C^F$ is determined by the condition that
$|Z_G(u_1)^F| = 2q^{26}(q^2-1)(q^3+1)^2$. 
Hence (9.16.2) holds for $C$. 

\par
We have verified that our bijection $\xi$ satisfies the condition (a), (b) in Theorem 2.21.
The uniqueness of $\xi$ is proved in a similar way as in Theorem 2.21.
\end{proof}  

\par\bigskip
\section{ The case ${}^3D_4$ } 

\para{10.1.}  
Assume that $G$ is a simply connected, almost simple group of
type $D_4$, with non-split Frobenius map $F$ such that $F^3$ is split.
In this case, if $(L, C_0, \SE_0) \in \SM_G$ with $L \ne T$,
then $L$ is of type $2A_1$ if $p \ne 2$, and is of type $D_4$ (namely,
$(C_0, \SE_0)$ is cuspidal in $G$) if $p = 2$.  In the former case, $L$ is not
$F$-stable, hence $\SM_G^F = \{ (T, \{1\}, 1) \}$. Thus we have only to
consider the Green functions on $G_{\ad}$.
\par
Green functions of $G^F$ (for any $p$) was first computed by Spaltenstein \cite{Sp3}
with respect to 
Deligne-Lusztig's Green functions, in the course of the computation of unipotent characters
of $G^F$. In the case where $p = 2$, it is also easy to compute them as in the case of
exceptional groups. From those results, it can be checked, for an $F$-stable class $C$ in $G$,
that there exists $u \in C^F$ such that $\g(1,1, u,\r) = 1$ for any $\r \in A_G(u)\wg$
such that $(u,\r)$ belongs to $(T,\{1\},1)$. It follows that split elements exist for $G$. 

\par

\par\bigskip\noindent
Frank L\"ubeck \\
LfAZ - Lehrstuhl f\"ur Algebra und Zahlentheorie  \\
RWTH Aachen University  \\
Pontdriesch 14/16, 52062 Aachen, Germany  \\
frank.luebeck@math.rwth-aachen.de 

\par\bigskip\noindent
Toshiaki Shoji  \\
School of Mathematical Sciences,  Tongji University  \\
1239 Siping Road, Shanghai 200092, P.R of China \\
shoji@tongji.edu.cn

\end{document}